\documentclass[onefignum,onetabnum,reqno]{siamonline190516} 

\usepackage{amsfonts, amsmath, amssymb}
\usepackage{graphbox,graphicx}
\usepackage{epstopdf}
\usepackage[algo2e,linesnumbered,boxed]{algorithm2e}
\usepackage{dsfont}
\usepackage{booktabs}
\usepackage{todonotes}
\usepackage{mathtools}
\usepackage{tikz}
\usepackage{wrapfig}

\ifpdf
  \DeclareGraphicsExtensions{.eps,.pdf,.png,.jpg}
\else
  \DeclareGraphicsExtensions{.eps}
\fi

\usepackage{enumitem}
\setlist[enumerate]{leftmargin=.5in}
\setlist[itemize]{leftmargin=.5in}

\newsiamremark{remark}{Remark}
\newsiamthm{assumption}{Assumption}
\crefname{assumption}{Assumption}{Assumptions}
\newsiamremark{hypothesis}{Hypothesis}
\crefname{hypothesis}{Hypothesis}{Hypotheses}
\newsiamthm{claim}{Claim}
\newsiamthm{scenario}{Scenario}
\crefname{scenario}{Scenario}{Scenarios}
\crefname{section}{Section}{Sections}
\crefname{subsection}{Section}{Sections}

\headers{Risk-averse mean field games: exploitability and non-asymptotic analysis}{Cheng \& Jaimungal}

\title{Risk-averse mean field games: exploitability and non-asymptotic analysis \thanks{
SJ  would like to acknowledge support from the Natural Sciences and Engineering Research Council of Canada (grants RGPIN-2018-05705 and RGPAS-2018-522715). ZC would like to acknowledge support from the Guangzhou-HKUST(GZ) Joint Funding Program (No. 2024A03J0630). ZC's work on this project was mostly done during his postdoctoral appointment at UofT.}}

\author{
Ziteng Cheng
\thanks{Financial Technology Thrust, The Hong Kong University of Science and Technology (Guangzhou)
  (\email{zitengcheng@hkust-gz.edu.cn}).}
 \and
Sebastian Jaimungal
\thanks{Department of Statistical Sciences, University of Toronto (\email{sebastian.jaimungal@utoronto.ca}, \url{http://sebastian.statistics.utoronto.ca})}
}

\usepackage{amsopn}

\tikzstyle{reward}=[shape=circle,draw=blue!50,fill=blue!10]
\tikzstyle{action}=[shape=circle,draw=green,fill=green!10]
\tikzstyle{state}=[shape=circle,draw=red!50,fill=red!10]
\tikzstyle{gru}=[shape=rectangle,draw=black!50,fill=lime!10]
\tikzstyle{obs}=[shape=circle,draw=blue!50,fill=blue!10]
\tikzstyle{lightedge}=[<-,dotted]
\tikzstyle{mainstate}=[state,thick]
\tikzstyle{mainedge}=[<-,thick]

\newcounter{example}[section]

\usepackage{bm}
\usepackage[normalem]{ulem} 
\usepackage{bbm} 
\usepackage{kpfonts}
\usepackage{cancel}

\newcommand{\wt}{\widetilde}

\newcommand{\1}{\mathbbm{1}}            
\newcommand{\set}[1]{\{#1\}}            
\DeclareMathOperator{\dif}{d \!}        

\DeclareMathOperator*{\argmin}{arg\,min} 

\def\cB{\mathcal{B}}
\def\cC{\mathcal{C}}

\def\cE{\mathcal{E}}

\def\cN{\mathcal{N}}

\def\cP{\mathcal{P}}

\def\cR{\mathcal{R}}
\def\cS{\mathcal{S}}
\def\cT{\mathcal{T}}

\def\bA{\mathbb{A}}

\def\bE{\mathbb{E}}

\def\bG{\mathbb{G}}

\def\bM{\mathbb{M}}
\def\bN{\mathbb{N}}

\def\bP{\mathbb{P}}
\def\bQ{\mathbb{Q}}
\def\bR{\mathbb{R}}
\def\bS{\mathbb{S}}
\def\bT{\mathbb{T}}

\def\bX{\mathbb{X}}
\def\bY{\mathbb{Y}}
\def\bZ{\mathbb{Z}}

\usepackage{mathrsfs}
\def\sA{\mathscr{A}}

\def\sD{\mathscr{D}}

\def\sF{\mathscr{F}}

\def\sX{\mathscr{X}}
\def\sY{\mathscr{Y}}

\def\fA{\mathfrak{A}}

\def\fE{\mathfrak{E}}

\def\fK{\mathfrak{K}}

\def\fN{\mathfrak{N}}

\def\fP{\mathfrak{P}}

\def\fR{\mathfrak{R}}

\def\fc{\mathfrak{c}}

\def\fe{\mathfrak{e}}

\def\fk{\mathfrak{k}}

\def\fm{\mathfrak{m}}
\def\fn{\mathfrak{n}}

\def\fp{\mathfrak{p}}

\def\fr{\mathfrak{r}}

\usepackage{makecell}

\graphicspath{{./Figures/}}

\begin{document}

\maketitle
\begin{abstract}
In this paper, we use mean field games (MFGs) to investigate approximations of $N$-player games ($N$pGs) with uniformly symmetrically continuous heterogeneous closed-loop actions. To incorporate agents' risk aversion (beyond the classical expected utility of total costs), we use an abstract evaluation functional for their performance criteria. Centered around the notion of exploitability, we conduct non-asymptotic analysis on the approximation capability of MFGs from the perspective of state-action distributions without requiring the uniqueness of equilibria. Under suitable assumptions, we first show that scenarios in the $N$pGs with large $N$ and small average exploitabilities can be well approximated by approximate solutions of MFGs with relatively small exploitabilities. We then show that $\delta$-mean field equilibria can be used to construct $\varepsilon$-equilibria in $N$pGs. Furthermore, in this general setting, we prove the existence of mean field equilibria. This proof reveals a possible avenue for incorporating penalization for randomized action into MFGs.
\end{abstract}



\tableofcontents

\section{Introduction}

Since the pioneering work of \cite{Huang2006Large,huang2007large} and \cite{Lasry2007Mean}, the literature on mean field games (MFGs) has gone through, and continues to experience, a massive expansion. MFGs can be viewed as symmetric games with infinitely many players and are primarily used to approximate $N$-player games ($N$pGs) with large population where the impact of each individual is ``small''. The majority of the literature focuses on analysis in continuous time. We refer to \cite{Bensoussan2016Linear,Ma2020Linear,Moon2017Risk,Firoozi2021Epsilon,Firoozi2022Exploratory} and the reference therein for studies of MFGs under different settings in continuous time, such as linear-quadratic games, major-minor agent games, robust games, partially observed games, and exploratory games. We also refer to \cite{Caines2018Mean} and \cite{Carmona2018book} for overviews.

While MFGs are used to approximate large games, it is crucial to understand their approximation capabilities. Following \cite{Huang2006Large}, there has been enormous work developing the idea that uses mean field equilibriia (MFEs) to construct $\varepsilon$-Nash equilibria in $N$pGs. Conversely, the results of \cite{Lasry2007Mean} suggests that certain equilibria in $N$pGs, for large $N$, can be captured by MFEs. Research in this direction is indispensable because it helps to answer whether using MFGs to approximate $N$pGs misses any equilibria. Below we highlight some work that progress the understanding on this matter. Assuming the uniqueness of MFEs, \cite{Cardaliaguet2015book} uses the master equation to study closed-loop\footnote{Roughly speaking, closed-loop means each player has access to information of other players, while open-loop assumes no access to such information. The precise definitions vary across the literature. We refer to \cite{Fudenberg1988Open} and \cite[Definition 2.4 and 2.7]{Carmona2018book} for more discussions.} $N$pGs, and shows that the equilibria converge to the MFE as $N\to\infty$. \cite{Lacker2016General} and \cite{Fischer2017Connection} study the convergence of open-loop equlibria under various settings without assuming uniqueness. \cite{Lacker2020Converrgence} and \cite{Lacker2022Closed} further extend the convergence results from \cite{Lacker2016General} to closed-loop equilibria with or without common noise. We note that the convergence results in \cite{Lacker2020Converrgence} and \cite{Lacker2022Closed} are established for weak MFEs. Weak MFEs consists of a random flow of measures. Convergence to weak MFEs is possible even with considerably weak assumptions on the actions of players. On the contrary, the more common-seen MFEs in \cite{Huang2006Large,Lasry2007Mean,Fischer2017Connection}, among many others, consists of a deterministic flow of measures, and is called strong MFEs, or simply, MFEs. It often requires suitable assumptions on players' actions to facilitate the convergences to MFEs. Apart from the above,  \cite{Iseri2022Set} pioneers the study of the value function approximation of MFGs within the discrete time framework without requiring the uniqueness of MFEs. Unlike previous studies that have focused on the capturing capabilities of MFGs, the analysis conducted in \cite{Iseri2022Set} is predominantly non-asymptotic. Our work draws inspiration from \cite{Lacker2020Converrgence} and \cite{Iseri2022Set}, among others, and we provide a detailed comparison in \cref{subsec:HeurComment}. For completeness, we also refer to \cite{Bayraktar2020Non,Cecchin2022Selection,Dianetti2023Unifying,Sanjari2024Optimality} and the references therein for other studies that address the non-uniqueness of equilibria, although capturing capability is not their primary focus. Finally, we would like to point to \cite{Djete2021Large,Djete2021Mean,Lauriere2022Convergence} for recent work on approximation analysis on extended MFGs with unique equilibrium. Extended MFGs consider interactions through players' actions.

This paper's goal is to further explore the approximation capability of MFGs from the perspective of state-action distributions -- with a focus on (strong) MFEs and non-asymp-\\totics description, while allowing for multiple equilibria, certain heterogeneous closed-loop actions, as well as $\varepsilon$-relaxation in individual decision making. Readers familiar with MFGs may realize that some features in our exploration has appeared in the extant literature, however, not all together in a single setting. To avoid excessive technicalities, we choose to work in discrete time and finite horizon with state and action spaces being Polish. 

The idea of MFGs in discrete time dates back to \cite{Jovanovic1988Anonymous}, although the connection to $N$pGs was left unstudied for quite a while. Notably,  \cite{Saldi2018Markov} considers an infinite horizon $N$pG with discounted unbounded costs on Polish state and action spaces, proves the existence of MFEs, and constructs $\varepsilon$-equilibra in $N$pGs using MFEs. Analogous results are established for risk sensitive games \cite{Saldi2020Approximate}, partially observed games \cite{Saldi2022Partially}, and risk averse games \cite{Bonnans2021Discrete}. There has been recent progress on machine learning algorithms for MFEs in discrete time, see, e.g., \cite{Guo2022General} and \cite{Hu2022Recent}. We would also like to point out \cite{Acciaio2021Cournot} for a causal transport-based method potentially useful for handling weak MEFs. 

Another aspect we aim to address in our study is the incorporation of risk aversion into MFG theory. In practical scenarios, it is often observed—and sometimes preferred—that decision-makers prioritize avoiding unfavorable outcomes over pursuing favorable ones. A performance criterion that aligns with this concept is introduced by \cite{Ruszczynski2010Risk}, with further extensions allowing for randomized actions discussed in \cite{Chu2014Markov}. This criterion offers the advantage of inherent time consistency, and in the meantime, naturally generalizes the standard risk-neutral framework, which uses expected total cost as the performance metric. This setup facilitates subsequent analyses from the standpoint of dynamic programming. For related applications in areas such as automation and finance, we refer to \cite{Majumdar2019How, Wang2022Risk, Coache2023conditionally} and the references therein. Although risk-aversion theory has been extensively explored in the context of single-agent decision-making, risk aversion in a MFG framework remains relatively uncharted. Whether one chooses to follow the principles proposed by \cite{Ruszczynski2010Risk} or not, the discourse on this topic in MFG literature is notably sparse. For an study that aligns with \cite{Ruszczynski2010Risk}, we direct readers to \cite{Bonnans2021Discrete}. Regarding other forms of risk aversion within MFGs, we refer to \cite{Saldi20Approximate} and the reference therein.

In our setup, we assume all players are subject to the same Markovian controlled dynamics and consider weak interaction, that is, the impact of the other players on an individual is through the empirical distribution of the population states. Regarding the performance criteria, we abstract the procedure of backward induction, and define an evaluation functional as a composition of operators that maps value functions backward. This abstraction not only encompasses the Bellman operators commonly studied in risk-neutral mean field Markov Decision Process (MDP) frameworks (e.g., \cite[(3.1)]{Saldi2018Markov} and \cite[(16)]{Lauriere23Model}), but also facilitates the examination of risk-averse variants tied to the aforementioned risk-averse performance criteria. Additionally, it contributes to a more concise presentation and simplifies notation. To approximate $N$pGs with MFGs, we impose further assumptions. In particular, we consider $N$pGs where all players use $\vartheta$-symmetrically continuous policy (see \cref{assump:SymCont}). An example illustrating the necessity\footnote{It is necessary in the sense that, getting rid of this assumption requires adding other assumptions on other components of the games.} of this assumption is provided in an example in \cref{subsec:ExmpGame}. The example is constructed in a degenerate environment with absorption, which is in spirit similar to that in \cite[Section 7]{Campi2018N}.

Our main results are centered around the notion called exploitability. In $N$pGs, exploitability quantifies the best possible improvement a player can achieve by revising her original policy, while assuming the policies of other players remain unchanged. It effectively describes the scale of disequilibrium in a game scenario. Note that $\varepsilon$-equilibria in an $N$pG can be equivalently defined via exploitability. The introduction of exploitability paves ways for investigating the global properties of games, in that it provides insights beyond the existence and uniqueness of equilibria. In MFGs, exploitability can be defined analogously. More precisely, exploitability of a policy is defined as the difference between the outcome under the considered policy and the outcome under the optimal policy, where both outcomes are computed under the mean field flow induced by the considered policy. $\varepsilon$-MFE can be defined using exploitability.

The notion of exploitability is frequently used for analyzing the convergence of algorithms that approximate MFEs, e.g.,  \cite{Perrin2020Fictitious,Cui2021Approximately,Bonnans2021Generalized,Dumitrescu2022Linear,Guo2022MF,Hu24MF}. Notably, \cite{Guo2022MF} develops an algorithm that approximate MFEs by minimizing exploitability. This algorithm is capable of handling the presence of multiple MFEs.

Despite its widespread use in equilibrium-finding algorithms, exploitability has rarely been the focal point in the approximation analysis of MFGs, as evidenced by the literature review above. This oversight may stem from the current focus in approximation analysis, which often overlooks non-asymptotic considerations surrounding approximate $N$-player equilibria ($N$pEs). To the best of our knowledge, the only relevant study that concerns capturing approximate $N$pEs non-asymptotically thus far is \cite{Iseri2022Set}. They primarily concentrates on value functions and does not involve exploitability. 

To advocate for the use of exploitability, we contend that in more realistic considerations of $N$pGs, equilibria are often attained in an approximate sense. Without sufficiently strong assumptions, there is no assurance that an approximate equilibrium closely aligns with any exact equilibrium (in terms of, for example, the empirical state distribution), even when there is a unique equilibrium. To substantiate that MFGs approximate $N$pGs effectively, one approach is to relate each approximate $N$pE to an approximate MFE, if feasible, or to identify the reasons for failure if not. The use of exploitability is naturally justified, as it is anticipated that related approximate equilibria, if they exist, should exhibit comparable scales of exploitability.

The main contributions of this paper can be summarized as follows:
\begin{enumerate}
    \item Various types of exploitabilities are proposed and their interrelationships are scrutinized. In particular, to incorporate the notion of exploitability in $N$pGs and MFGs into our setting that involves abstract evaluation operations, we introduce the \textit{total (stepwise) exploitability} (see \eqref{eq:DefStepExploitability} and \eqref{eq:DefMeanStepExploitability}). The relationship between total exploitability and the standard notion of exploitability, here called \textit{end exploitability} (see \eqref{eq:DefEndExploitability} and \eqref{eq:DefMeanEndExploitability}), is examined in \cref{prop:EstEndExploitabilityCont} and \cref{prop:EstMFEndExploitability}. We also consider end exploitability based on the best symmetrically continuous policy in $N$pG (see \eqref{subsec:NplayerExploitability}), as this aligns naturally with our setup. In \cref{prop:EstEndExploitabilityCont}, we demonstrate that the difference between the constrained end exploitability and the unconstrained one is negligible under suitable conditions.       
    
    \item We construct a state-action mean field flow from an $N$-player scenario, and show that the expected bounded-Lipschitz distance between the empirical flow in the $N$-player scenario and the mean field flow is small when $N$ is large. As long as certain continuity assumptions hold true, we prove that the total exploitability in a mean field scenario is dominated by the average of the total exploitabilities of all $N$ players plus a small error term. Under stronger conditions, the difference between two quantities can be bounded by the same error term. Detailed statements are presented in \cref{thm:MFApprox}. Conversely, in \cref{thm:MFConstr}, we use $\varepsilon$-MFEs to construct an open-loop $N$pG where all players have exploitabilities bounded by $\varepsilon$ plus a small error term. These results collectively lead to a non-asymptotic description on the approximation capability of MFGs.
    
    \item We prove in an abstract formulation the existence of MFEs. This formulation encompasses a range of performance criteria beyond the classical expected total cost, such as risk aversions via nested compositions (see \cref{subsec:ExmpG} for example), certain extended forms of cost (e.g., \eqref{eq:ExmpExtendedCost}).
    The proof uses the Kakutani fixed point theorem, and is motivated by earlier work on the existence of MFEs in discrete time (cf. \cite{Jovanovic1988Anonymous} and \cite{Saldi2018Markov}). 
    Due to our abstract formulation, however, we must develop significant modifications to the original analysis. Our approach provides substantial extensions of the previous methods and our techniques could be of interest on their own. As an unexpected byproduct, our proof uncovers a potential avenue for incorporating penalization for randomized actions within the framework of MFGs. We refer to \cref{thm:ZeroExploitability} and thereafter for related discussions.
\end{enumerate}

The organization of the remainder of this paper is as follows. To facilitate understanding, we begin with heuristic discussions of our approximation results in \cref{sec:Heur}. In \cref{sec:Setup}, we describe our model setup and introduce the necessary notations. Specifically, the technical assumptions are detailed in \cref{subsec:Assumptions}. We then define and examine various notions of exploitability for $N$pGs and MFGs in \cref{sec:Exploitability}. Our findings on the approximation capabilities of MFGs are presented in \cref{sec:MFApprox}. The subsequent section, \cref{sec:MFE}, addresses the existence of MFE. The proofs for the results in \cref{sec:Exploitability}, \cref{sec:MFApprox}, and \cref{sec:MFE} are provided in \cref{sec:Proofs}. Additionally, \cref{sec:Exmp} contains several examples, and \cref{sec:Lemmas} presents technical results. The proofs of preliminary results mentioned in \cref{sec:Setup} are included in \cref{sec:Suppl}. Finally, a glossary of notations is provided for reference in \cref{sec:GoN}.

\section{Heuristics on the approximation capability of mean field games}\label{sec:Heur}
In \cref{subsec:HeurDerv}, we heuristically derive our key findings in a discrete setting, focusing on how $N$pEs can be captured by MFEs, in the absence of uniqueness. We place more emphasis on this aspect rather than on constructing $N$pEs from MFEs, as the latter is generally better understood, while understanding of the former is still evolving. In \cref{subsec:HeurComment}, we comment on the approximation capabilities of MFGs. Additionally, we discuss how our approach differs from existing ones. Finally in \cref{subsec:Tech}, we highlight some of the technical challenges we face when rigorously establishing these heuristics in a more general setting. 

The discussion in this section primarily pertains to the results presented in \cref{sec:Exploitability} and \cref{sec:MFApprox}, which regards various notions of exploitability and the approximation capabilities of MFGs. For insights into the existence of MFE, we refer the reader to \cref{sec:MFE}.

\subsection{Derivation}\label{subsec:HeurDerv}
For simplicity, we consider finite state space $\bX$, finite action space $\bA$, and finite horizon $T\in\bN$.  Let $N$ represent the number of players. As we consider non-asymptotics, we suppress $N$ from the notations introduced below unless necessary. Suppose all players start from the same initial position $x_1\in\bX$ at $t=1$. We consider a Markov environment with transition kernels $P_t:\bX^N\times\bA\to\cP(\bX)$ for $t=1,\dots,T-1$. Here, the state transition of an individual player depends only on the positions of all players and the action of the player herself. The case where the transition is also affected by other players' actions will eventually lead to extended mean field games, which is out of the scope of this paper. Similarly, we consider cost functions $C_t:\bX^N\times\bA\to\bR$ for $t=1,\dots,T$. In what follows, we let $\fp^n_t:\bX^N\to\cP(\bA)$ represent the randomized actions of player-$n$ at time $t$, and let $\fp^n=(\fp^n_1,\dots,\fp^n_T)$ denote the player-$n$'s policy. By convention, we  use subscripts for the inputs of measure-valued functions. Moreover, we set $\bm\fP=(\fp^1,\dots,\fp^N)$ to be the collection of all player's policies, which we call a \textit{scenario} in the $N$pG.  

Let us consider a probability space $(\Omega,\sF,\bP^{\bm\fP})$. For notational convenience, we  suppress $\bm\fP$ from $\bP^{\bm\fP}$ and $\bE^{\bm\fP}$. Under $\bP$, we have the following dynamics 
\begin{align*}
\bm X_1=(x_1,\dots,x_1),\quad \bm A_t\sim\bigotimes_{n=1}^N\fp^n_{t,\bm X_t}\,,\quad \bm X_{t+1}\sim\bigotimes_{n=1}^N P_{t,\bm X_t, A^n_t}\,,
\end{align*} 
where $\bm X_t = (X^1_t,\dots,X^N_t)$ and $\bm A_t = \left(A^1_t,\dots,A^N_t\right)$ represent the state and actions of the players at time $t$, respectively. Furthermore, player-$n$ aims to find
\begin{align*}
\min_{\fp^n} \bE\left[ \sum_{t=1}^T C_t(\bm X_t, A^n_t) \right].
\end{align*}

To ensure that MFGs effectively approximate the aforementioned $N$pG, additional structure is needed. We impose the required structure shortly. For the moment, however, let us introduce a key concept in our analysis: the \textit{exploitability} of player-$n$ in an $N$pG. To do so, consider a hypothetical case where all other players adhere to their original policies, while player-$n$ has complete observation of $\bm X_t$ at each decision-making point in time $t$ and is allowed to consider a different policy than her original one. From player-$n$'s viewpoint, this problem may be posed as a MDP. We next define player-$n$'s exploitability as 
\begin{align}\label{eq:HeurDefEndExploitability}
R^n(\bm\fP) := \bE\left[ \sum_{t=1}^T C_t(\bm X_t, A^n_t) \right] - \min_{\fp^n} \bE\left[ \sum_{t=1}^T C_t(\bm X_t, A^n_t) \right].
\end{align}
It is straightforward to see that $\bm\fP$ is a $\varepsilon$-Nash equilibrium if and only if $R^n(\bm\fP)\le\varepsilon$ for all $n$. Later, we also use $\frac1N\sum_{n=1}^N R^n(\bm\fP)$ to characterize approximate equilibrium in the average sense. We introduce below a few terms for the Dynamic Programing Principle (DPP) related to the aforementioned MDP. For $\bm x = (x^1,\dots,x^N)$, we define
\begin{align*}
V^{n*}_t(\bm x) := \min_{\fp^n} \bE\left[ \sum_{r=t}^T C(\bm X_r,A^n_r) \,\Bigg|\,\bm X_r = \bm x \right], \quad t=1,\dots, T,
\end{align*}
i.e., $V^{n*}_t$ is the optimal value function of player-$n$ at time $t$. For convenience, we stipulate that $V^{n*}_t\equiv 0$ for $t>T$. The following DPP in terms of the Bellman equation is well-known
\begin{align}\label{eq:NpDPP} 
V^{n*}_t(\bm x) = \min_{\lambda\in\cP(\bA)} \bE\big[ C_t(\bm x, A^n) + V^{n*}_{t+1}(\bm X) \big],\quad t=1,\dots,T-1,
\end{align}
where we have employed dummy random variables $\bm A=(A^1,\dots,A^N)$ and $\bm X=(X^1,\dots,X^N)$ with dynamics:

\begin{align*}
\bm A\sim\fp^1_{t,\bm x}\otimes\dots\otimes \fp^{n-1}_{t,\bm x}\otimes\;\lambda\;\otimes\fp^{n+1}_{t,\bm x}\otimes\dots\otimes
\fp^{n+1}_{t,\bm x}\;,
\qquad
\bm X\sim\bigotimes_{i=1}^N P_{t,\bm x,A^i}\;.
\end{align*}

Next, in this risk neutral setting, we develop an equivalent form of $R^n$. We first expand $R^n$ by subtracting and adding auxiliary terms,
\begin{align}\label{eq:ExploitabilityTelescopingSum}
R^n(\bm\fP) &= \bE\left[ \sum_{r=1}^{T-1} C_r(\bm X_r, A^n_r) + C_T(\bm X_T, A^n_T) + V^{n*}_{T+1}(\bm X_{T+1}) \right] - \bE\left[ \sum_{r=1}^{T-1}C_r(\bm X_r, A^n_r) + V^{n*}_T(\bm X_T) \right] \nonumber\\
&\quad +  \bE\left[ \sum_{r=1}^{T-1}C(\bm X_r, A^n_r) + V^{n*}_T(\bm X_T) \right] - V^{n*}_1(x_1,\dots,x_1) \nonumber\\
&= \bE\left[  C_T(\bm X_T, A^n_T) + V^{n*}_{T+1}(\bm X_{T+1}) - V^{n*}_T(\bm X_T) \right]\nonumber\\
&\quad +\bE\left[ \sum_{r=1}^{T-2} C_r(\bm X_r, A^n_r) + C_{T-1}(\bm X_{T-1}, A^n_{T-1}) + V^{n*}_{T}(\bm X_{T}) \right] \!-\! \bE\left[ \sum_{r=1}^{T-2}C_r(\bm X_r, A^n_r) + V^{n*}_{T-1}(\bm X_{T-1}) \right] \nonumber\\
&\quad +  \bE\left[ \sum_{r=1}^{T-2}C_r(\bm X_r, A^n_r) + V^{n*}_{T-1}(\bm X_{T-1}) \right] - V^{n*}_1(x_1,\dots,x_1) \nonumber\\
&=\dots\dots 
\end{align} 
Eventually, we obtain the following expression,
\begin{align}\label{eq:HeurDefStepwiseExploitability}
R^n(\bm\fP) &= \sum_{t=1}^T \bE\left[ C_t(\bm X_t,A^n_t) + V^{n*}_{t+1}(\bm X_{t+1}) - V^{n*}_t(\bm X_t) \right]
\nonumber\\
&= \sum_{t=1}^T \bE\Big[ \bE\big[ C_t(\bm X_t,A^n_t) + V^{n*}_{t+1}(\bm X_{t+1}) \big| \bm X_t \big] - V^{n*}_t(\bm X_t) \Big] .
\end{align}
The reader might notice that \eqref{eq:HeurDefStepwiseExploitability} may be derived more concisely using telescoping sums under the expectation, rather than using  \eqref{eq:ExploitabilityTelescopingSum}. The approach we here take, however,  generalizes to the risk-averse case while the simpler approach does not. We refer to \eqref{eq:EndExploitabilityTelescope} for the related application in proofs. Moreover, we stress that the risk-averse analogue to the right hand side of \eqref{eq:HeurDefStepwiseExploitability}, see \cref{sec:Exploitability}, provides a crucial foundation for deriving our main results which concerns risk-aversion beyond expectation. Therefore, we maintain working through \eqref{eq:HeurDefStepwiseExploitability} although simpler approach is possible in the risk-neural setting.

To distinguish the two representations, we call the right hand side of \eqref{eq:HeurDefStepwiseExploitability} 
\textit{total (stepwise) exploitability}, while the original definition of $R^n$ in \eqref{eq:HeurDefEndExploitability} we call \textit{end exploitability}. For clarification, we note that end exploitability compares the (aggregated) outcome of the original policy against the optimal (aggregated) outcome. Contrastingly, total exploitability compares the expected one-step cost  plus the one-step ahead optimal value, conditional on the state at time $t$, against the optimal value at time $t$ (see also \eqref{eq:NpDPP}), and then averages over state at time $t$. The total stepwise exploitability is obtained by then summing over  these individual stepwise exploitabilities.

Next, we, heuristically, provide a few technical assumptions and lay the foundation for the upcoming discussion on how, in the absence of uniqueness, $N$pEs can be captured by MFEs. Let $\delta_x$ be the Dirac delta measure concentrated at $x$. For $\bm x=(x_1,\dots,x_n)$, we define $\overline\delta_{\bm x}:=\frac1N\sum_{n=1}^N\delta_{x^n}$ -- the empirical measure. The conditions below are used to ensure the impact of an individual player is small, given a sufficiently large $N$. 
\noindent 
\begin{assumption}\label{assump:Heur}
The following is true for all players at all $t$,
\begin{itemize}
\item[(i)] $P_{t,\bm x, a} = P_{t, x^n, \overline\delta_{\bm x}, a}$\,, 
\item[(ii)] $C_t(\bm x, a) = C_t(x^n, \overline\delta_{\bm x}, a)$\,,
\item[(iii)] $\fp^n_{t,\bm x} = \fp^n_{t,x^n,\overline\delta_{\bm x}}$,
\item[(iv)] $\xi\mapsto P_{x,\xi,a}$, $\xi\mapsto C_t(x,\xi,a)$, and $\xi\mapsto\fp^n_{t,x,\xi}$ are continuous in $\xi\in\cP(\bX)$.\footnote{Since in this section, $\bX$ and $\bA$ are finite, the continuity on $\bX$ and $\bA$ are given by default.}
\end{itemize}
\end{assumption}
Under this assumption, the dynamics of the $N$pG becomes 
\begin{align*}
\bm X_1=(x_1,\dots,x_1),\quad \bm A_t\sim\bigotimes_{n=1}^N\fp^n_{t,X^n_t,\overline\delta_{\bm X_t}}\,,\quad \bm X_{t+1}\sim\bigotimes_{n=1}^N P_{t,X^n_t,\overline\delta_{\bm X_t},A^n_t}\,.
\end{align*} 
We note that, eventually, the accuracy of the mean field approximation hinges on the modulus of continuity in \cref{assump:Heur} (iv). If the continuity is excessively rough, the approximation bound derived may become vacuous. For an illustrative example, we refer to \cref{subsec:ExmpGame}. 

Next, we introduce an analogous scenario but now in the MFG setting. We propose to characterize MFG from the perspective of state-action distribution, rather than through the policy of the representative player.  In this approach, the MFG is described by a time-indexed vector of state-action distributions from $\cP(\bX\times\bA)$. This perspective is frequently adopted in the discrete-time MFG literature for various reasons. For instance, in \cite{Saldi2018Markov}, it facilitates the application of the Kakutani fixed-point theorem, and in \cite{Guo2022MF}, it aids in algorithm development.

Based on $\bm\fP$, we introduce an auxiliary probability $\overline\bP$ and auxiliary $\overline\xi_1,\dots,\overline\xi_T\in\cP(\bX)$. As previously, we omit the explicit dependence of these terms on $\bm\fP$ in our notation for convenience. We construct $\overline\bP$ and $\overline\xi_1,\dots,\overline\xi_T\in\cP(\bX)$ such that, under $\overline\bP$,
\begin{gather}\label{eq:MFPDynamics}
\bm X_1=(x_1,\dots,x_1)\,,\quad\bm A_t\sim\bigotimes_{n=1}^N\fp^n_{t, X^n_t, \overline\xi_t}\,,\quad \bm X_{t+1}\sim\bigotimes_{n=1}^N P_{t, X^n_t, \overline\xi_t, A^n_t}\,.
\end{gather}
and
\begin{gather*}
\overline\xi_1:=\delta_{x_1}\,,\quad \overline\xi_{t+1}(\cdot) := \frac1N \sum_{n=1}^N \overline\bP\left[X^n_{t+1}\in\cdot\right]\,.
\end{gather*}
Note that $\overline\xi_t$'s are (deterministic) elements from $\cP(\bX)$. It follows that the players are independent under $\overline\bP$. In short, $\overline\xi_t$'s represent the average across players of their individual state distributions, where players are mutually independent and use $\bm\fP:=(\fp^1,\dots,\fp^N)$. For future reference, under $\overline\bP$, we have the following dynamics, independently across $n$, 
\begin{align}\label{eq:DynPbar}
X^n_1=x_1,\quad A^n_t\sim \fp^n_{t, X^n_t, \overline\xi_t},\quad X^{n}_{t+1}\sim P_{t, X^n_t, \overline\xi_t, A^n_t}\,.
\end{align}

We next introduce a heuristic proposition, which essentially follow from Assumption \ref{assump:Heur}, Hoeffding's inequality, and some meticulous application of induction. The ``$\approx$'' symbol below is a notation for approximation. For rigorous counterparts of \cref{prop:HeurMFApprox}, including a precise meaning of ``$\approx$", please see \cref{prop:ProcConc} and \cref{prop:EstDiffNPMean}.
We also refer to \cite[Theorem 5.3]{Iseri2022Set} for a similar result, where they consider deterministic Lipschitz path-dependent policies.
\begin{proposition}\label{prop:HeurMFApprox}
Under \cref{assump:Heur}, the following is true:
\begin{itemize}
\item[(a)] $\bP\left[\overline\delta_{\mathbf X_t} \approx \overline\xi_t\right]\approx 1$ for all $t$. Moreover, this approximation remains valid uniformly if we vary one $\fp^n$ in $\bP$,\footnote{Recall that $\bP$ is in fact $\bP^{\bm\fP}$.} while still defining $\overline\xi_t$ with the original $\bm\fP$.
\item[(b)] $\bP\left[X^n_{t}\in\cdot\right] \approx \overline\bP\left[X^n_t\in\cdot\right]$ for each $n$ and $t$.
\end{itemize}
\end{proposition}
Note in \cref{prop:HeurMFApprox} (a), we consider approximation under a $\bP$ where we may vary one $\fp^n$. This aligns with the aforementioned hypothetical case in defining the end exploitability, where all other players maintain their original policies, while player-$n$ may alter hers. 
The approximation accuracy relies
on the modulous of continuity in \cref{assump:Heur} (iv) as well as other model parameters, such as $N$ and the cardinality of $\bX$ and $\bA$. We refer to \cref{prop:ProcConc} and \cref{prop:EstDiffNPMean} for the related details. 

The construction above only specifies the state marginal distributions of a vector on $\cP(\bX\times\bA)$. To provide a complete description of the MFG scenario from the perspective of state-action distributions, we procedd to construct the full state-action distributions. For notational convenience, we write
\begin{align*}
\bP^{X^n_{t}}:=\bP[X^n_{t}\in\cdot] \quad\text{and}\quad \overline\bP^{X^n_t}:=\overline\bP[X^n_t\in\cdot].
\end{align*}
For $t=1,\dots,T$, we define $\overline\psi_t\in\cP(\bX\times\bA)$ by setting
\begin{align}\label{eq:HearDefpsi}
\overline\psi_t(B\times A) = \frac{1}{N} \sum_{n=1}^N \int_B \fp^n_{t,x,\overline\xi_t}(A)\; \overline\bP^{X^n_{t}}(\dif x).
\end{align}
Clearly, the marginal distribution of $\overline\psi_t$ on $\bX$, denoted by $\xi^{\overline\psi_t}$, 
equals $\overline\xi_t$. Additionally, by \eqref{eq:DynPbar},
\begin{align}\label{eq:MFScene}
\xi^{\overline\psi_{t+1}}(B) &= \frac{1}{N} \sum_{n=1}^N \overline\bP^{X^n_{t+1}}(B) = \frac1N \sum_{n=1}^N \int_\bX\int_\bA P_{x,\overline\xi_t,a}(B) \fp^n_{t,x,\overline\xi_t}(\dif a)\; \overline\bP^{X^n_{t}}(\dif x) \nonumber\\
&\quad = \int_{\bX\times\bA} P_{x,\overline\xi_t,a}(B) \; \overline\psi_t(\dif x \dif a).
\end{align}
We then construct the mean field scenario as
\begin{align}\label{eq:HeurDefMFF}
\overline\Psi:=(\overline\psi_1,\dots,\overline\psi_{T}).
\end{align}
Note that $\overline\Psi$ depends solely on $\bm\fP$. Moreover, $\overline\Psi$ indeed represents a scenario in MFGs: in light of the law of large numbers, when all the infinitely many players adopt the same (randomized) policy, it leads to a deterministic flow of state-action distributions, which in turn statistically summarizes the movement and behaviour of the players. Furthermore, by letting $\overline\pi^{\overline\psi_t}$ be the conditional distribution of action given the state induced by $\overline\psi_t$, we have
\begin{align}\label{eq:InducedMFPolicy}
\overline\pi^{\overline\psi_t}_{x}(\set{a}) = \sum_{n=1}^N \frac{ \overline\bP^{X^n_{t}}(\set{x})}{\sum_{m=1}^N \overline\bP^{X^m_{t}}(\set{x})} \fp^n_{t,x,\overline\xi_t}(\set{a}) ,\quad x\in\bX,\;a\in\bA.
\end{align}
This depicts the policy of the representative player in the MFG scenario. 
It can be shown that \eqref{eq:HeurDefMFF} and \eqref{eq:InducedMFPolicy} provide equivalence characterization of the mean field scenario (see also \cref{lem:MFF}).

In order to proceed, we introduce another auxiliary MDP in the MFG setting. Consider a $\overline\fp = (\overline\fp_1,\dots,\overline\fp_T)$, where $\overline\fp_t:\bX\to\cP(\bA)$ for $t=1,\dots,T$. Additionally, for the representative player in MFG scenario $\overline\Psi$, we introduce state-action processes $\overline X_t$ and $\overline A_t$. We extend the dependence of $\overline{\bP}$ from $\bm{\fP}$ to $(\bm{\fP}, \overline{\fp})$ so that, under $\overline{\bP}$, the following dynamics hold:
\begin{align}\label{eq:ReprDynamics}
\overline{X}_1 = x_1, \quad \overline{A}_t \sim \overline{\fp}_{t, \overline{X}_t}, \quad \overline{X}_{t+1} \sim P_{t, \overline{X}_t, \overline{\xi}_t, \overline{A}_t}.
\end{align}
Following the discussion leading to \eqref{eq:InducedMFPolicy}, the representative player naturally employs the policy induced by $\overline\Psi$. More precisely,
\begin{align}\label{eq:ReprPolicy}
\overline\fp = \left(\overline\pi^{\overline\psi_1},\dots,\overline\pi^{\overline\psi_T}\right).
\end{align}
However, this policy need not to be optimal in the hypothetical case where the representative player may revise her policy while the others maintain their original policies. This leads to the following auxiliary optimization problem
\begin{align*}
\min_{\overline\fp} \overline\bE\left[ \sum_{t=1}^T C_t(\overline X_t, \overline\xi_t, \overline A_t) \right],
\end{align*}
where we emphasize that $\overline\xi_1,\dots,\overline\xi_T$ depend solely on $\bm\fP$ but not $\overline\fp$. We subsequently introduce the corresponding optimal value function 
\begin{align*}
\overline V^*_t(x) := \min_{\overline\fp} \overline\bE\left[ \sum_{r=t}^T C_t(\overline X_r, \overline\xi_r, \overline A_r) \,\Bigg| \,\overline X_r=x  \right], \quad t=1,\dots,T.
\end{align*}
For convenience, we stipulate that $\overline V^*_t\equiv 0$ for $t>T$. We then have the following DPP
\begin{align}\label{eq:MFDPP} 
\overline V^*_t(x) = \min_{\lambda} \overline\bE\left[ 
\,
C_t(x,\overline\xi_t,\overline A) + \overline V^*_{t+1}(\overline X) 
\,\right], \quad t=1,\dots,T,
\end{align}
where we have employed dummy random variables with dynamics:
\begin{align*}
\overline A\sim\lambda,\quad \overline X\sim P_{t,x,\overline\xi_t,\overline A}.
\end{align*}
The related mean field exploitabilities can be defined in a manner similar to earlier discussions. However, their introduction is postponed to slightly later for smoother narrative flow.
Instead, we introduce below a crucial proposition. We refer to \cref{prop:EstfTDiffcSstar} for the corresponding formal statement. 
\begin{proposition}\label{prop:HeurApproxOptV}
Under \cref{assump:Heur}, we have $\bP\big[V^{n*}_t(\bm X_t) \approx \overline V^*_t(X^n_t)\big] \approx 1$.
\end{proposition}
Heuristically, establishing \cref{prop:HeurApproxOptV} requires examining an altered $\bm\fP$, where the original $\fp^n$ is replaced by a concatenation of the original $\fp^n$ up to time $t$ with the optimal policy obtained from \eqref{eq:NpDPP} after time $t$. Under this altered $\bm\fP$, \cref{prop:HeurMFApprox} (a) together with Assumption \ref{assump:Heur} and a careful backward induction allows, with a high probability, an approximate reduction of \eqref{eq:NpDPP} to \eqref{eq:MFDPP} when $\bm x$ in \eqref{eq:NpDPP} is substituted by $\bm X_t$, which eventually leads to \cref{prop:HeurApproxOptV}.

We are ready to heuristically derive one of our main results. By \eqref{eq:HeurDefStepwiseExploitability}, \cref{assump:Heur}, \cref{prop:HeurMFApprox} (a) and \cref{prop:HeurApproxOptV}, we have
\begin{align*}
R^n(\bm\fP) \approx \sum_{t=1}^T \bE\Big[ \overline{\bE}\big[ C_t(X^n_t,\overline\xi_t,A^n_t) + \overline V^{*}_{t+1}(X^n_{t+1}) \big| X^n_t \big] - \overline V^{*}_t(X^n_t) \Big].
\end{align*}
Above, note that the only randomness inside $\bE[\cdot]$ is $X^n_t$. Expanding the right hand side into integral form and invoking \cref{prop:HeurMFApprox} (b), we yield
\begin{align}
R^n(\bm\fP) &\approx \sum_{t=1}^T \int_{\bX} \left( \int_{\bA} \left( C_t(x,\overline\xi_t,a) + \int_{\bX}\overline V^*_{t+1}(y)P_{x,\overline\xi_t,a}(\dif y) \right) \fp^n_{t,x,\overline\xi_t}(\dif a) - \overline V^*_t(x) \right) \; \overline\bP^{X^n_{t}} (\dif x),
\label{eqn:average-stepwise-approx}
\end{align} 
Then, \eqref{eqn:average-stepwise-approx} together with \eqref{eq:HearDefpsi} and \eqref{eq:InducedMFPolicy} implies
\begin{align}\label{eq:HeurPreCapture}
\frac1N \sum_{n=1}^N R^n(\bm\fP) &\approx \sum_{t=1}^T \frac1N \sum_{n=1}^N \int_{\bX} \left( \int_{\bA} \resizebox{0.35\hsize}{!}{$ \left( C_t(x,\overline\xi_t,a) + \int_{\bX}\overline V^*_{t+1}(y)P_{x,\overline\xi_t,a}(\dif y) \right) $} \fp^n_{t,x,\overline\xi_t}(\dif a) - \overline V^*_t(x) \right) \; \overline\bP^{X^n_{t}} (\dif x) \nonumber\\
&= \sum_{t=1}^T \int_{\bX\times\bA} \left( C_t(x,\overline\xi_t, a) + \int_{\bX}\overline V^*_{t+1}(y)P_{x,\overline\xi_t,a}(\dif y) - \overline V^*_t(x) \right) \overline\psi_t(\dif x\dif a) \nonumber\\
&= \sum_{t=1}^T \int_{\bX} \left(  \int_\bA \left(C_t(x,\overline\xi_t, a) + \int_{\bX}\overline V^*_{t+1}(y)P_{x,\overline\xi_t,a}(\dif y) \right) \overline\pi^{\overline\psi_t}_x(\dif a) - \overline V^*_t(x) \right) \xi^{\overline\psi_t}(\dif x).
\end{align}
In view of \eqref{eq:ReprDynamics} and \eqref{eq:ReprPolicy}, we subsequently rewrite \eqref{eqn:average-stepwise-approx} into
\begin{align}\label{eq:HeurCaptureExpn}
\frac1N \sum_{n=1}^N R^n(\bm\fP) 
&\approx \sum_{t=1}^T \overline\bE\left[ \overline\bE\left[C_t(\overline X_t, \overline\xi_t, \overline A_t) + \overline V^*_{t+1}(\overline X_{t+1})\Big| \overline X_t \right] - \overline V^*_t(\overline X_{t}) \right],
\end{align}
In view of \eqref{eq:MFDPP}, using the same reasoning that leads to \eqref{eq:HeurDefStepwiseExploitability}, we obtain the following approximation as one of our keys results 
\begin{align}\label{eq:ApproxAveNplayerExploitability}
\frac1N \sum_{n=1}^N R^n(\bm\fP) &\approx \overline\bE\left[ \sum_{t=1}^T C_t(\overline X_t, \overline\xi_t, \overline A_t) \right] - \min_{\overline\fp} \overline\bE\left[ \sum_{t=1}^T C_t(\overline X_t, \overline\xi_t, \overline A_t) \right] =: \overline R(\overline\Psi).
\end{align}
Above, we call $\overline R$ the \textit{mean field end exploitability}. Here, we emphasize that that, in analogy with  \eqref{eq:HeurDefEndExploitability}, $\overline R$ is also computed under the hypothetical case in MFG setting where all (infinitely many) other players maintain their original policies while the representative player may vary hers. Analogously to $N$pG, we term the right hand side of \eqref{eq:HeurCaptureExpn} as \textit{mean field total (stepwise) exploitability}. Clearly, $\overline\Psi$ is a $\varepsilon$-MFE if and only if $\overline R(\overline\Psi)\le\varepsilon$.

To proceed, we briefly highlight two important results. These results complement \eqref{eq:ApproxAveNplayerExploitability} in offering a comprehensive description of the approximation capabilities of MFGs. Firstly, as a (partial) enhancement of \cref{prop:HeurMFApprox} (a), we have that 
\begin{align}\label{eq:ApproxStateActionEmp}
\bP\left[\overline\delta_{(\bm X, \bm A)_t} \approx \overline\psi_t\right] \approx 1,
\end{align}
where the notation $\overline\delta_{(\bm X, \bm A)_t}$ is the empirical measure of $(\bm X, \bm A)_t$. Secondly, we consider a $\Psi=(\psi_1,\dots,\psi_T)\in\cP(\bX\times\bA)^{T}$ that satisfies
\begin{align}\label{eq:HeurDefMFF2}
\xi^{\psi_1}=\delta_{x_1} \quad\text{and}\quad \xi^{\psi_{t+1}}(B)=\int_{\bX\times\bA} P_{x,\xi^{\psi_t},a}(B)\psi_t(\dif x \dif a),\;B\in\cB(\bX),\,t=1,\dots,T-1,
\end{align}
i.e., $\Psi$ can be related to some scenario in the MFG setting. For $\psi \in \cP(\bX \times \bA)$, we define $\overline{\pi}^{\psi}$ as the conditional distribution of actions given the state. This definition naturally leads to an action kernel that is solely dependent on the state of the representative player, i.e., the law of the random action remains invariant under changes in the empirical state distribution. We call $\fp^{\Psi}=(\overline\pi^{\psi_1},\dots,\overline\pi^{\psi_T})$ the the policy induced by $\Psi$. It can be shown that the homogeneous $N$-player scenario $\overline{\bm\fP}^{\Psi} = (\fp^{\Psi},\dots,\fp^{\Psi})$ satisfy
\begin{align}\label{eq:ApproxMFExploitability}
\overline R(\Psi) \approx R^n(\overline{\bm\fP}^{\Psi}),\quad n=1,\dots,N.
\end{align} 
Results such as \eqref{eq:ApproxStateActionEmp} and \eqref{eq:ApproxMFExploitability} are generally better understood in extant literature compared to \eqref{eq:ApproxAveNplayerExploitability}. Nonetheless, proving \eqref{eq:ApproxStateActionEmp} in the context of $N$pG with features like heterogeneous agents and non-asymptotics demands some additional treatment. Regarding \eqref{eq:ApproxMFExploitability}, while it is well-established in many risk-neutral settings (at least for when $\Psi$ is a MFE), it still requires case-by-case analysis in various risk-averse or risk-sensitive settings. Establishing \eqref{eq:ApproxMFExploitability} in a general setting is another objective of this paper. 

We summarize the discussion above into following heuristic theorems. For formal statements, we refer to \cref{thm:MFApprox} and \cref{thm:MFConstr}, respectively.
\begin{theorem}\label{thm:HeurNoMiss}
Suppose \cref{assump:Heur} and let $\overline\Psi$ be constructed in \eqref{eq:HeurDefMFF}. Then, we have $\bP\left[\overline\delta_{(\bm X, \bm A)_t} \approx \overline\psi_t\right] \approx 1$ and $\frac1N \sum_{n=1}^N R^n(\bm\fP)\approx\overline R(\overline\Psi)$.
\end{theorem}
\begin{theorem}\label{thm:HeurNoWaste}
For any $\Psi=(\psi_1,\dots,\psi_T)$ that satisfies \eqref{eq:HeurDefMFF2}, consider the induced $N$-player scenario $\overline{\bm\fP}^{\Psi} = (\fp^{\Psi},\dots,\fp^{\Psi})$. Then,  $\overline R(\Psi) \approx R^n(\overline{\bm\fP}^{\Psi})$ for $n=1,\dots,N$.
\end{theorem}

\subsection{Comments on \cref{thm:HeurNoMiss} and \cref{thm:HeurNoWaste}}\label{subsec:HeurComment}
In general, for an $N$pG with a large $N$ that allows heterogeneous policies, it is exceedingly challenging to characterize the entire spectrum of (approximate) equilibria. It is, however, plausible that different combinations of individual policies could still result in similar population statistics in terms of state-action distribution. This leads to the natural emergence of an approximate refinement of equilibria in $N$pGs via state-action distributions. Under \cref{assump:Heur}, in view of \cref{thm:HeurNoMiss}, we can categorize different ${\bm\fP}$'s based on the corresponding $\overline\Psi$'s defined in \eqref{eq:HeurDefMFF}. Such approximate refinement eradicates redundancy stemming from permutations of player labels. It also conserves effort by accounting for the situation where distinct $\fP$'s could lead to similar state-action distributions. 
 
Beyond the refinement perspective, \cref{thm:HeurNoMiss} and \cref{thm:HeurNoWaste} together advocate for the use of $\delta\text{-}\argmin_{\Psi}\overline R(\Psi)$, and ultimately the examination of $\Psi\mapsto\overline R(\Psi)$. Indeed, \cref{thm:HeurNoMiss} asserts that whenever an $\varepsilon$-equilibrium in the average sense is achieved in an $N$pG, there must be some 
$$\overline\Psi\in\delta\text{-}\argmin_{\Psi}\overline R(\Psi)$$ 
that approximates the (empirical) state-action distributions of said $\varepsilon$-equilibrium. In other words, MFGs do not overlook equilibria in $N$pGs. Conversely, \cref{thm:HeurNoWaste} demonstrates that any element from $\varepsilon\text{-}\argmin_{\Psi}\overline R(\Psi)$ can be associated with a scenario in an $N$pG that is $\delta$-equilibrium, implying that $\varepsilon\text{-}\argmin_{\Psi}\overline R(\Psi)$ does not produce nonsensical equilibria.

In concluding this section, we juxtapose \cref{thm:HeurNoMiss} with results established in \cite{Lacker2020Converrgence} and \cite{Iseri2022Set}. These results represent important progress in understanding the capturing capabilities of MFGs without requiring the uniqueness of equilibria, offering valuable insights that have informed our current research. \cite{Lacker2020Converrgence} explores continuous-time games driven by diffusion with controlled drift. It is shown in \cite[Theorem 2.7]{Lacker2020Converrgence} that, under suitable conditions, any subsequential limits of $N$-player approximate equilibria are weak MFE, a relaxed notation of MFE. A key strength of this result lies in its minimal assumptions on player policies in $N$pGs, in contrast to our \cref{assump:Heur} (iii) (iv). However, the asymptotic approach of \cite{Lacker2020Converrgence} leaves some ambiguity in the relationship between an $N$pE for a finite $N$ and the corresponding MFE. In contrast, \cref{thm:HeurNoMiss} offer a non-asymptotic perspective to this issue under additional assumptions.
\cite{Iseri2022Set} works in mostly in discrete-time setups, and utilizes a set of conditions that arguably parallels those in \cref{assump:Heur}. They focus on the similarity between value functions from $N$pGs and MFGs at approximate equilibria.  Our interest, however, lies more in linking (empirical) state-action distributions from both games with exploitabilities, although we do share some common heuristics as the concepts of propagation of chaos is usually indispensable in such analysis. We would also like to highlight our technical endeavor. For instance, we extend the discussion from discrete spaces to Polish spaces (though still in discrete time). We also remove some conditions that we do not deem necessary for our interest, such as the restriction to deterministic policies in $N$pGs.

\subsection{Technical concerns}\label{subsec:Tech}

In this section, we will outline some of the challenges associated with formally establishing the heuristic results presented in \cref{subsec:HeurDerv} in a more general setting.

\paragraph{Exploitability under continuity constrain} In light of \cref{assump:Heur}, it is natural to consider in $N$pGs an exploitability with candidate policies that also satisfy \cref{assump:Heur} (iii) (iv). Here we simply call this the constrained exploitability. Unfortunately, due to the lack of DPP for the related constrained problem in MDPs, the running optimal value function is not well-defined, hindering derivation of an analogue to \eqref{eq:ApproxAveNplayerExploitability} for constrained exploitability. However, this issue can be side-stepped by arguing that the constrained exploitability is approximately no different from the unconstrained one. We refer to \cref{sec:Exploitability} for more discussion. The proof of this side-stepping argument in fact largely overlaps with that of {eq:ApproxAveNplayerExploitability}.

\paragraph{Risk aversion} Following \cite{Ruszczynski2010Risk}, we would like to incorporate into our models risk averse agents. Since\cite{Ruszczynski2010Risk} does not consider randomized actions, we use the extended version in \cite{Chu2014Markov} for illustration. Heuristically, a typical form of policy evaluation in risk averse MDPs reads (likewise for the Bellman equation)
\begin{align*}
v_t(x) = \int_{\bA} \big(c( x,a) + \rho(v_{t+1}(X^{x,a})) \big) \fp_{t,x}(\dif a),
\end{align*}
where $X^{x,a}$ is generated from some distribution that may depends on $(x,a)$, and $\rho$ could be a convex risk measures, say average value-at-risk (see \cite[Section 6.2.4]{Shapiro2021book}). The usual absence of linearity in $\rho$ hinders the derivation of total stepwise exploitability, which is crucial for establishing \eqref{eq:ApproxAveNplayerExploitability}. This issue will be discussed with more details in \cref{sec:Exploitability}.

\paragraph{Probabilistic kernels on Polish spaces and related continuities} In the following analysis, we use complete separable metric spaces (e.g., $\bR^d$) for state and action spaces. This introduces additional technical complexity. In particular, we need to choose between weak or strong continuity in replacement of \cref{assump:Heur} (i) (iv), an issue that is trivial in finite environment as the two concepts coincides.\footnote{Roughly speaking, for measure-valued function, weak continuity associates the output (probability) space with weak convergence, while strong continuity associates the output space with set-wise convergence.} We note that, weak continuity are in general a weaker condition, allowing, e.g., dynamics with atoms, a feature that is often desirable but not compatible with strong continuity. However, equipping $({x,\xi, a})\mapsto P_{x,\xi, a}$ with weak continuity typically requires such continuity to present in not only all inputs jointly but also other components of the models, such as $C$ and $\fp^n_t$. In contrast, we later find out that imposing strong continuity on $(\xi,a)\mapsto P_{x,\xi,a}$ allows discontinuity in the $x$ argument of $P, C, \fp^n_t$, while still enabling certain desirable approximation property. This ultimately enables us to encompass in our approximation analysis a broader spectrum of $N$-player scenarios (at the expense of more stringent assumptions on $\xi,a$ inputs of $P$). In the meantime, we would like to point out that jointly weak continuity, though not a necessity in approximation analysis, will be eventually needed in showing the existence of MFEs via Kakutani-Fan-Glicksberg fixed point theorem. There is not a universal rule on which continuity is more suitable. In our approximation analysis, we have opted to work with a mixture of continuities, while assuming joint weak continuity for establishing the existence of MFEs. Of course, the ideal case would be that all continuity assumptions are satisfied simultaneously. 

\paragraph{Complex notations and formulas} 
The heuristic derivation in \cref{subsec:HeurDerv} already involves complex notations. A major reason for this complexity is the need to work with difference processes when comparing a scenario in $N$pGs against the counterpart in MFGs. These processes are potentially defined on different probability bases, which in turn demands further notations for handling related calculations. For the sake of simplicity, we have decided to focus solely on the operator perspective for the remainder of the paper. Another challenge arises from lengthy expressions like \eqref{eq:HeurPreCapture}, which, combined with numerous upcoming applications of the triangle inequality, as well as the consideration of risk aversion, add to the complexity. To reduce the notational complexity, we propose to adopt an abstract formulation. This approach proves particularly beneficial in the context of triangle inequalities, as it narrows down the search area for identifying key differences.

\vspace{.3in}

\section{Setup and Preliminaries}\label{sec:Setup}
We initiate our discussion by introducing some notations in \cref{subsec:Notation} that are used throughout this paper. The dynamics of players are summarized in \cref{subsec:Dynamics}, followed by the definitions of transition operators in \cref{subsec:TransOp}. In \cref{subsec:ScoreOp}, we present a series of score operators that formulate the performance criteria for the players. The concept of state-action mean field flow is formally introduced in \cref{subsec:MFF}. \cref{subsec:Scenarios} presents a crucial construction of the mean field flow associated with a specific $N$-player scenario. We compile the technical assumptions in \cref{subsec:Assumptions}.  An important technical lemma that underscores our main results on mean field approximation is presented in \cref{subsec:EmpMeasConc}. Lastly, in \cref{subsec:ErrorTerms}, we introduce several additional error terms.

\subsection{Preliminary notations}\label{subsec:Notation}

We first summarize the rules of notations. Bold-faced notations are associated with $N$pGs, while overlined notations pertain to MFGs. Subscripts typically refer to arguments or parameters whose effects are more directly defined—such as time or (empirical) population distributions. Conversely, superscripts indicate the involvement of more complex operations during definition, like computing marginal or conditional distributions. Additionally, superscripts are used to index individual players in $N$pGs.

In what follows, we let $t\in\set{1,2,\dots,T}$ be the time index. We fix the total number of players in the $N$pG for the remainder of the paper.

We say $\beta$ is a \textit{subadditive modulus of continuity} if $\beta:[0,\infty]\to\overline[0,\infty]$ is non-decreasing, subaditive and satisfies $\lim_{\ell\to0+}\beta(\ell) = 0$. It can be shown that $\beta$ is concave. For convenience, we sometimes view $\beta$ as an operator so that $\beta f=\beta(f(\cdot))$ for any $[0,\infty]$-valued function $f$. 

The state space $\bX$ is a complete separable metric space and $\cB(\bX)$ is the corresponding Borel $\sigma$-algebra. $\cP(\bX)$ is the set of probability measures on $\cB(\bX)$, we endow $\cP(\bX)$ with the weak topology, $\cB(\cP(\bX))$ is the corresponding Borel $\sigma$-algebra. We let $\cE(\cP(\bX))$ be the evaluation $\sigma$-algebra, i.e.,  $\cE(\cP(\bX))$ is the $\sigma$-algebra generated by sets $\set{\xi\in\cP(\bX):\int_{\bY}f(y)\xi(\dif y) \in B}\,$ for any real-valued bounded $\cB(\bX)$-$\cB(\bR)$ measurable $f$ and $B\in\cB(\bR)$. Due to \cref{lem:sigmaAlgBE}, $\cB(\cP(\bX))=\cE(\cP(\bX))$. 

The action domain $\bA$ is another complete separable metric space and $\cB(\bA)$ is the corresponding Borel $\sigma$-algebra. $\cP(\bA)$ is the set of probability measures on $\cB(\bA)$, we endow $\cP(\bA)$ with the weak topology, $\cB(\cP(\bA))$ is the corresponding Borel $\sigma$-algebra, $\cE(\cP(\bA))$ is the evaluation $\sigma$-algebra. Again, by  \cref{lem:sigmaAlgBE}, we have $\cB(\cP(\bA))=\cE(\cP(\bA))$.

The Cartesian product of metric spaces in this paper are always equipped with the metric equal to the sum of metrics of the component spaces and endowed with the corresponding Borel $\sigma$-algebra.

We let $\cP(\bX\times\bA)$ be the set of probability measures on $\cB(\bX\times\bA)$ endowed with weak topology. For $\psi\in\cP(\bX\times\bA)$, we let $\xi^{\psi}$ be the marginal measure of $\psi$ on $\bX$. Additionally, we find it convenient to denote $\Psi = (\psi_1, \dots, \pi_{T-1}) \in \cP(\bX \times \bA)^{T-1}$ and $\Xi^\Psi = (\xi^{\psi_1}, \dots, \xi^{\psi_{T-1}})$.

We let $\delta_y(B):=\1_{B}(y)$ be the Dirac measure at $y$. Additionally, we write $\overline\delta_{\bm y}:=\frac1N\sum_{n=1}^N\delta_{y^n}$, where $\bm y=(y^1,\dots,y^N)$.

Let $\bY$ be a complete separable metric space and $\cB(\bY)$ be the corresponding Borel $\sigma$-algebra. Let $\cP(\bY)$ be the set of probability measures on $\cB(\bY)$, endowed with weak topology and Borel $\sigma$-algebra. $B_b(\bY)$ is the set of Borel measurable real-valued bounded functions on $\bY$. Subsequently, $C_b(\bY)\subseteq B_b(\bY)$ is the set of continuous functions, equipped with sup-norm. Let $L>0$ and let $C_{L-BL}(\bY)$ be the set of $L$-Lipschitz continuous functions bounded by $L$. For finite measures $m,m'$ on $\cB(\bY)$, we define the following norm and metric:
\begin{gather*}
\|m\|_{L-BL} := \sup_{h\in C_{L-BL}}\frac12 \int_\bY h(y) m(\dif y),\\
d_{L-BL}(m,m') := \|m-m'\|_{L-BL} = \sup_{h\in C_{L-BL}(\bY)}\frac12\left|\int_{\bY}h(y)m(\dif y) - \int_{\bY}h(y)m'(\dif y)\right|. 
\end{gather*}  
We write $C_{BL}$ and $\|\cdot\|_{BL}$ for abbreviations of $C_{1-BL}$ and $\|\cdot\|_{1-BL}$. Clearly, $\|m-m'\|_{L-BL} = L \|m-m'\|_{BL}$ and $\|m-m'\|_{BL}\le 1$. Below we recall a well-known result regarding the relationship between $\|\,\cdot\,\|_{BL}$ and weak convergence of probabilities; see, for example, \cite[Theorem 8.3.2]{Bogachev2007book}.
\begin{lemma} Let $(\mu_n)_{n\in\bN}\subseteq\cP(\bY)$ and $\mu\in\cP(\bY)$. Then,  $\lim_{n\to\infty}\|\mu_n-\mu\|_{BL}=0$ if and only if $\mu_n$ converges weakly to $\mu$. 
\end{lemma}

Throughout the rest of the paper, we will consistently associate spaces with their respective Borel \(\sigma\)-algebras and consider exclusively measurable mappings. Consequently, notations like \(\mathcal{B}(\mathbb{X})\) and \(\mathcal{B}(\mathcal{P}(\mathbb{X}))\) will be omitted, except where specifically noted.

\subsection{Controlled dynamics of the $N$pG}\label{subsec:Dynamics} 

Throughout the rest of the paper, we will fix $N$, the total number of players in the finite-player game. The scripts indicating this number is omitted for brevity. The only exception is the peripheral result, \cref{lem:ErrorConv}, which connects our non-asymptotic analysis to an asymptotic setting.

For $t=1,\dots,T$, let $P_t:\bX\times\cP(\bX)\times\bA\to\cP(\bX)$ be the transition kernel. 

We define $\Pi$ as the set of $\pi:\bX^N\to\cP(\bA)$, i.e., $\Pi$ is the set of Markovian action kernels. Let $\wt\Pi$ be the set of $\tilde\pi:\bX\times\cP(\bX)\to\cP(\bA)$, i.e., $\wt\Pi$ is the set of symmetric Markovian action kernels. Additionally, $\overline\Pi$ is the set of $\overline\pi:\bX\to\cP(\bA)$, i.e.,  $\overline\Pi$ is the set of oblivious Markovian action kernels \cite{Weintraub05Oblivious}. ``Oblivious'' refers to a player's decision-making process that considers only their own state, ignoring the states of other players. Sometimes, it is convenient to extend elements in $\overline\Pi$ (resp. $\wt\Pi$) to elements in $\wt\Pi$ (resp. $\Pi$) by considering $\overline\pi$ as constant in $\cP(\bA)$ (resp. redacting $\bm x$ to $(x^n,\overline\delta_{\bm x})$ from some $n$, depending on the context). In this sense, we have $\overline\Pi\subseteq\wt\Pi\subseteq\Pi$.

Let $N\in\bN$ be the number of players. For the remainder of the paper, we will fix $N$ and, unless necessary, omit it from the notation concerning the entire $N$pG.  We use $\fp^n=(\fp^n_1,\dots,\fp^n_{T-1})\in\Pi^{T-1}$ to denote the policy of player-$n$, and $\bm\fP=(\fp^1,\dots,\fp^N)\in\Pi^{N\times(T-1)}$ to denote the collection of policies in an $N$pG. Occasionally, we will use $\bm\fP=(\bm\fP_1,\dots,\bm\fP_{T-1})$, where $\bm\fP_t=(\fp^1_t,\dots,\fp^N_t)\in\Pi^N$ -- this notation shows the time index in the policy for all players, while the former one shows the player index for each player, however, the set of action kernels considered is the same in both cases; this notation is also more convenient for compositions of operators, see e.g., \eqref{eq:DefcT} and \eqref{eq:DefcS} later. 

A $\bm\fP\in\Pi^{N\times(T-1)}$ corresponds to an $N$pG scenario with the following dynamics. The position of each player at $t=1$ is drawn independently from $\mathring{\xi}\in\cP(\bX)$. At time $t=2,\dots,T-1$, given the realization of state $\bm x=(x^1,\dots,x^N)$, the generation of the actions of all player is governed by $\bigotimes_{n=1}^N \fp^n_{t,\bm x}$. In particular, for $\wt{\bm\fP}\in\wt{\Pi}^{N\times(T-1)}$, the generation is then governed by $\bigotimes_{n=1}^N\tilde\fp^n_{t,x^n,\overline\delta_{\bm x}}$. Moreover, given the realization of actions $\bm a = (a^1,\dots,a^N)$, the transition of the states of all $N$ players is governed by $\bigotimes_{n=1}^N P_{t,\bm x,\overline\delta_{\bm x},a^n}$. In summary,  we have the following transition dynamics:
\begin{align}\label{eq:NpDyn}
	\bm X_1=(x_1,\dots,x_1),\quad \bm A_t\sim\bigotimes_{n=1}^N\tilde\fp^n_{t,X^n_t,\overline\delta_{\bm X_t}}\,,\quad \bm X_{t+1}\sim\bigotimes_{n=1}^N P_{t,X^n_t,\overline\delta_{\bm X_t}, A^n_t}.
\end{align} 
For notional convenience, in what follows, we use $\bm\lambda=(\lambda^1,\dots,\lambda^N)\in\cP(\bA)^N$ for an $N$-tuple of elements from $\cP(\bA)$.

\subsection{Transition operators}\label{subsec:TransOp}
In this section, we introduce several operators related to state transitions. All measurability issues can be addressed by referring to \cref{lem:IntfMeasurability}, and as such, the related details are not reiterated here.

In order to introduce the transition and expectation operators for $N$pGs, we first define an auxiliary notation. With $t=1,\dots,T-1$ and $(x,\xi,\lambda)\in\bX\times\cP(\bX)\times\cP(\bA)$, we let
\begin{align}\label{eq:DefQ}
Q^\lambda_{t,x,\xi}(B) := \int_{\bA} P_{t,x,\xi,a}(B)\lambda(\dif a), \quad B\in\cB(\bX).
\end{align} 
Consider additionally $\bm x\in\bX^N$ and $\bm\lambda\in\cP(\bA)^N$. By utilizing standard tools in probability (e.g., \cite[Section 4]{Aliprantis2006book}), we obtain
\begin{align}\label{eq:IndIntProdQ}
&\int_{\bX^N} f\left(\bm y\right) \left[\bigotimes_{n=1}^N Q^{\lambda^n}_{t,x^n,\xi}\right]\left(\dif\bm y\right)\nonumber\\
&\quad=\int_{\bA^N} \int_{\bX^N}f\left(\bm y\right)\left[\bigotimes_{n=1}^N P_{t,x^n,\overline\delta_{\bm x},a^n}\right]\left(\dif\bm y\right)   \left[\bigotimes_{n=1}^N\lambda^n\right]\left(\dif\bm a\right),\quad f\in B_b(\bX^N).
\end{align}

We are ready to introduce the transition and expectation operators for $N$pGs. Given $\bm\pi=(\pi^1,\dots,\pi^N)\in\Pi^N$, we define 
\begin{align}\label{eq:DefT}
\bm T^{\bm\pi}_t f\left(\bm x\right) := \int_{\bX^N}f(\bm y) \left[\bigotimes_{n=1}^N Q^{\pi^n_{\bm x}}_{t,x^n,\overline\delta_{\bm x}}\right]\left(\dif\bm y\right),\; t=1,\dots,T-1, \quad f\in B_b(\bX^N).
\end{align}
With $\bm\fP\in\Pi^{N\times(T-1)}$, we continue to define
\begin{align}\label{eq:DefcT}
{\bm \cT}^{\bm\fP}_{s,s}f := f \quad\text{and}\quad {\bm \cT}^{\bm\fP}_{s,t}f\left(\bm x\right) := \bm T^{\bm\fP_s}_s\circ\cdots\circ\bm T^{\bm\fP_{t-1}}_{t-1} f\left(\bm x\right), \;1\le s < t \le T, \quad f\in B_b(\bX^N).
\end{align}
We also define the following expectation operator
\begin{align}\label{eq:DeffT}
\mathring{\bm\bT} f := \int_{\bX^N}f(\bm y)\mathring{\xi}^{\otimes N}(\dif\bm y) \quad\text{and}\quad {\bm\bT}^{\bm\fP}_{t}f := \mathring{\bm\bT} \circ {\bm\cT}^{\bm\fP}_{1,t}f,\; t=1,\dots,T, \quad f\in B_b(\bX^N).
\end{align}
It is clear that ${\bm\bT}^{\bm\fP}_{t}$ is a non-negative linear functional on $ B_b(\bX^N)$ with operator norm of $1$. It follows that ${\bm\bT}^{\bm\fP}_{t}$ is also non-decreasing. Additionally, the definitions provided above are applicable to symmetric/oblivious Markovian policies.

Below we introduce the analogous transition and expectation operators for MFGs. With $\xi\in\cP(\bX)$ and $\tilde\pi\in\wt\Pi$, we define 
\begin{align}\label{eq:DefMeanT}
\overline{T}^{\tilde\pi}_{t,\xi} h(x) := \int_{\bX} h(y) Q^{\tilde\pi_{x,\xi}}_{t,x,\xi}(\dif y),\; t=1,\dots,T-1, \quad h\in B_b(\bX).
\end{align}
Furthermore, with $\Xi=(\xi_1,\cdots,\xi_{T-1})\in\cP(\bX)^T$ and $\tilde\fp=(\tilde\fp_1,\dots,\tilde\fp_{T-1})\in\wt{\Pi}^{T-1}$, we define
\begin{align}\label{eq:DefMeancT}
\overline{\cT}^{\tilde\fp}_{s,s,\Xi}h:=h \quad\text{and}\quad \overline{\cT}^{\tilde\fp}_{s,t,\Xi}h := \overline{T}^{\tilde\fp_s}_{s,\xi_s}\circ\cdots\circ\overline{T}^{\tilde\fp_{t-1}}_{t-1,\xi_{t-1}} h,\; 1\le s<t\le T,\quad h\in B_b(\bX).
\end{align} 
In a similar manner as before, we define the expectation operator in MFG as
\begin{align}\label{eq:DefMeanfT}
\mathring{\overline{\bT}}h:=\int_\bX h(x)\mathring\xi(\dif x) \quad\text{and}\quad \overline{\bT}^{\tilde\fp}_{t,\Xi}h := \mathring{\overline{\bT}} \circ \overline{\cT}^{\tilde\fp}_{1,t,\Xi}h,\, t=1,\dots,T,\quad h\in B_b(\bX).
\end{align}
Clearly, $\overline{\bT}^{\tilde\fp}_{t,\Xi}$ is a non-negative linear functional on $ B_b(\bX)$, and thus $\overline{\bT}^{\tilde\fp}_{t,\Xi}$ is non-decreasing. Moreover, it can be verified via induction that $D\mapsto\overline{\bT}^{\tilde\fp}_{t,\Xi}\1_{D}$ is a probability measure on $\cB(\bX)$, and we denote it by $\overline{\bQ}^{\tilde\fp}_{t,\Xi}$. Clearly, $\overline{\bQ}^{\tilde\fp}_{1,\Xi}=\mathring{\xi}$.

\subsection{Score operators}\label{subsec:ScoreOp}
The operators discussed in this section, combined with the assumptions detailed in \cref{subsec:Assumptions} later, abstractly depict the performance evaluation process that adheres to a finite horizon dynamic programming principle.

For $N$pGs, with $\bm\lambda\in\cP(\bA)^N$, we define an operators $\bm G^{\bm\lambda}_{t}: B_b(\bX^N) \to B_b(\bX^N)$. For MFGs, with $\lambda\in\cP(\bA)$ and $\xi\in\cP(\bX)$, we consider $\overline G^{\lambda}_{t,\xi}: B_b(\bX)\to B_b(\bX)$. Both $\bm G^{\bm\lambda}_t$ and $\overline G^\lambda_{t,\xi}$ are employed to construct mappings that facilitate the backward induction of value functions, transitioning them from time step $t+1$ to time step $t$. Additionally, we consider real-valued functionals $\mathring{\bm\bG}$ and $\mathring{\overline\bG}$ on $B_b(\bX^N)$ and $B_b(\bX)$, respectively. These functionals facilitate the final step of backward induction, yielding a real number as the outcome score. Our abstract formulation with $\bm G^{\bm\lambda}_t$ and $\overline G^\lambda_{t,\xi}$ aims to encompasses not only Bellman operators commonly studied in risk-neutral mean field MDPs (e.g., \cite[(3.1)]{Saldi2018Markov} and \cite[(16)]{Lauriere23Model}), but also risk-averse variant such as the example provided in \cref{subsec:ExmpG}. As a by-product, certain extended form of cost is also allowed, for example, \eqref{eq:ExmpExtendedCost}. However, we note that, while Bellman operators incorporating Kullback–Leibler divergence (e.g., \cite[(9)]{Leahy22Convergence}) can also be cast into our abstraction, the application is currently limited to discrete settings because Kullback–Leibler divergence is not compatible with latter assumptions (\cref{assump:GCont} and \cref{assump:GCont2}) in continuous spaces such as $\bR^d$.

Throughout the paper, we always assume the measurabilities below:
\begin{itemize}
\item[(i)] For any $\bm u\in B_b(\bX^N)$, the mapping $(\bm x,\bm\lambda)\mapsto\bm G^{\bm\lambda}_{t}\bm u(\bm x)$ is $\cB(\bX^N\times\cP(\bA)^N)$-$\cB(\bR)$ measurable.
\item[(ii)] For any $v\in B_b(\bX)$, the mapping $(x,\xi,\lambda)\mapsto \overline G^{\lambda}_{t,\xi} v(x)$ is $\cB(\bX\times\cP(\bX)\times\cP(\bA))$-$\cB(\bR)$ measurable.
\end{itemize}
In most cases, these measurabilities can be verified using standard tools such as listed in \cite[Section 4, Section 11, and Section 18]{Aliprantis2006book}.

$\bm G^{\bm\lambda}_t$ is designed to evaluate the performance of player-1 in an $N$pG. Under certain symmetry assumptions (see \cref{assump:GBasic} (iv) later), $\bm G^{\bm\lambda}_t$ can also be applied to assess other players through permutations. It is crucial that $\bm G^{\bm\lambda}_t$ and $\overline G^\lambda_{t,\xi}$ maintain a specific relationship to facilitate approximating an $N$pG with a MFG. These assumptions will be introduced in \cref{subsec:Assumptions}. For examples of $\bm G^{\bm\lambda}_t$ and $\overline G^\lambda_{t,\xi}$, we refer to \cref{subsec:ExmpG}.

In the $N$pG, we assume all players aim to minimize their scores. More precisely, we let $\bm U\in B_b(\bX)$ be the terminal cost function. With the score operator $\bm\bS^{\bm\fP}_{T}$ introduced later in \eqref{eq:DeffS}, player-$1$ aims to find 
\begin{align}\tag{NpObj}\label{eq:NplayerObj}
\inf_{\fp^1\in\wt{\Pi}^{T-1}_\vartheta}\bm\bS^{\bm\fP}_{T}\bm U,
\end{align}
where $\vartheta$ is a subadditive modular of continuity, and $\wt{\Pi}_\vartheta$ is an subset of $\wt{\Pi}$ such that for any $\tilde\pi\in\wt{\Pi}_\vartheta$ we have $\xi\mapsto\tilde\pi(x,\xi)$ is $\vartheta$-continuous in $\cP(\bX)$ for any $x\in\bX$, where $\cP(\bX)$ and $\cP(\bA)$ are equipped with $d_{BL}$. We call $\tilde\fp\in\wt{\Pi}^{T-1}_\vartheta$ the \textit{$\vartheta$-symmetrically continuous policy}. This objective aligns with the scenario we aim to explore in this paper, where all players employ a $\vartheta$-symmetrically continuous policy. Heuristically, this scenario might arise from practical situations where players have limited accuracy or certainty in perceiving the empirical population distribution, making them reluctant to drastically alter their behavior in response to changes in the population. Besides, to highlight the technical importance of symmetrical continuity, we note that a mild $\vartheta$ is typically required for the mean field approximation to be effective. For an illustrative example, please refer to \cref{subsec:ExmpGame}.

To facilitate the comparison between $N$pGs and MFGs, we consider the following hypothetical objective in MFG from the perspective of the representative player. Let $V\in B_b(\bX\times\cP(\bX))$ be the terminal cost and write $V_\xi(x):=V(x,\xi)$. Given a $\Xi\in\cP(\bX)^{T-1}$ and $\xi_T\in\cP(\bX)$, with the score operator $\overline\bS^{\tilde\fp}_{T,\Xi}$ introduced later in \eqref{eq:DefMeanfS}, the representative player aims to find
\begin{align}\tag{MFObj}\label{eq:MFObj}
\inf_{\tilde\fp\in\wt{\Pi}^{T-1}}\overline\bS^{\tilde\fp}_{T,\Xi} V_{\xi_T},
\end{align}
Above, $\tilde\fp\in\wt{\Pi}^{T-1}$ can be effectively replaced by $\bar\fp\in\overline{\Pi}^{T-1}$ as $\Xi$ and $\xi_{T}$ are fixed. Note that \eqref{eq:MFObj} corresponds to the hypothetical case where the representative player is allowed to revise her policy while other players maintain their original behaviour so that $\Xi$ and $\xi_{T}$ remain unchanged.

In the remainder of this section, we will provide the definitions of the aforementioned score operators.

For the $N$pG, we let $\bm\pi=(\pi^1,\dots,\pi^N)\in\Pi^N$ and $\bm\fP=(\bm\fP_1,\dots,\bm\fP_{T-1})\in\Pi^{N\times(T-1)}$. For $\bm u\in B_b(\bX^N)$ we define
\begin{gather}
\bm S^{\bm\pi}_{t}\bm u({\bm x}) := \bm G^{(\pi^n_{\bm x})_{n=1}^N}_{t} \bm u (\bm x), \quad t=1,\dots,T-1,\label{eq:DefS}\\
\bm\cS^{\bm\fP}_{s,s}\bm u = \bm u \quad\text{and}\quad \bm\cS^{\bm\fP}_{s,t}\bm u := \bm S^{\bm\fP_s}_{s}\circ\cdots\circ \bm S^{\bm\fP_{t-1}}_{t-1}\bm u,\quad 1\le s<t\le T,\label{eq:DefcS}\\
\bm\bS^{\bm\fP}_{t} \bm u := \mathring{\bm\bG} \circ \bm\cS^{\bm\fP}_{1,t}\bm u,\quad t=1,\dots,T.\label{eq:DeffS}
\end{gather} 
We present the associated optimization operators below. Note that these optimization operators are not directly related to \eqref{eq:NplayerObj}, as the constrain in \eqref{eq:NplayerObj} hinders the derivation of dynamic programming principle. Nevertheless, the following operators serve as important auxiliary tools in our analysis. We first define the (one-step) Bellman operator
\begin{align}\label{eq:DefSstar}
\bm S^{*\bm\pi}_t \bm u(\bm x) := \inf_{\lambda\in\cP(\bA)} G^{(\lambda,\pi^2_{\bm x},\dots,\pi^N_{\bm x})}_t \bm u(\bm x), \quad t=1,\dots,T-1.
\end{align}
We note that $\bm S^{*\bm\pi}_t \bm u$ does not depend on $\pi^1$ but we insist on this notation for the sake of neatness. Under suitable conditions (see \cref{lem:Sstar}), we further define: 
\begin{gather}
\bm\cS^{*\bm\fP}_{s,s} \bm u := \bm u \quad\text{and}\quad \bm\cS^{*\bm\fP}_{s,t} \bm u :=  \bm S^{*\bm\fP_s}_s \circ \cdots \circ \bm S^{*\bm\fP_{t-1}}_{t-1} \bm u,\quad 1\le s<t\le T,\label{eq:DefcSstar}\\
\bm\bS^{*\bm\fP}_{t} \bm u := \mathring{\bm\bG}\circ\bm\cS^{*\bm\fP}_{1,t} \bm u,\quad t=1,\dots,T.\label{eq:DeffSstar}
\end{gather}
Again, we note that $\bm\cS^{*\bm\fP}_{s,t}$ and $\bm\bS^{*\bm\fP}_{t}$ are constant in $\fp^1$.

Regarding MFGs, let $\xi\in\cP(\bX)$, $\tilde\pi\in\wt\Pi$, $\Xi=(\xi_1,\dots,\xi_{T-1})\in\cP(\bX)^{T-1}$, and $\tilde\fp=(\tilde\fp_1,\dots,\tilde\fp_{T-1})\in{\wt\Pi}^{T-1}$. To align with previous setting, unless specified otherwise, we set $\xi_1=\mathring\xi$. For $v\in B_b(\bX)$ we define
\begin{gather}
\overline S^{\tilde\pi}_{t,\xi} v(x) := \overline G^{\tilde\pi_{x,\xi}}_{t,\xi} v(x),\quad t=1,\dots,T-1,\label{eq:DefMeanS}\\
\overline \cS^{\tilde\fp}_{s,s,\Xi}v := v\quad\text{and}\quad \overline\cS^{\tilde\fp}_{s,t,\Xi} v := \overline S^{\tilde\fp_s}_{s,\xi_s} \circ \cdots \circ \overline S^{\tilde\fp_{t-1}}_{t-1,\xi_{t-1}} v , \; 1\le s<t\le T,\label{eq:DefMeancS}\\
\overline \bS^{\tilde\fp}_{t,\Xi} := \mathring{\overline\bG} \circ \overline\cS^{\tilde\fp}_{1,t,\Xi} v, \quad t=1,\dots,T. \label{eq:DefMeanfS}
\end{gather}
Note that $\overline\cS^{\tilde\fp}_{s,t,\Xi}$ depend on $\Xi$ only through $\xi_s,\dots,\xi_{t-1}$. In addition, we define the corresponding Bellman operator as 
\begin{align}\label{eq:DefMeanSstar}
\overline S^{*}_{t,\xi} v(x) := \inf_{\lambda\in\cP(\bA)}\overline G^{\lambda}_{t,\xi} v(x),\quad t=1,\dots,T-1.
\end{align}
Under suitable conditions (see \cref{lem:Sstar}), we further define: 
\begin{gather}
\overline\cS^*_{s,s,\Xi} v := v\quad\text{and}\quad \overline\cS^*_{s,t,\Xi} v := \overline S^*_{s,\xi_s} \circ\cdots\circ \overline S^*_{t-1,\xi_{t-1}} v,\; 1\le s<t\le T,\label{eq:DefMeancSstar}\\
\overline\bS^*_{t,\Xi} v := \mathring{\overline\bG} \circ \overline\cS^*_{t,\Xi} v,\quad t=1,\dots,T.\label{eq:DefMeanfSstar}
\end{gather}

Below we make a few obvious observations. Under the setting of \cref{lem:Sstar} (with technical assumptions to be introduced in \cref{subsec:Assumptions}), we have
\begin{align*}
\bm\cS^{*\bm\fP}_{s,t} \bm u \le \bm\cS^{(\fp,\fp^2,\dots,\fp^N)}_{s,t} \bm u,\;\fp\in\bm\Pi \quad\text{and}\quad \overline\cS^*_{s,t,\Xi}v \le \overline\cS^{\tilde\fp}_{s,t,\Xi}v,\;\tilde\fp\in\wt{\bm\Pi}.
\end{align*}
Similar results holds for $\bm\bS^{*\bm\fP}_t$ and $\overline\bS^*_{t,\Xi}$ as well. Moreover, there exist $\fp^*\in\Pi^{T-1}$ and $\overline\fp^*\in\overline{\Pi}^{T-1}$ that attain the optimality above (for fixed $\bm\fP, \bm u$ and $\Xi, v$), that is,
\begin{gather*}
\bm\cS^{*\bm\fP}_{s,t} \bm u=\bm\cS^{(\fp^*,\fp^2,\dots,\fp^N)}_{s,t} \bm u,\quad \bm\bS^{*\bm\fP}_{t} \bm u = \bm\bS^{(\fp^*,\fp^2,\dots,\fp^N)}_{t} \bm u,\\
\overline\cS^*_{s,t,\Xi} v = \overline\cS^{\overline\fp^*}_{s,t,\Xi} v,\quad \overline\bS^*_{t,\Xi} v = \overline\bS^{\overline\fp^*}_{t,\Xi} v.
\end{gather*}

For the rest of this paper, functions from $ B_b(\bX)$, when acted upon by $\bm G^{\bm\lambda}_t$ and related operators, are treated as functions from $B_b(\bX^N)$ that are constant in $x^2,\dots,x^N$.

\subsection{State-action mean field flow}\label{subsec:MFF}
It can be beneficial to describe a scenario in MFG using a flow of state-action distributions. This section is dedicated to explaining this concept. Recall that $\xi^{\psi}$ denotes the marginal distribution of $\psi$ on $\bX$, where $\psi\in\cP(\bX\times\bA)$.

Let $\Psi=(\psi_{1},\dots,\psi_{T-1})\in\cP(\bX\times\bA)^{T-1}$. We say $\Psi$ is a \textit{mean field flow (on $\bX\times\bA$ with initial marginal state distribution $\mathring{\xi}$)}, if $\Psi$ satisfies 
\begin{align}\label{eq:DefMFF}
\xi^{\psi_1}=\mathring{\xi}\quad\text{and}\quad \xi^{\psi_{t+1}}(B) = \int_{\bX\times\bA} P_{t,y,\xi^{\psi_t},a}(B) \psi_t(\dif y\dif a),\; B\in\cB(\bX),\, t=1,\dots,T-2.
\end{align}
For clarification, we point out that, based on the heuristics of the law of large numbers, all of the infinitely many independent players employing the same (randomized) policy will result in a deterministic flow of state-action distributions. This flow should comply with the evolution equation \eqref{eq:DefMFF} imposed on the state marginal. Moreover, these distributions summarize the movements and behaviors of the players and can be viewed as a scenario in MFGs, possibly not at equilibrium.

The policy of the representative player can be characterized by a mean field flow via conditioning. To elucidate this procedure, we first revisit a pertinent notion in current context, that of the regular conditional distribution. Let $\psi\in\cP(\bX\times\bA)$ and consider $Y(x,a):=a$. By \cite[Section 4.9, Theorem 4.44]{Aliprantis2006book}, we have $\cB(\bX\times\bA)=\cB(\bX)\otimes\cB(\bA)$. Additionally, $\cB(\bX)\otimes\set{\emptyset,\bA}$ is a sub-$\sigma$-algebra of $\cB(\bX)\otimes\cB(\bA)$. By \cite[Theorem 10.4.8, Example
10.4.9 and Example 6.5.2]{Bogachev2007book}, under $\psi$, there exists a $\varpi^\psi:(\bX\times\bA,\cB(\bX)\otimes\set{\emptyset,\bA})\to(\cP(\bA),\cE(\cP(\bA)))$ such that
\begin{align*}
\psi({B}\cap\set{Y \in A}) = \int_{\bX\times\bA} \1_{{B}}(y,a) \varpi^\psi_{y,a}(A) \psi(\dif y \dif a),\quad B\in\cB(\bX)\otimes\set{\emptyset,\bA},\; A\in\cB(\bA),
\end{align*}
i.e., $\varpi^\psi$ is a regular conditional distribution of $Y$ given $\cB(\bX)\otimes\set{\emptyset,\bA}$. By definition, for any $A\in\cB(\bA)$, $\varpi^\psi_{y,a}(A)$ is constant in $a\in\bA$. Therefore, we let $\overline\pi^\psi$ be $\varpi^\psi$ restrained on $\bX$ and call $\overline\pi^\psi$ the \textit{action kernel induced by $\psi$}. It follows that $\overline\pi^\psi$ is $\cB(\bX)$-$\cE(\cP(\bA))$ measurable, i.e.  $\overline\pi^{\psi}\in\overline\Pi$. Furthermore, we have 
\begin{align}\label{eq:IndInthpsi}
\psi(B\times A) = \int_{\bX} \1_{B}(y) \int_{\bA} \1_{A}(a)\overline\pi^\psi_{y}(\dif a) \,\xi^\psi(\dif y), \quad B\in\cB(\bX),\;A\in\cB(\bA).
\end{align}
Below, we present a result asserting that \eqref{eq:IndInthpsi} uniquely characterizes the induced action kernel. The proof of this statement can be found in \cref{subsec:Prooflem:pipsiUnique}.
\begin{lemma}\label{lem:pipsiUnique}
If \eqref{eq:IndInthpsi} holds with $\overline\pi^\psi$ replaced by $\hat\pi\in\overline\Pi$, then $\hat\pi(x)=\overline\pi^\psi(x)$ for $\xi^{\psi}$-almost every $x$.
\end{lemma}

Subsequently, we denote $\overline\fp^\Psi=(\overline\pi^{\psi_1},\dots,\overline\pi^{\psi_{T-1}})$ and refer to this as the \textit{policy induced by $\Psi$}. Additionally, recall that $\Xi^\Psi = (\xi^{\psi_1}, \dots, \xi^{\psi_{T-1}})$. Below we establish an equivalence between the mean field flow, $\Psi$, and the induced policy, $\overline\fp^\Psi$, using the transition operators defined in \cref{subsec:TransOp}. Furthermore, this result confirms that, regardless of the version of the induced policy, it consistently reproduces the mean field flow from which it is derived. The proof is deferred to \cref{subsec:Prooflem:MFF}.
\begin{lemma}\label{lem:MFF}
Let $\Psi$ be a mean field flow. Then, for any $t=1,\dots,T$, $h\in B_b(\bX)$, and any version of $\overline\fp^\Psi$, we have $\overline\bT^{\overline\fp^\Psi}_{t,\Xi^\Psi}h = \int_{\bX} h(y)\xi^{\psi_t}(\dif y)$, or equivalently, $\overline\bQ^{\overline\fp^\Psi}_{t,\Xi^\Psi}=\xi^{\psi_t}$, where $\overline\bQ^{\tilde\fp}_{t,\Xi}$ is defined below \eqref{eq:DefMeanfT}.
\end{lemma}

\subsection{An $N$-player scenario and its mean field counterpart}\label{subsec:Scenarios}

For remainder of the paper, we will fix the collection of policies of all players, $\wt{\bm\fP}\in\wt{\Pi}^{N\times(T-1)}$, in the $N$pG, dubbed the $N$-player scenario. 

To construct the corresponding scenario in MFG, we first introduce the auxiliary terms $\overline\Xi=(\overline\xi_1,\dots,\overline\xi_{T-1})\in\cP(\bX)^{T-1}$ and $\overline\xi_T\in\cP(\bX)$ by defining
\begin{align}\label{eq:Defxibar}
\overline\xi_1 := \mathring{\xi},\quad\text{and}\quad \overline\xi_{t} := \frac1{N}\sum_{n=1}^N\overline{\bQ}^{\tilde{\fp}^n}_{t,\overline\Xi},\; t=2,\dots,T,
\end{align} 
where we recall that $\overline\bQ^{\tilde\fp}_{t,\Xi}$ is introduced below \eqref{eq:DefMeanfT}, and we note that $\overline{\bQ}^{\tilde{\fp}^n}_{t,\overline\Xi}$ depends on $\overline\Xi$ only through $(\overline\xi_1,\cdots,\overline\xi_{t-1})$. We then specify the mean field scenario corresponding to $\wt{\bm\fP}$ as a $\overline\Psi=(\overline\psi_1,\dots,\overline\psi_{T-1})\in\cP(\bX\times\bA)^{T-1}$ satisfying\footnote{By a version of Carath\'eodory extension theorem (cf. \cite[Section 10.4, Theorem 10.23]{Aliprantis2006book}), $\overline\psi_t$ exists and is specified uniquely. }
\begin{align}\label{eq:DefPsiBar}
\overline{\psi}_{t}(B\times A) = \frac1N\sum_{n=1}^N\int_{\bX} \int_{\bA} \1_B(y)\1_A(a) \tilde\fp^n_{t,y,\overline\xi_{t}}(\dif a) \;\overline{\bQ}^{\tilde{\fp}^n}_{t,\overline\Xi} (\dif y),\quad B\in\cB(\bX),\,A\in\cB(\bA).
\end{align}

The lemma below shows that $\overline\Psi$ is a bona fide mean field flow. The proof is deferred to \cref{subsec:Prooflem:MeanPsiMFF}.
\begin{lemma}\label{lem:MeanPsiMFF}
With the notations above, we have $\xi^{\overline\psi_t}=\overline\xi_t$ for $t=1,\dots,T-1$, and
\begin{align*}
\overline\xi_T(B)=\int_{\bX\times\bA}  P_{T-1,x,\xi^{\overline\psi_{T-1}},a}(B)\overline\psi_{T-1}(\dif x \dif a), \quad B\in\cB(\bX).
\end{align*} 
Moreover, $\overline\Psi$ is a mean field flow.
\end{lemma}

\subsection{Technical assumptions}\label{subsec:Assumptions}
In this section, we consolidate the technical assumptions that are utilized in the forthcoming study. We note that not all of these assumptions are applied concurrently. Instead, they are invoked individually as and when required.

The first assumption regards the compactness of the action domain.
\begin{assumption}\label{assump:ActionDomain}
$\bA$ is compact. 
\end{assumption}
Under \cref{assump:ActionDomain}, we also have that $\cP(\bA)$ is also compact under weak topology due to Prokhorov's theorem (see \cite[Section 15.3, 15.11]{Aliprantis2006book}).

The second assumption regards the uniform tightness of $\overline\bQ^{\tilde\fp}_{t,\Xi}$, where we recall $\overline\bQ^{\tilde\fp}_{t,\Xi}$ is defined below \cref{eq:DefMeanfT}.
\begin{assumption}\label{assump:QTight}
There is a sequence of increasing compact subsets of $\bX$, denoted by $\fK=(K_i)_{i\in\bN}$, such that $(\overline\bQ^{\tilde\fp}_{t,\Xi})_{(t,\Xi,\tilde\fp)\in\set{1,\dots,T}\times\cP(\bX)^{T-1}\times\wt\Pi^{T-1}}$ is uniformly tight with respect to $\fK$ in the following sense
\begin{align*}
\overline\bQ^{\tilde\fp}_{t,\Xi}(K_i^c) \le i^{-1},\quad (t,\Xi,\tilde\fp)\in\set{1,\dots,T}\times\cP(\bX)^{T-1}\times\wt\Pi^{T-1},\quad i\in\bN.
\end{align*}
\end{assumption}
Clearly, \cref{assump:QTight} holds if $\bX$ is compact. When $\bX$ is non-compact, the moment condition may serve as a sufficient condition; see, e.g., \cite[Appendix E]{Hernandez-Lerma1996book}. We refer to \cref{assump:QTightsigma} and \cref{lem:sigmaTight} for additional discussion related to moment functions.

Below, we outline several technical conditions for $\mathring{\bm\bG}$, $\mathring{\overline\bG}$, $\bm G^{\bm\lambda}_t$, and $\overline G^{\lambda}_{t,\xi}$. In particular, \cref{assump:GBasic} supports the abstract formulations introduced in \cref{subsec:ScoreOp}. We refer to \cref{subsec:ExmpG} for related examples.
\begin{assumption}\label{assump:GBasic}
The following is true for any $t\in\set{0,1,\dots,T-1}$, $\bm\lambda\in\cP(\bA)^N$, $\lambda,\lambda'\in\cP(\bA)$, $\xi\in\cP(\bX)$, $\bm u,{\bm u}'\in B_b(\bX^N)$, and $v,v'\in B_b(\bX)$:
\begin{itemize}
\item[(i)]  $\mathring{\bm\bG}$, $\mathring{\overline\bG}$, $\bm G^{\bm\lambda}_t$, and $\overline G^{\lambda}_{t,\xi}$ are non-decreasing (i.e., $\bm G^{\bm\lambda}_{t} \bm u \ge \bm G^{\bm\lambda}_{t} \bm u'$ if $\bm u\ge \bm u'$).
\item[(ii)] There are constants $c_0,c_1>0$ such that 
\begin{align*}
\left\|\bm G^{\bm\lambda}_{t} \bm u\right\|_\infty \le c_0 + c_1\|\bm u\|_\infty, \quad \left\|\overline G^\lambda_{t,\xi} v\right\|_\infty\le c_0 + c_1\|v\|_\infty.
\end{align*}
\item[(iii)] There is a constant $\bar{c}>0$ such that, for any $\bm x\in\bX^N$,
\begin{gather*}
\left|\mathring{\bm G} \bm u - \mathring{\bm G} \bm u'\right| \le \bar{c}\int_{\bX^N} \left|\bm u(\bm y)-\bm u'(\bm y)\right| \mathring{\xi}^{\otimes N}(\dif\bm y),\\
\left|\bm G^{\bm\lambda}_t \bm u(\bm x) - \bm G^{\bm\lambda}_t \bm u'(\bm x)\right| \le \bar{c}\int_{\bX^N} \left|\bm u(\bm y)-\bm u'(\bm y)\right| \left[\bigotimes_{n=1}^N Q^{\lambda^n}_{t,x^n,\overline\delta_{\bm x}}\right](\dif\bm y),
\end{gather*}
and, for any $x\in\bX$, 
\begin{gather*}
\left|\mathring{\overline G} v - \mathring{\overline G} v'\right| \le \bar{c}\int_{\bX} \left|v(y)-v'(y)\right| \mathring{\xi}(\dif y),\\
\left|\overline G^{\lambda}_{t,\xi} v(x) - \overline G^{\lambda}_{t,\xi} v'(x)\right| \le \bar{c}\int_{\bX} \left|v(y)-v'(y)\right| Q^{\lambda}_{t,x,\xi}(\dif y).
\end{gather*}
\item[(iv)] If $\bm u$ satisfies $\bm u(\bm x)=\bm u(\phi \bm x)$ for any $\phi$ that permutes $z^2,\dots,z^N$ in $(z^1,z^2,\dots,z^N)$, then for any such permutation $\phi$ we have
\begin{align*}
\mathring{\bm G} u = \mathring{\bm G} u(\phi\,\cdot),\quad \bm G^{\bm\lambda}_{t}\bm u(\bm x) = \bm G^{\phi\bm\lambda}_{t} \bm u(\phi\bm x).
\end{align*} 
\item[(v)] If there is $v$ such that $u(\bm x) = v(x^1)$ for any $\bm x\in\bX^N$, then
\begin{align*}
\mathring{\bm G} \bm u = \mathring{\overline G} v,\quad \bm G^{\bm\lambda}_{t}\bm u(\bm x) = \overline G^{\lambda^1}_{t,\overline{\delta}_{\bm x}} v(x^1).
\end{align*}
\item[(vi)] For any $\gamma\in(0,1)$, we have 
\begin{align*}
\overline G^{\gamma \lambda + (1-\gamma) \lambda'}_{t,\xi}v \le \gamma \overline G^{ \lambda}_{t,\xi}v + (1-\gamma) \overline G^{ \lambda'}_{t,\xi}v.
\end{align*}
\end{itemize}
\end{assumption}

\cref{assump:GBasic} (i) and (ii) are common for operators involved in performance evaluation. In particular, \cref{assump:GBasic} (ii) imposes boundedness for simplicity. \cref{assump:GBasic} (iii) - (vi) are pivotal for our approximation of $N$pGs using MFGs. \cref{assump:GBasic} (iii) allows us to control the propagation of errors in the value functions when approximating $N$pGs; we refer to \cref{lem:EstDiffS} for further discussion. The symmetries in \cref{assump:GBasic} (iv) and (v) facilitate mean field approximations, although one might consider using approximate equality to relax these conditions. \cref{assump:GBasic} (vi) is designed to encourage randomized action. Reversing this inequality could cause our constructed mean field approximation, as done in \cref{subsec:HeurDerv} or \cref{subsec:Scenarios}, to miss the $N$pE. For more discussion, please refer to \cref{rmk:MFPathExploitability} and the example in \cref{subsec:ExmpConvex}. 

In addition, we would like to highlight the roles of \cref{assump:GBasic} (iii) and (vi) in the existence of MFEs. \cref{assump:GBasic} (iii) is crucial for proving the existence via the Kakutani–Fan–Glicksberg fixed point theorem (e.g., \cite[Section 17.9, Corollary 17.55]{Aliprantis2006book}). Together with the continuity conditions to introduced later, it facilitates the construction of a set-valued map with a closed graph, a requirement of the theorem; see \cref{lem:JointContGSS} and thereafter. \cref{assump:GBasic} (vi) is generally necessary for the existence. For an example supporting this claim, we again refer to \cref{subsec:ExmpConvex}. Lastly, we observe that \cref{assump:GBasic} (vi) permits costs that extend beyond the usual functional form $C(x,\xi,a)$, provided they also satisfy other conditions listed in this section. For example, with predefined $g_1,g_2:\bX\times\cP(\bX)\times\bA\to\bR$, we can incorporate a cost component of the form 
\begin{align}\label{eq:ExmpExtendedCost}
\lambda \mapsto \max_{g\in\{g_1,g_2\}} \int_{\bA} g(x,\xi,a)\lambda(\dif a),
\end{align}
as opposed to $\lambda\mapsto\int_\bA C(x,\xi,a)\lambda(\dif a)$ in the usual setting.

Following are two sets of assumptions concerning the continuities of the transition kernels and the score operators. These two sets of assumptions serve different results and can be considered independently. However, it is ideal when both sets of conditions are met.

The first set of assumptions pertains to the mean field approximation, as presented in \cref{sec:MFApprox}. Recall that the total-variation norm of a finite measure $m$ on $\cB(\bX)$ is defined as
\begin{align*}
\|m\|_{TV} := \sup_{h\in B_b(\bX),\|h\|_\infty\le1}\frac12\int_{\bX}h(y)m(\dif y).
\end{align*}
Accordingly, the total-variation distance between two finite measures is defined as 
\begin{align*}
d_{TV}(m,m') := \|m-m'\|_{TV} = \sup_{h\in B_b(\bX),\|h\|_\infty\le 1}\frac12\left|\int_{\bX}h(y)m(\dif y)-\int_{\bX}h(y)m'(\dif y)\right|.
\end{align*}
\begin{assumption}\label{assump:PCont}
There is a subadditive modulus of continuity $\eta$ such that 
\begin{align*}
\big\|P_t(x,\xi,a) - P_t(x,\xi',a')\big\|_{TV} \le \eta_t\big(d_{\bA}(a,a')\big) + \eta_t\big(\big\|\xi-\xi'\big\|_{BL}\big)
\end{align*}
for any $t=1,\dots,T-1$, $x\in\bX$, $\xi,\xi'\in\cP(\bX)$ and $a,a'\in\bA$.
\end{assumption}
\begin{assumption}\label{assump:GCont}
The following is true for any $t=1,\dots,T-1$, $N\in\bN$, $\bm\lambda,\bm\lambda'\in\cP(\bA)^N$, $\bm u\in B_b(\bX^N)$, $\lambda,\lambda'\in\cP(\bA)$, $\xi,\xi'\in\cP(\bX)$, and $v\in B_b(\bX)$:\footnote{Without loss of generality, we employ the same $c_0, c_1$ from \cref{assump:GBasic}, and maintain $\zeta_t$ for both $N$-play and mean field settings.}
\begin{itemize}
\item[(i)] There are constants $c_0,c_1>0$ and a subadditive modulus of continuity $\zeta$ such that 
\begin{align*}
\left\|\bm G^{\bm\lambda}_t \bm u - \bm G^{\bm\lambda'}_t \bm u\right\|_\infty \le \left(c_0 + c_1\|\bm u\|_\infty\right)\;\zeta(\|\lambda^1-{\lambda'}^1\|_{BL}).
\end{align*} 
\item[(ii)] There are constants $c_0,c_1>0$ and a subadditive modulus of continuity $\zeta$ such that
\begin{align*}
\left\|\overline G^{\lambda}_{t,\xi}v - \overline G^{\lambda'}_{t,\xi'}v\right\|_\infty \le \left(c_0 + c_1\|v\|_\infty\right)\;\big(\zeta(\|\xi-\xi'\|_{BL}) + \zeta(\|\lambda-\lambda'\|_{BL}) \big).
\end{align*}
\end{itemize}
\end{assumption}

Recall the definitions of $\wt\Pi_\vartheta$ and a symmetrically continuous policy from below \cref{eq:NplayerObj}. We say $\wt{\bm\fP}$ is \textit{$\vartheta$-symmetrically continuous} if it belongs to $\wt{\Pi}^{N\times(T-1)}_\vartheta$. The analogous definition also applies to $V\in B_b(\bX\times\cP(\bX))$. 

\begin{assumption}\label{assump:SymCont}
Let $\vartheta,\iota$ be subaddtive modulus of continuity. The $N$-player scenario $\wt{\bm\fP}$ is $\vartheta$-symmetrically continuous. All players are subject to the same terminal cost $\bm U\in B_b(\bX^N)$. Moreover, there is an $\iota$-symmetrically continuous $V\in B_b(\bX\times\cP(\bX))$ such that $\bm U(\bm x)=V(x^1,\overline\delta_{\bm x})$ for all $\bm x\in\bX^N$.
\end{assumption}

It's important to note that \cref{assump:PCont} does not enforce joint weak continuity. Instead, it imposes strong continuity with respect to $(\xi,a)$. Strong continuities are frequently assumed in games where approximations are concerned (cf. \cite[Theorem 2]{Jaskiewicz2020Non}).\footnote{These approximations, in a broad sense, can include cases such as using finite horizon games to approximate infinite horizon games.} We acknowledge that alternative sets of continuity assumptions may be applicable. For instance, in \cite[Theorem 4.1]{Saldi2018Markov}, which parallels our \cref{thm:HeurNoWaste} or \cref{thm:MFConstr}, \cref{assump:PCont} is imposed, but solely with respect to $\xi$. However, they also assume joint weak continuity of $P_t$ and the continuity of the representative player's policy with respect to the state. Determining which continuity assumptions are most appropriate, or whether it is feasible to relax these conditions, may necessitate additional modeling details that are beyond the scope of this study.

\cref{assump:GCont} plays a role similar to \cref{assump:PCont}. In certain settings, \cref{assump:GCont} can be derived as a result of \cref{assump:PCont}. For a concrete example, we refer to \cref{subsec:ExmpG}. In order to emphasize the importance of $\vartheta$ in \cref{assump:SymCont}, we refer to \cref{subsec:ExmpGame} for an example illustrating how an overly rough $\vartheta$ can lead to a vacuous approximation.

Lastly, we repeat that the symmetrically continuous policies in \cref{assump:SymCont} is reasonable in some realistic considerations. For example, when players have limited accuracy or certainty in perceiving the empirical population distribution, they might become reluctant to drastically alter their behavior in response to changes in the population. This concern could be represented by the aforementioned symmetric continuity. 

The following constitutes the second set of assumptions. These assumptions are pivotal when establishing the existence of a MFE through Kakutani–Fan–Glicksberg fixed point theorem (e.g., \cite[Section 17.9, Corollary 17.55]{Aliprantis2006book}). \cref{assump:QTightsigma}, as adopted from \cite[Assumption 1 (c)]{Saldi2018Markov}, implies \cref{assump:QTight}; see \cref{lem:sigmaTight}. Clearly, this assumption holds if $\bX$ is compact.\footnote{By convention, $\inf\emptyset=\infty$} In addition, we refer to \cref{subsec:ExmpQTightSigma} for another example with $\bX=\bR$. \cref{assump:QTightsigma} is crucial in the construction of a set-valued mapping that maps into its own domain, and this mapping will be ultimately used to construct a MFE. We note that \cref{assump:QTightsigma} is also used in \cite{Saldi2020Approximate} for the existence of a MFE. \cref{assump:PCont2} and \cref{assump:GCont2} impose joint continuity on the model's components and are crucial in ensuring that the aforementioned mapping has a closed graph, a condition required by the fixed point theorem.  In \cref{subsec:ExmpG}, we provide an example where \cref{assump:PCont2} under a suitable setting implies \cref{assump:GCont2}.

\begin{assumption}\label{assump:QTightsigma}
There are an increasing sequence of compact subsets of $\bX$, denoted by $\check\fK=(\check K_i)_{i\in\bN}$, a constant $\check c \ge 1$, and a continuous $\sigma:\bX\to\bR$ such that $\inf_{x\in \check K_i^c}\sigma(x) \ge i$ for $i\in\bN$, and
\begin{gather}
\int_{\bX}\sigma(y)\mathring{\xi}(\dif y) \le \check c,\label{eq:MomentCondInit}\\
\int_{\bX}\sigma(y)P_{t,x,\xi,a}(\dif y) \le \check c\sigma(x),\quad (t,x,\xi,a)\in\set{1,\dots,T-1}\times\bX\times\cP(\bX)\times\bA.\label{eq:MomentCond}
\end{gather}
\end{assumption}

\begin{assumption}\label{assump:PCont2}
$(x,\xi,a)\mapsto P_{t,x,\xi,a}$ is weakly continuous in $\bX\times\cP(\bX)\times\bA$.
\end{assumption}

\begin{assumption}\label{assump:GCont2}
 $(x,\xi,\lambda)\mapsto \overline G^{\lambda}_{t,\xi}v(x)$ is continuous in $\bX\times\cP(\bX)\times\cP(\bA)$ for any $t=1,\dots,T-1$ and $v\in C_b(\bX)$.
\end{assumption}

\subsection{Empirical measures under independent sampling}\label{subsec:EmpMeasConc}
In this section, we present a technical lemma, \cref{lem:EmpMeasConc}, concerning the convergence of the empirical measure derived from independent sampling. It's important to note that \cref{lem:EmpMeasConc} does not assume identical distribution. \cref{lem:EmpMeasConc} is inherently required by one of our main results, \cref{thm:MFApprox} (see also \cref{thm:HeurNoMiss}),  which pertains to capturing approximate $N$pEs with approximate MFEs. This necessity arises from our setup where, in the $N$pG, players are allowed to use different policies.

Let $\bY$ be a complete separable metric space. For $j\in\bN$ and $A\subseteq\bY$ we define
\begin{align*}
\fN_j(A) := \min\left\{n\in\bN: C_{BL}(A)\subseteq\bigcup_{i=1}^n B_{j^{-1}}(h_i) \text{ for some } h_1,\dots,h_n\in C_{BL}(A)\right\},
\end{align*} 
i.e., $\fN_j(A)$ is the smallest number of $j^{-1}$-open balls needed to cover $C_{BL}(A)$. If $A$ is compact, by Ascoli-Arzel\'a theorem (cf. \cite[ Chapter 7, Corollary 41]{Royden1988book}) and the fact that compact metric space is totally bounded (cf. \cite[Chapter 7, Proposition 19]{Royden1988book}), $\fN_j(A)$ is finite for $j\in\bN$. 

Let $N\in\bN$ be fixed. Suppose $(Y^n)_{n\in\bN}$ are independent $\bY$-valued random variables and let $\upsilon^n$ be the law of $Y^n$. We denote $\bm Y:=(Y^n)_{n=1}^N$. Consider the empirical measure $\overline\delta_{\bm Y}$ and the corresponding intensity measure $\overline\upsilon := \frac1N\sum_{n=1}^N\upsilon^n$. Note that $\left\|\overline\delta_{\bm Y}-\overline\upsilon\right\|_{BL}$ is $\sA$-$\cB(\bR)$ measurable because $\|\cdot\|_{BL}$ is continuous.  \cref{lem:EmpMeasConc} below provides an upper bound for the expectation of the bounded-Lipschitz distance between the empirical measure under independent sampling and the corresponding intensity measure. The proof of \cref{lem:EmpMeasConc}, which is deferred to \cref{subsec:Prooflem:EmpMeasConc}, primarily involves applying Hoeffding's inequality following a discretization enabled by the assumed uniform tightness. \cref{lem:EmpMeasConc} embodies the concept of a law of large numbers without identical distribution (cf. \cite[Theorem IV.3.2]{Shiryaev2019book}). The lack of a citation in \cref{lem:EmpMeasConc} should not be interpreted as a claim of originality, but rather as an indication that this is a commonly expected result (if not sharper).
\begin{lemma}\label{lem:EmpMeasConc}
Let $(\upsilon^n)_{n=1}^N$, $\overline\upsilon$, and $\bm Y$ be as introduced earlier. Suppose there is an increasing sequence of compact subsets of $\bY$, denoted by $\fA:=(A_i)_{i\in\bN}$, such that $\sup_{n}\upsilon_n(A_i^c)\le i^{-1}$ for all $i\in\bN$, i.e., $(\upsilon^n)_{n\in\bN}$ is uniformly tight with respect to $\fA$. Then, 
\begin{align}\label{eq:EmpMeasConc}
\bE\left(\left\|\overline\delta_{\bm Y}-\overline\upsilon\right\|_{BL}\right) \le \inf_{i,j\in\bN} \left\{ \frac12j^{-1} + i^{-1} + \frac{\sqrt{\pi}\;\fN_{j}(A_i)}{\sqrt{2 N}} \right\} =: \fr_{\fA}(N).
\end{align}
Moreover, we have $\lim_{N\to\infty}\fr_{\fA}(N)=0$ as $N\to\infty$.
\end{lemma}
The convergence rate established above leaves significant room for improvement. It is important to note that the forthcoming main results, while involving $\fr_\fA$, are independent of the specific procedure used to derive $\fr_\fA$. Therefore, any improvement in the convergence rate can directly enhance these results simply by substituting with the improved rate. A better rate will be pursuit elsewhere.

Results akin to \cref{lem:EmpMeasConc}  can be found in \cite{Dudley1969Speed,Fournier2015Rate,Lei2020Convergence,Kloeckner2020Empirical} and the reference therein. These results provides sharper estimations but requires identical distribution and more specific spatial structure. 

\subsection{Error terms}\label{subsec:ErrorTerms}
In this section, we will introduce several error terms, primarily for the sake of notational convenience. 

Recall that $\overline\delta_{\bm x}=\frac1N\sum_{n=1}^N \delta_{x^n}$ for $\bm x=(x^1,\dots,x^N)\in\bX^N$. Let $\overline\xi_t$ and $\overline\psi_t$ be as introduced in \cref{subsec:Scenarios}. we define
\begin{align}\label{eq:DefEmpErr}
{\bm e}_t({\bm x}):=\left\|\overline\delta_{\bm x} - \overline\xi_t\right\|_{BL},\quad t=1,\dots,T.
\end{align}
We additionally define
\begin{align}\label{eq:DefStateActionEmpErr}
\breve {\bm e}_t({\bm x}):=\int_{\bA^N}\left\|\overline\delta_{((x^1,a^1),\dots,(x^N,a^N))}-{\overline\psi}_t\right\|_{BL} \left[\bigotimes_{n=1}^N{\tilde\fp^n_{t,x^n,\overline\delta_{\bm x}}}\right](\dif {\bm a}), \quad t=1,\dots,T-1.
\end{align}
It follows from definition that $\bm\bT^{\wt{\bm\fP}}_t{\bm e}_t$ calculates, in the $N$-player scenario $\wt{\bm\fP}$ at time $t$, the expectation of bounded-Lipschitz distance between the empirical state distribution and $\overline\xi_t$. Similarly, $\bm\bT^{\wt{\bm\fP}}_t\breve{\bm e}_t$ calculates the expectation of bounded-Lipschitz distance between the empirical state-action distribution and $\overline\psi_t$. 

Below we introduce a few more constants
\begin{gather*}
\fe_t\,, \quad \breve\fe_t\,, \quad \fe^0_t\,, \quad t=1,\dots, T,\\
\underline \fE\,, \quad \fE\,, \quad \fE^0\,,\quad \fE^\diamond\,.
\end{gather*}
We defer the detailed definitions of these constants to \cref{subsec:DError}. These constants depend on the total number of players $N$, $\fK$ in \cref{assump:QTight}, $c_0,c_1,\bar{c}$ in \cref{assump:GBasic}, $\eta$ in \cref{assump:PCont}, $\zeta$ in \cref{assump:GCont}, $\vartheta,\iota,\|V\|_\infty$ in \cref{assump:SymCont}, and the convergence rate established in \cref{lem:EmpMeasConc}. It is important to note that values of these constants do not hinge on the particular choice of $\wt{\bm\fP}$.  

The forthcoming lemma offers a qualitative understanding of the aforementioned constants. Unlike the other parts of this paper, in this lemma, we explicitly incorporate the dependence on the total number of players, $N$, into our notations and allow $N$ to approach infinity. We refer to \cref{subsec:DError} for the proof.
\begin{lemma}\label{lem:ErrorConv}
Suppose \cref{assump:ActionDomain} and consider the aforementioned notations. Then, for any $t=1,\dots, T$, we have $\lim_{N\to\infty}{\fe_t}(N)=0$, $\lim_{N\to\infty}{\breve\fe_t}(N)=0$, and $\lim_{N\to\infty}\fe^0_t(N)=0$. Moreover, we have $\lim_{N\to\infty}\underline{\fE}(N)=0$, $\lim_{N\to\infty}{\fE}(N)=0$, $\lim_{N\to\infty}{\fE^0}(N)=0$, and $\lim_{N\to\infty}{\fE^\diamond}(N)=0$. 
\end{lemma}

\begin{remark}
\cref{lem:ErrorConv} implicitly requires that the regularities of the transition and score operators, as imposed in \cref{assump:QTight} - \cref{assump:SymCont}, are preserved as the size of the $N$pG increases.  
\end{remark}

\section{Exploitabilities in $N$pGs and MFGs}\label{sec:Exploitability} 
This section is devoted to the introduction of various notions of exploitabilities and the examination of their interrelationships. Exploitabilities for $N$pG are discussed in \cref{subsec:NplayerExploitability}, while exploitabilities for MFGs are the focus of \cref{subsec:MFexploitability}.

\subsection{Exploitabilities in $N$pGs}\label{subsec:NplayerExploitability}
In light of the setting in \cref{subsec:Scenarios} and \cref{assump:SymCont}, from the player-$1$'s perspective, a natural definition of exploitability emerges as follows. For any $N$-player scenario ${\bm\fP}=(\fp^1,\dots,\fp^N)\in\Pi^{N\times(T-1)}$ and terminal cost function ${\bm U}\in B_b(\bX^N)$. we define
\begin{align}\label{eq:DefSymContEndExploitability}
\bm\cR_{\vartheta}({\bm\fP};{\bm U}):= \bm\bS^{{\bm\fP}}_T {\bm U} - \inf_{\tilde\fp\in\wt{\Pi}_\vartheta}\bm\bS^{(\tilde\fp,\fp^2,\dots,\fp^N)}_{t} {\bm U}.
\end{align}
We note that, while $\bm\cR_{\vartheta}$ is defined for generic ${\bm\fP}$ and ${\bm U}$, our approximation result developed later is limited to ${\bm\fP}$ and ${\bm U}$ that satisfies \cref{assump:SymCont}. Let $\varphi^{1}$ be the identity permeation. For $n=2,\dots,N$, let $\varphi^{n}$ be a permutation such that $\varphi^{n}(z^1,\dots,z^N) = (z^n,z^1,\dots,z^{n-1},z^{n+1},\dots,z^N)$. Under \cref{assump:GBasic} (iv), definition \eqref{eq:DefSymContEndExploitability} also applies to player-$n$ via
\begin{align*}
\bm\cR^n_{\vartheta}({\bm\fP};{\bm U}) := \bm\cR_{\vartheta}(\varphi^{n}{\bm\fP};{\bm U}), \quad n=1,\dots,N.
\end{align*}

Although ${\bm\cR_{\vartheta}}$ is inherently suited to the scenario described in \cref{subsec:Scenarios}, the associated constrained optimization problem presents a significant challenge. To circumvent this, we employ an approximation approach. To aid in this process, we introduce the following auxiliary definitions of exploitability. Suppose \cref{assump:ActionDomain}, \cref{assump:GBasic} (i) (ii), and \cref{assump:GCont} for the validity of the upcoming definitions (see also \cref{lem:Sstar}). We define
\begin{gather}
\bm\cR({\bm\fP}; {\bm U}) := \bm\bS^{{\bm\fP}}_T {\bm U} - \bm\bS^{*{\bm\fP}}_{T} {\bm U},\label{eq:DefEndExploitability} 
\end{gather}
and
\begin{gather}
\bm\fR({\bm\fP}; {\bm U}) := \sum_{t=1}^{T-1} \bar{c}^t \bm\bT^{{\bm\fP}}_t \left(\bm S^{{\bm\fP}_t}_t  \circ \bm\cS^{*{\bm\fP}}_{t+1,T} {\bm U}  - \bm\cS^{*{\bm\fP}}_{t,T} {\bm U}\right),\label{eq:DefStepExploitability} 
\end{gather}
where we recall $\bar{c}$ from \cref{assump:GBasic} (iii). Similarly as before, with \cref{assump:GBasic} (iv), we further define
\begin{align}\label{eq:DefPermExploitabilities}
\bm\cR^n({\bm\fP}; {\bm U}) := \bm\cR(\varphi^{n}{\bm\fP}; {\bm U}\circ\varphi^n),\quad \bm\fR^n({\bm\fP};{\bm U}) := \bm\fR(\varphi^{n}{\bm\fP};{\bm U}\circ\varphi^n), \quad n=1,\dots,N.
\end{align}

Following the discussion in \cref{subsec:HeurDerv}, we refer to $\bm\cR$ and $\bm\fR$ as \textit{end exploitability} and \textit{total (stepwise) exploitability}, respectively. Accordingly, we refer to ${\bm\cR_{\vartheta}}$ as the \textit{end exploitability constrained to $\vartheta$-symmetrically continuous policy}, or simply, the constrained end exploitability. By definition (recall \cref{subsec:ScoreOp}), both end exploitability and total exploitability are non-negative. Technically, the constrained end exploitability could be negative if the individual player employs a policy that does not meet the continuity constrain. 

\begin{remark}
In this remark, we discuss a modeling issue of the end exploitability within a risk-averse framework, which ultimately prompts the exploration of total exploitability. For certain risk-averse performance criterion, such as \eqref{eq:DefExmpGInit} with $J=1$, $w_1=1$, and $\kappa_1=\kappa\in(0,1)$, when combined with specific environments, achieving $0$ end exploitability may not necessarily optimize sample paths that yield superior overall outcomes.  This is in fact an immediate consequence of the robust representation of average value at risk (see, e.g., \cite[(6.70)]{Shapiro2021book}). This phenomenon of under-optimization appears somewhat problematic, as it is typically anticipated that a player would continue to optimize, regardless of a strong initial start. To mitigate this issue, one approach is to modify the criterion by employing a blend with expectation to prevent the negligence of sample paths. Alternatively, considering total exploitability could be a solution, as it demands optimal actions across almost every sample paths to achieve a value of 0.
\end{remark}

Below we summarize the relations between different notions of $N$-player exploitabilities introduced above, the proof of which is deferred to \cref{subsec:Proofprop:EstEndExploitabilityCont} .
\begin{proposition}\label{prop:EstEndExploitabilityCont}
Suppose \cref{assump:ActionDomain}, \cref{assump:GBasic} (i) (ii), and \cref{assump:GCont} for $\bm\cR,\bm\fR,\bm\cR_\vartheta$ to be well-defined. The following is true:
\begin{itemize}
\item[(a)] If \cref{assump:GBasic} (iii) holds, then
\begin{align}\label{eq:EndvsStep}
\bm\cR({\bm\fP}; {\bm U})\le \fR({\bm\fP}; {\bm U}), \quad \bm\fP\in\Pi^{N\times(T-1)},\, \bm U\in B_b(\bX^N).
\end{align}
\item[(b)] Suppose \cref{assump:GBasic} (i). Additionally, assume the existence of a $\underline c>0$ such that, for any ${\bm\lambda}\in\cP(\bA)^N$ and ${\bm u}\ge{\bm u}'$,
\begin{gather*}
\mathring{\bm G}{\bm u} - \mathring{\bm G}{\bm u}' \ge \underline c\int_{\bX^N} \left({\bm u}({\bm y})-{\bm u}'({\bm y})\right) \mathring{\xi}^{\otimes N}(\dif{\bm y}),\\
\bm G^{{\bm\lambda}}_t{\bm u}({\bm x}) - \bm G^{{\bm\lambda}}_t{\bm u}'({\bm x}) \ge \underline c\int_{\bX^N} \left({\bm u}({\bm y})-{\bm u}'({\bm y})\right) \left[\bigotimes_{n=1}^N Q^{\lambda^n}_{t,x^n,\overline\delta_{\bm x}}\right](\dif{\bm y}),
\end{gather*}
then
\begin{align}\label{eq:EndvsStep2}
\bm\cR({\bm\fP}; {\bm U}) \ge \left({\underline c}\middle/{\bar{c}}\right)^{T-1}\bm\fR({\bm\fP}; {\bm U}).
\end{align}
\item[(c)] Suppose \cref{assump:QTight}, \cref{assump:GBasic} (iii) (v), \cref{assump:PCont}, \cref{assump:GCont} and \cref{assump:SymCont}. With $\underline\fE$ introduced in \cref{subsec:ErrorTerms}, we have
\begin{align*}
\left|{\bm\cR_{\vartheta}}(\wt{\bm\fP};{\bm U}) - \bm\cR(\wt{\bm\fP};{\bm U})\right| \le \underline\fE.
\end{align*}
\end{itemize}
\end{proposition}
For an example where $\mathring{\bm G}$ and $\bm G^{\bm\lambda}_t$ satisfy the assumptions in \cref{prop:EstEndExploitabilityCont} (b), we refer to the expressions in \eqref{eq:DefExmpGInit} and \eqref{eq:DefExmpG} with $w_1\in(0,1]$ and $\kappa_1=1$, i.e., we blend average value at risk's with the expectation. This is related to the spectral risk measure \cite{Acerbi2002Spectral}. In this case, we can set $\underline c=w_1$. For another example, we refer to the standard risk-neutral setting as used in \cref{subsec:HeurDerv}, where the expectation of total cost is used as performance criteria (i.e., $w_1=1$ and $\kappa_1=1$ in \eqref{eq:DefExmpGInit} and \eqref{eq:DefExmpG}). Consequently, we have $\underline c=\bar{c}=1$, and thus end exploitability and total exploitability coincide,
\begin{align}\label{eq:IndEndStepExploitability}
\bm\cR({\bm\fP}; {\bm U}) = \bm\fR({\bm\fP}; {\bm U}).
\end{align}

\subsection{Exploitabilities in mean field games}\label{subsec:MFexploitability}
We first recall the notations introduced in \cref{subsec:MFF}. In particular, $\Psi\in\cP(\bX\times\bA)^{T-1}$ and $\Xi^{\Psi}\in\cP(\bX)^{T-1}$ is the corresponding vector of state marginals. For $V\in B_b(\bX\times\cP(\bX))$, we write\footnote{For convenience, we slightly abuse the notation here. This should not be confused with $V_\xi$ introduced above \eqref{eq:MFObj}.}
\begin{align}\label{eq:DefVM}
V_\Psi(x):=V\left(x,\int_{\bX\times\bA} P_{t,y,\xi_t,a}(\cdot) \psi_{T-1}(\dif y\dif a)\right).
\end{align} 
We define the \textit{mean field end exploitability } as
\begin{align}\label{eq:DefMeanEndExploitability}
\overline{\cR}(\Psi;V) := \overline\bS^{{\overline\fp^\Psi}}_{T,\Xi^{\Psi}} V_{\Psi} - \overline\bS^{*}_{T,{\Xi^{\Psi}}} V_{\Psi}.
\end{align}
The \textit{mean field total (stepwise) exploitability} is 
\begin{align}\label{eq:DefMeanStepExploitability}
\overline\fR(\Psi;V) &:= \sum_{t=1}^{T-1} \bar{c}^t \overline\bT^{{\overline\fp^\Psi}}_{t,{\Xi^{\Psi}}} \left( \overline S^{{\overline\pi^{\psi_t}}}_{t,{\xi^{\psi_t}}} \circ \overline{\cS}^{*}_{t+1,T,{\Xi^{\Psi}}} V_{\Psi} - \overline{\cS}^{*}_{t,T,{\overline\Xi^{\Psi}}} V_{\Psi} \right)\nonumber\\
&= \sum_{t=1}^{T-1} \bar{c}^t \int_{\bX} \left(\overline S^{{\overline\pi^{\psi_t}}}_{t,{\xi^{\psi_t}}} \circ \overline{\cS}^{*}_{t+1,T,{\Xi^{\Psi}}} V_{\Psi}(y) - \overline{\cS}^{*}_{t,T,{\overline\Xi^{\Psi}}} V_{\Psi}(y)\right) {\xi^{\psi_t}}(\dif y),
\end{align}
where we have used \cref{lem:MFF} in the last equality. In view of \cref{assump:ActionDomain}, \cref{assump:GBasic} (i) (ii), and \cref{assump:GCont} (see also \cref{lem:Sstar}), both $\overline\fR(\Psi;V)$ and $\overline\cR(\Psi;V)$ are well-defined and non-negative. 
\begin{remark}\label{rmk:MFPathExploitability} 
The process of defining $\overline{\fR}$ entails deriving regular conditional distributions from specific joint distributions, which may be cumbersome in certain instances. In \cref{subsec:ExmpMeanStepExploitability}, we present an example illustrating how a more explicit form of $\overline{G}^{\lambda}_{t,\xi}$ can lead to a simplification of $\overline{\fR}$.
\end{remark}

The lemma below shows $\overline\cR(\Psi;V)$ and $\overline\fR(\Psi;V)$ are version-independent in ${\overline\fp^\Psi}$, the proof of which is deferred to \cref{subsec:Prooflem:VersionIndpi}.
\begin{lemma}\label{lem:VersionIndpi}
Suppose \cref{assump:GBasic} (iii). Consider a mean field system $\Psi$ and $v\in B_b(\bX)$.  Let ${\overline\fp^\Psi}$ and ${\overline\fp^\Psi}'$ be two version of policy induced by $\Psi$. Then,
\begin{align*}
\overline\bS^{{\overline\fp^\Psi}}_{T,{\Xi^{\Psi}}} v = \overline\bS^{{\overline\fp^\Psi}'}_{T,{\Xi^{\Psi}}} v,\quad \int_\bX\overline S^{{\overline\pi^{\psi_t}}}_{t,{\xi^{\psi_t}}} v(y) {\xi^{\psi_t}}(\dif y) = \int_\bX\overline S^{{\overline\pi^{\psi_t}}'}_{t,{\xi^{\psi_t}}} v(y) {\xi^{\psi_t}}(\dif y),\,t=1,\dots,T-1.
\end{align*}
\end{lemma}

Analogous to previous results for $N$pGs, we observe a similar relationship between mean field end exploitability and mean field total exploitability. The proof mirrors the argument used to prove \cref{prop:EstEndExploitabilityCont}, and is therefore omitted here.
\begin{proposition}\label{prop:EstMFEndExploitability}
Suppose \cref{assump:ActionDomain}, \cref{assump:GBasic} (i) (ii), and \cref{assump:GCont} for $\overline\cR$ and $\overline\fR$ to be well-defined. The following is true for any mean field flow $\Psi\in\cP(\bX\times\bA)^{T-1}$ and $V\in B_b(\bX\times\cP(\bA))$:
\begin{itemize}
\item[(a)] If \cref{assump:GBasic} (iii) holds, then 
\begin{align}\label{eq:MeanEndvsStep} 
\overline\cR(\Psi;V) \le \overline\fR(\Psi;V).
\end{align}
\item[(b)] Suppose \cref{assump:GBasic} (i). Additionally, assume the existence of a $\underline c>0$ such that, for any $\lambda\in\cP(\bA)$ and $v\ge{v'}$,
\begin{gather*}
\mathring{\overline G} v - \mathring{\overline G} {v'} \ge \underline c\int_{\bX^N} \left(v(y)-{v'}(y)\right) \mathring{\xi}(\dif y),\\
\overline G^{\lambda}_{t,\xi} v(x) - \overline G^{\lambda}_{t,\xi} {v'}(x) \ge \underline c\int_{\bX} \left(v(y)-{v'}(y)\right) Q^{\lambda}_{t,x,\xi}(u).
\end{gather*}
Then,
\begin{align*}
\overline\cR(\Psi;V) \ge \left({\underline c}\middle/{\bar{c}}\right)^{T-1}\overline\fR(\Psi;V).
\end{align*}
\end{itemize}
\end{proposition}
Similarly as before, for an example where $\mathring{\overline G}$ and $\overline G^\lambda_{t,\xi}$ satisfy the assumptions in \cref{prop:EstMFEndExploitability} (b), we refer to \eqref{eq:DefExmpMeanGInit} and \eqref{eq:DefExmpMeanG} with $w_1\in(0,1]$ and $\kappa_1=1$. In this case, we may set $\underline c=w_1$. 
Additionally, in the standard risk-neutral setting as in \cref{subsec:HeurDerv}, we have $\underline c=\bar{c}=1$, and thus
\begin{align}\label{eq:IndMFEndStepExploitability}
\overline\cR({\Psi}; V) = \overline\fR({\Psi}; V).
\end{align}

\section{Approximation capability of mean field games}\label{sec:MFApprox}

In the following theorem, we examine an $N$pG where all players employ symmetrically continuous policies. The theorem asserts that, given a sufficiently large $N$, the empirical state-action flow can be closely approximated by the corresponding mean field flow. Furthermore, the average of stepwise exploitabilities in the $N$pG dominates the stepwise exploitability of the corresponding MFG up to a minor error term. The proof of this theorem is deferred to \cref{subsec:Proofthm:MFApprox1} and \cref{subsec:Proofthm:MFApprox2}.
\begin{theorem}\label{thm:MFApprox}
Let $\wt{\bm\fP}$ and $\overline\Psi$ be as introduced in \cref{subsec:Scenarios}. Suppose \cref{assump:ActionDomain}, \cref{assump:QTight}, \cref{assump:GBasic}, \cref{assump:PCont}, \cref{assump:GCont} and \cref{assump:SymCont}. Then, with the error terms defined in \cref{subsec:ErrorTerms}, we have
\begin{align}\label{eq:GoodDescription}
\bm\bT^{\wt{\bm\fP}}_{t}\breve {\bm e}_t \le {\breve\fe_t} 
\end{align}
and 
\begin{align}\label{eq:NoMiss}
\overline\fR({\overline\Psi};V) \le \frac{1}{N}\sum_{n=1}^N{\bm\fR^n}(\wt{\bm\fP};{\bm U}) + {\fE}.
\end{align}
\end{theorem}

\begin{remark}\label{rmk:LinearityConcavity}
As seen in the proof of \cref{thm:MFApprox}, if we have equality in \cref{assump:GBasic} (vi),\footnote{This is true in, for example, the risk neutral setup. See also \cref{subsec:HeurDerv} and \cref{subsec:NplayerExploitability}.} we achieve equality instead of inequality in \eqref{eq:EstAveMeanfTS}. As a result, we obtain a stronger version of \cref{thm:MFApprox} that 
\begin{align}\label{eq:RNNpRMFR}
\left|\overline\fR({\overline\Psi};V) - \frac{1}{N}\sum_{n=1}^N{\bm\fR^n}(\wt{\bm\fP};{\bm U})\right| \le {\fE}.
\end{align}
On the other hand, if we assume concavity in \cref{assump:GBasic} (vi), by using in \eqref{eq:EstAveMeanfTS} Jensen's inequality for concave functions, we eventually yield
$$\frac{1}{N}\sum_{n=1}^N{\bm\fR^n}(\wt{\bm\fP};{\bm U}) \le \overline\fR({\overline\Psi};V) + {\fE}.$$ 
This suggests that, in the study of an $N$pG that encourages deterministic actions over randomized actions, approximation via $\overline\Psi$, as constructed in \cref{subsec:Scenarios}, may result in overlooking approximate equilibria in the $N$pG.
\end{remark}

\begin{remark}
We conjecture that \cref{thm:MFApprox} (b) remain trues if $\bm\fR^N$ and $\overline\fR$ are replaced by $\bm\cR^N$ and $\overline\cR$; whether additional assumption is needed remains unknown. Alternatively, we can utilize \cref{prop:EstEndExploitabilityCont} and \cref{prop:EstMFEndExploitability} to derive less sharp estimates.
\end{remark}

The following theorem asserts that any mean field flow with a small mean field scenario can be used to construct a homogeneous open-loop $N$-player scenario, where the exploitability of each player closely approximates the mean field exploitability. We note that \cref{assump:GBasic} (vi) is not required in this context. We refer to \cref{Proofthm:MFConstr} for the proof.
\begin{theorem}\label{thm:MFConstr}
Let $\Psi$ be a mean field flow as defined in \cref{subsec:MFF}. For $n=1,\dots,N$, let ${\overline\fp^\Psi}^n$ be a version of policy induced by $\Psi$. Consider the $N$-player scenario given by $\overline{\bm\fP}:=({\overline\fp^\Psi}^1,\dots,{\overline\fp^\Psi}^N)$. If \cref{assump:ActionDomain}, \cref{assump:QTight}, \cref{assump:GBasic} (i) - (v), \cref{assump:PCont}, \cref{assump:GCont}, and \cref{assump:SymCont} hold, then
$$\left|{\bm\fR^n}(\overline{\bm\fP};{\bm U}) - \overline\fR(\Psi;V)\right| \le{\fE^0}, \quad n=1,\dots,N,$$
and
$$\left|\bm{\cR^n}(\overline{\bm\fP};{\bm U}) - \overline\cR(\Psi;V)\right| \le {\fE^\diamond}, \quad n=1,\dots,N,$$
where $\fE^0$ and $\fE^\diamond$ are introduced in \cref{subsec:ErrorTerms}.
\end{theorem}

\section{Existence of mean field equilibrium}\label{sec:MFE}
Recall that, in our context, a mean field equilibrium is a mean field flow $\Psi$ with $\overline\cR(\Psi;V)=0$. In light of \eqref{eq:MeanEndvsStep},  \cref{thm:ZeroExploitability} below reveals the existence of a mean field equilibrium.  We refer to \cref{subsec:ProofMFEandOthers} for the detailed proof.
\begin{theorem}\label{thm:ZeroExploitability}
Let $V\in C_b(\bX\times\cP(\bX))$. Suppose \cref{assump:ActionDomain}, \cref{assump:GBasic} (i) (ii) (iii) (vi), \cref{assump:QTightsigma}, \cref{assump:PCont2} and \cref{assump:GCont2}. Then, there is a mean field flow $\Psi^*$ such that $\overline\fR(\Psi^*;V)=0$.
\end{theorem}

Below we would like to share some insights we gleaned from the proof of \cref{thm:ZeroExploitability}. These insights pertain to incorporating penalization for randomized action into mean field game theory. The proof of \cref{thm:ZeroExploitability} hinges significantly on \cref{assump:GBasic} (vi).  We note that \cref{assump:GBasic} (vi) encourages the use of randomized actions and is crucial for both the approximation and the existence of equilibrium in the current mean field game theory framework; see the discussion following \cref{assump:GBasic} (vi). On the other hand, the reluctance to embrace randomized actions is commonly observed in real world, prompting a pertinent discussion. In our proof of the existence of an MFE, \cref{assump:GBasic} (vi) is not utilized until the final step. This observation could illuminate potential pathways for incorporating penalization for randomized actions into mean field game theory. To delve deeper, we subsequently highlight an important intermediate result, \cref{prop:GammaFixedpt}, which does not require \cref{assump:GBasic} (vi).

Let $\cP(\bX\times\cP(\bA))$ be the set of probability measures on $\cB(\bX)\otimes\cB(\cP(\bA))$. To connect $\mu\in\cP(\bX\times\cP(\bA))$ with state-action distributions, which are fundamental in constructing mean field flows, we introduce a transformation that collapses the $\cP(\bA)$ component of $\mu$. Specifically, we define $\psi^\mu\in\cP(\bX\times\bA)$ by setting
\begin{align*}
\psi^\mu(B\times A) = \int_{\bX\times\cP(\bA)} \1_{B}(y) \lambda(A) \mu(\dif y\dif\lambda),\quad B\in\cB(\bX),\, A\in\cB(\bA).
\end{align*}
Heuristically, $\mu \in \cP(\bX \times \cP(\bA))$ represents the state-action-kernel distribution, in contrast to $\psi \in \cP(\bX \times \bA)$, which represents the state-action distribution. Note the $\mu$ can be viewed as a statistics that summaries the the state distribution of players, and the (conditional) distribution of action kernels (instead of actions) players at a given state employ. The process of defining $\psi^\mu$ entails averaging the action kernels of players who share the same state, in an expected sense.

In what follows, we let $\xi^\mu$ be the marginal distribution of $\mu$ on $\bX$. Additionally, for $\fm=(\mu_1,\dots,\mu_{T-1})\in\cP(\bX\times\cP(\bA))^{T-1}$, we denote $\Xi^\fm=(\xi^{\mu_1},\dots,\xi^{\mu_{T-1}})$ and $\Psi^\fm=(\psi^{\mu_1},\dots,\psi^{\mu_{T-1}})$. \cref{prop:GammaFixedpt} reveals the existence of an $\fm\in\cP(\bX\times\cP(\bA))^{T-1}$ that satisfies $\xi^{\mu_1}=\mathring\xi$, 
\begin{align}\label{eq:HeurIndGammaMarginal}
\mu_{t+1}(B) = \int_{\bX\times\cP(\bA)}\int_{\bA}P_{t,x,\xi^{\mu_t},a}(B)\lambda(\dif a)\mu_t(\dif x\dif\lambda),\quad B\in\cB(\bX),\, t=1,\dots,{T-2},
\end{align}
and 
\begin{align}\label{eq:HeurIndGammaOpt}
\int_{\bX\times\cP(\bA)} \left(\overline G^{\lambda}_{t,\xi^{\mu_t}} \circ \overline{\cS}^{*}_{t+1,T,\Xi^\fm} V_{\fm}(y) - \overline{\cS}^{*}_{t,T,\Xi^\fm} V_{\fm}(y)\right) \mu_t(\dif y \dif\lambda) = 0,
\end{align}
where, in light of the definition given in \eqref{eq:DefVM}, and with a slight abuse of notation, we denote $V_\fm:=V_{\Psi^\fm}$. This suggests that an `equilibrium', characterized by a flow of state-action-kernel distributions, can exist even when \cref{assump:GBasic} (vi) is not met. Note that, by \eqref{eq:HeurIndGammaMarginal}, $\Psi^\fm$ is a genuine mean field flow; see \eqref{eq:ImpliesMFF} for detailed deduction. Additionally, \eqref{eq:HeurIndGammaOpt} guarantees that $\mu_t$ only assign masses to optimal action kernels. However, \eqref{eq:HeurIndGammaOpt} does not necessarily ensure that $\Psi^\fm$ exhibits zero exploitability. Indeed, without \cref{assump:GBasic} (vi), averaging (optimal) action kernels could potentially degrade performance. We refer to \eqref{eq:ImpliesOpt} for the implied optimality when \cref{assump:GBasic} (vi) is assumed.

In conclusion, we posit that integrating penalization for randomized actions into mean field game theory is viable. However, we advise against adopting the state-action mean field flow notation, or its equivalent, the representative agent framework. A more suitable approach involves considering the broader concept, flows of state-action-kernel distributions. For a related example in discrete spaces, we refer to \cref{subsec:ExmpConvex}.

\section{Proofs}\label{sec:Proofs} 
In \cref{subsec:EmpMeasProc} and \cref{subsec:ApproxExploitability}, we initially present some auxiliary results that lay the groundwork for the proofs of \cref{prop:EstEndExploitabilityCont}, \cref{thm:MFApprox}, and \cref{thm:MFConstr}. Specifically, we first establish a technical result, \cref{prop:ProcConc}, based on \cref{lem:EmpMeasConc}, which characterizes the concentration of empirical state-action distributions at time $t$ in the $N$pG. Leveraging \cref{prop:ProcConc}, we then demonstrate in \cref{prop:EstDiffRMF} that total exploitabilities in the $N$pG can be approximated by quantities computed solely by mean field-type operators. \cref{prop:EstEndExploitabilityCont}, \cref{thm:MFApprox}, and \cref{thm:MFConstr} are subsequently proved in \cref{subsec:Proofprop:EstEndExploitabilityCont} through \cref{Proofthm:MFConstr}. Lastly, \cref{subsec:ProofMFEandOthers} is devoted to the proof of \cref{thm:ZeroExploitability}, which is done independently of the preceding results.


\subsection{Empirical measures in the $N$pG}\label{subsec:EmpMeasProc}

\begin{proposition}\label{prop:ProcConc}
Let $\wt{\bm\fP}$ and $\overline\Xi$ be as introduced in \cref{subsec:Scenarios}. Consider a ${\bm\fP}\in\Pi^{N\times(T-1)}$ that satisfies $\fp^n=\tilde\fp^n$ for $n=2,\dots,N$. If \cref{assump:QTight}, \cref{assump:PCont}, and \cref{assump:SymCont} hold, then
for any $t=1,\dots,T$ and subadditive modulus of continuity $\beta$, we have ${\bm\bT}^{{\bm\fP}}_{t} \circ \beta {\bm e}_t \le \beta({\fe_t})$. 
\end{proposition}

To facilitate proving \cref{prop:ProcConc}, we first introduce some auxiliary notations. For any $\tilde{\bm\pi}=(\tilde\pi^1,\dots,\tilde\pi^N)\in\wt\Pi^N$ and $\xi\in\cP(\bX)$, we define
\begin{gather}
\check{\bm T}^{\tilde{\bm\pi}}_{t,\xi} \bm u({\bm x}):= \int_{\bX^{N}} \bm u({\bm y}) \left[\bigotimes_{n=1}^N Q^{\tilde\pi^n_{x^n, \xi}}_{t,x^n,\xi}\right](\dif{\bm y}), \quad \bm u \in B_b(\bX).
\end{gather}
In contrary to \eqref{eq:NpDyn},  $\check{\bm T}^{\tilde{\bm\pi}}_{t,\xi}$ suggests the following auxiliary dynamics where the empirical population distribution in the $N$pG is replaced by a pre-specified $\xi$,
\begin{align}
\bm A_t\sim\bigotimes_{n=1}^N\tilde\pi^n_{t,X^n_t,\xi}\,,\quad \bm X_{t+1}\sim\bigotimes_{n=1}^N P_{t,X^n_t,\xi, A^n_t}.
\end{align}
With the notations introduced in \cref{subsec:Scenarios}, we further define
\begin{align*}\label{eq:DefMeancTN}
{\check{\bm\cT}}^{\wt{\bm\fP}}_{s,s,\overline\Xi}\bm u=\bm u \quad\text{and}\quad {\check{\bm\cT}}^{\wt{\bm\fP}}_{s,t,{\overline\Xi}} \bm u := {\check{\bm T}}^{\wt{\bm\fP}_s}_{s,\overline\xi_s}\circ\cdots\circ{\check{\bm T}}^{\wt{\bm\fP}_{t-1}}_{t-1,\overline\xi_{t-1}} \bm u,\; 1\le s<t \le T, \quad \bm u\in B_b(\bX^N),
\end{align*}
where we note the definition is also valid for generic $\wt{\bm\fP}\in\wt\Pi^{N\times(T-1)}$ and $\Xi\in\cP(\bX)^{T-1}$, although the generic case will not be considered. Observe that, if $\bm u({\bm x})=\prod_{n=1}^N h_n(x^n)$ for some $h_1,\dots,h_N\in B_b(\bX)$, we have 
\begin{align}\label{eq:DecompMeanT}
{\check{\bm\cT}}^{\wt{\bm\fP}}_{s,t,{\overline\Xi}}\bm u({\bm x}) 
= \prod_{n=1}^N\overline{\cT}^{\fp^n}_{s,t,{\overline\Xi}}h_n(x^n).
\end{align}
This in particular implies that, for fixed ${\bm x}\in\bX^N$, the probability measure on $\cB(\bX^{N})$ characterized by $D\mapsto{\check{\bm\cT}}^{{\bm\fP}}_{s,t,{\overline\Xi}}\1_D({\bm x})$ is the independent product of the probability measures on $\cB(\bX)$ characterized by $B\mapsto \overline{\cT}^{\fp^n}_{s,t,{\overline\Xi}}\1_B(x^n)$, $n=1,\dots,N$.

\begin{lemma}\label{lem:EstDiffMeancTInfMeancT}
Under the setting of \cref{prop:ProcConc}, we have
\begin{align*}
\left|{\check{\bm\cT}}^{\wt{\bm\fP}}_{s,t,{\overline\Xi}} {\bm e}_t({\bm x}) - \inf_{x^n\in\bX}{\check{\bm\cT}}^{\wt{\bm\fP}}_{s,t,{\overline\Xi}} {\bm e}_t({\bm x})\right| \le \frac1{N},\quad  {\bm x}\in\bX^N,\;n=1,\dots,N.
\end{align*}
\end{lemma}
\begin{proof}
The statement is an immediate consequence of the independence described in \eqref{eq:DecompMeanT}, \cref{lem:EstIndepInt} and the definition of ${\bm e}_t$ in \cref{subsec:ErrorTerms}. 
\end{proof}

\begin{lemma}\label{lem:EstDiffTMeanT}
Under the setting of \cref{prop:ProcConc}, for $s< t\le T$, we have
\begin{align*}
\left| \bm T^{(\fp^1_t,\tilde\fp^2_t,\dots,\tilde\fp^N_t)}_{s} \circ {\check{\bm\cT}}^{\wt{\bm\fP}}_{s+1,t,{\overline\Xi}} {\bm e}_{t}({\bm x}) - {\check{\bm\cT}}^{\wt{\bm\fP}}_{s,t,{\overline\Xi}} {\bm e}_{t}({\bm x}) \right| \le \frac1N + 2\left((L\vartheta+\eta){\bm e}_s({\bm x}) + \eta(2L^{-1})\right),\quad \bm x\in\bX^N.
\end{align*}
\end{lemma}
\begin{proof}
Recall the definition of $\bm T^{\bm\pi}_t$ from \eqref{eq:DefT} as well as the definitions in \eqref{eq:DefMeancTN}. Notice that
\begin{align*}
\bm T^{(\fp^1_t,\tilde\fp^2_t,\dots,\tilde\fp^N_t)}_{s} \circ {\check{\bm\cT}}^{\wt{\bm\fP}}_{s+1,t,{\overline\Xi}}  {\bm e}_{t}({\bm x}) = \int_{\bX} {\check{\bm\cT}}^{\wt{\bm\fP}}_{s+1,t,{\overline\Xi}} {\bm e}_t(\bm y)\left[Q^{\fp^1_{s,{\bm x}}}_{t,x^1,\overline\delta_{\bm x}}\otimes\bigotimes_{n=2}^N Q^{\tilde\fp^n_{s,x^n,\overline\delta_{\bm x}}}_{s,x^n,\overline\delta_{\bm x}}\right](\dif \bm y).
\end{align*}
Additionally, 
\begin{align*}
&{\check{\bm\cT}}^{\wt{\bm\fP}}_{s,t,{\overline\Xi}} {\bm e}_{t}({\bm x}) = {\check{\bm T}}^{\tilde{\bm\fP}_s}_{s,\overline\xi_{s}}\circ{\check{\bm\cT}}^{\wt{\bm\fP}}_{s+1,t,{\overline\Xi}} {\bm e}_{t}({\bm x}) = \int_{\bX} {\check{\bm\cT}}^{\wt{\bm\fP}}_{s+1,t,{\overline\Xi}} {\bm e}_t(\bm y) \left[\bigotimes_{n=1}^N Q^{\tilde\fp^n_{s,x^n,\overline\xi_{s}}}_{s,x^n,\overline\xi_{s}}\right](\dif \bm y).
\end{align*}
The above together with \cref{lem:EstIndepInt} implies that
\begin{align*}
&\left| \bm T^{(\fp^1_t,\tilde\fp^2_t,\dots,\tilde\fp^N_t)}_{s}\circ{\check{\bm\cT}}^{\wt{\bm\fP}}_{s+1,t,{\overline\Xi}} {\bm e}_{t}({\bm x}) - {\check{\bm\cT}}^{\wt{\bm\fP}}_{s,t,{\overline\Xi}} {\bm e}_{t}({\bm x}) \right|\\
&\quad\le \sup_{y^2,\dots,y^N} \left|\int_{\bX} {\check{\bm\cT}}^{\wt{\bm\fP}}_{s+1,t,{\overline\Xi}} {\bm e}_t({\bm y}) \left( Q^{\fp^1_{s,{\bm x}}}_{s,x^1,\overline\delta_{\bm x}}(\dif y^1) - Q^{\tilde\fp^1_{s,x^n,\overline\xi_{s}}}_{s,x^1,\overline\xi_{s}}(\dif y^1) \right)\right|\\
&\qquad+ \sum_{n=2}^N \sup_{\set{y^1,\dots,y^N}\setminus\set{y^n}} \left|\int_{\bX} {\check{\bm\cT}}^{\wt{\bm\fP}}_{s+1,t,{\overline\Xi}} {\bm e}_t({\bm y}) \left( Q^{\tilde\fp^n_{s,x^n,\overline\delta_{\bm x}}}_{s,x^n,\overline\delta_{\bm x}}(\dif y^n) - Q^{\tilde\fp^n_{s,x^n,\overline\xi_{s}}}_{s,x^n,\overline\xi_{s}}(\dif y^n) \right)\right|\\
&\quad= \sup_{y^2,\dots,y^N} \left|\int_{\bX} \left( {\check{\bm\cT}}^{\wt{\bm\fP}}_{s+1,t,{\overline\Xi}} {\bm e}_t({\bm y}) - \inf_{y^1} {\check{\bm\cT}}^{\wt{\bm\fP}}_{s+1,t,{\overline\Xi}} {\bm e}_t({\bm y}) \right)  \left( Q^{\fp^1_{s,{\bm x}}}_{s,x^1,\overline\delta_{\bm x}}(\dif y^1) - Q^{\tilde\fp^1_{s,x^1,\overline\xi_{s}}}_{s,x^1,\overline\xi_{s}}(\dif y^1) \right)\right|\\
&\qquad+ \sum_{n=2}^N \sup_{\set{y^1,\dots,y^N}\setminus\set{y^n}} \left|\int_{\bX} \left( {\check{\bm\cT}}^{\wt{\bm\fP}}_{s+1,t,{\overline\Xi}} {\bm e}_t({\bm y}) - \inf_{y^n} {\check{\bm\cT}}^{\wt{\bm\fP}}_{s+1,t,{\overline\Xi}} {\bm e}_t({\bm y}) \right)  \left( Q^{\tilde\fp^n_{s,x^n,\overline\delta_{\bm x}}}_{s,x^n,\overline\delta_{\bm x}}(\dif y^n) - Q^{\tilde\fp^n_{s,x^n,\overline\xi_{s}}}_{s,x^n,\overline\xi_{s}}(\dif y^n) \right)\right|.
\end{align*}
It follows from \cref{lem:EstDiffMeancTInfMeancT} and \cref{lem:QCont} that  
\begin{align*}
&\left| \bm T^{(\fp^1_t,\tilde\fp^2_t,\dots,\tilde\fp^N_t)}_{s} \circ {\check{\bm\cT}}^{{\bm\fP}}_{s+1,t,{\overline\Xi}} {\bm e}_{t}({\bm x}) - {\check{\bm\cT}}^{{\bm\fP}}_{s,t,{\overline\Xi}} {\bm e}_{t}({\bm x}) \right|\\
&\quad\le \frac1N + \frac2{N}\sum_{n=2}^N \left( L \|\tilde\fp^n(x^n,\overline\delta_{\bm x})-\tilde\fp^n(x^n,\overline\xi_s)\|_{BL}  + \eta(2L^{-1}) + \eta(\|\overline\delta_{\bm x}-\overline\xi_s\|_{BL}) \right)\\
&\quad\le \frac1N + 2\left(L\vartheta({\bm e}_s({\bm x})) + \eta(2L^{-1}) + \eta({\bm e}_s({\bm x}))\right) = \frac1N + 2\left((L\vartheta + \eta){\bm e}_s({\bm x}) + \eta(2L^{-1})\right).
\end{align*}
\end{proof}

\begin{lemma}\label{lem:EstfTe}
Under the setting of \cref{prop:ProcConc}, 
for $t\ge 2$ and $L>1+\eta(2L^{-1})$ we have
\begin{align*}
{\bm\bT}^{{\bm\fP}}_{t} {\bm e}_t \le  2\sum_{r=1}^{t-1} (L\vartheta + \eta) \circ {\bm\bT}^{{\bm\fP}}_{r}{\bm e}_{r} + (t-1)\left(\frac1{N} + 2\eta(2L^{-1})\right) + \fr_{\fK}(N),
\end{align*}
where $\fr_{\fA}$ is defined in \cref{lem:EmpMeasConc}.
\end{lemma}

\begin{proof}
Recall the definitions of ${\bm\bT}^{\bm\fP}_t$ and ${\check{\bm\cT}}^{\wt{\bm\fP}}_t$ in \eqref{eq:DeffT} and \eqref{eq:DefMeancTN}, respectively. Expanding ${\bm\bT}^{{\bm\fP}}_{t} {\bm e}_t$ via telescope sum, we yield
\begin{align*}
{\bm\bT}^{{\bm\fP}}_{t}  {\bm e}_t &\le \sum_{r=1}^{t-1} {\bm\bT}^{{\bm\fP}}_{r} \left| \bm T^{(\fp^1_r, \tilde\fp^2_r,\dots,\tilde\fp^N_r)}_{r} \circ {\check{\bm\cT}}^{\wt{\bm\fP}}_{r+1,t,{\overline\Xi}} {\bm e}_{t} - {\check{\bm\cT}}^{\wt{\bm\fP}}_{r,t,{\overline\Xi}} {\bm e}_{t} \right| + \mathring{\bm T}\circ{\check{\bm\cT}}^{\wt{\bm\fP}}_{1,t,{\overline\Xi}} {\bm e}_t.
\end{align*}
This together with \cref{lem:EstDiffTMeanT} implies
\begin{align*}
{\bm\bT}^{{\bm\fP}}_{t} {\bm e}_t &\le 2\sum_{r=1}^{t-1} {\bm\bT}^{{\bm\fP}}_{r} \circ (L\vartheta+\eta){\bm e}_r ({\bm x}) + (t-1)\left(2\eta(2L^{-1})+\frac1N\right) + \mathring{\bm T} \circ {\check{\bm\cT}}^{\wt{\bm\fP}}_{1,t,{\overline\Xi}} {\bm e}_t\\
&\le 2\sum_{r=1}^{t-1} (L\vartheta+\eta) \circ {\bm\bT}^{{\bm\fP}}_{r} {\bm e}_r ({\bm x}) + (t-1)\left(2\eta(2L^{-1})+\frac1N\right) + \mathring{\bm T} \circ {\check{\bm\cT}}^{\wt{\bm\fP}}_{1,t,{\overline\Xi}} {\bm e}_t
\end{align*} 
where we have use Jensen's inequality (cf. \cite[Section 11.5, Theorem 11.24]{Aliprantis2006book}) in the last line. Finally, recall the definition of $\overline{\bQ}^{\tilde\fp}_{t,\Xi}$ below \eqref{eq:DefMeanfT}. In view of the independence in \eqref{eq:DecompMeanT}, by \cref{lem:EmpMeasConc}, we obtain
\begin{align*}
\mathring{\bm T}\circ{\check{\bm\cT}}^{\wt{\bm\fP}}_{t,{\overline\Xi}} {\bm e}_t = \int_{\bX^N} {\bm e}_t({\bm y}) \left[\bigotimes_{n=1}^N\overline{\bQ}^{\tilde\fp^n}_{t,\overline\Xi}\right](\dif{\bm y}) \le \fr_{\fK}(N),
\end{align*}
which completes the proof.
\end{proof}

We are ready to prove Proposition \cref{prop:ProcConc}.
\begin{proof}[Proof of \cref{prop:ProcConc}]
Let $\fe_t$ be defined in \cref{subsec:ErrorTerms}. For $t=1$, the statement is obvious from \eqref{eq:DeffT} and \cref{lem:EmpMeasConc}. We proceed by induction. Suppose that there is a $t=1,\dots,T-1$ such that ${\bm\bT}^{{\bm\fP}}_{r}{\bm e}_r\le \fe_{r}$ for all $r=1,\dots,t$. By \cref{lem:EstfTe}, for $L>1+\eta(2L^{-1})$ we have
\begin{align*}
{\bm\bT}^{{\bm\fP}}_{t+1} {\bm e}_{t+1} \le 2\sum_{r=1}^{t} \big(L\vartheta(\fe_{r}) + \eta(\fe_{r})\big)  + t\left(\frac1N+2\eta(2L^{-1})\right) + \fr_{\fK}(N),
\end{align*}
thus ${\bm\bT}^{{\bm\fP}}_{t+1} {\bm e}_{t+1}\le\fe_{t+1}$.  Finally, invoking Jensen's inequality (cf. \cite[Section 11.5, Theorem 11.24]{Aliprantis2006book}), we conclude that ${\bm\bT}^{{\bm\fP}}_{t} \circ \beta {\bm e}_{t} \le \beta\circ {\bm\bT}^{{\bm\fP}}_{t} {\bm e}_{t}  \le  \beta(\fe_{t}(N))$. 
\end{proof}

Apart from Proposition \ref{prop:ProcConc}, the following result is also useful.
\begin{proposition}\label{prop:EstDiffNPMean}
Let $\wt{\bm\fP}$ and $\overline\Xi$ be as introduced in \cref{subsec:Scenarios}. If \cref{assump:QTight}, \cref{assump:PCont} and \cref{assump:SymCont} holds, then for any $h\in B_b(\bX)$ and $t=1,\dots,T$, we have 
$$\left|\bm\bT^{\wt{\bm\fP}}_{t}h-\overline\bT^{\tilde\fp^1}_{t,{\overline\Xi}}h\right| \le \|h\|_\infty \fe_{t},$$ 
where $\fe_t$ is defined in \cref{subsec:ErrorTerms}, and when acted by $\bm\bT^{\wt{\bm\fP}}_{t}$, $h$ is treated as a function from $ B_b(\bX^N)$ that is constant in $(x^2,\dots,x^N)$ .
\end{proposition}

\begin{proof}
Recall the definition of $\bm\bT^{\bm\fP}_t$ and $\overline\bT^{\wt{\bm\fP}}_t$ from \eqref{eq:DeffT} and \eqref{eq:DefMeanfT}, respectively. Note that
\begin{align*}
\left|\bm\bT^{\wt{\bm\fP}}_{t} h - \overline\bT^{\tilde\fp^1}_{t,{\overline\Xi}} h\right| &\le \sum_{r=1}^{t-1} \left|\bm\bT^{\wt{\bm\fP}}_{r+1} \circ \overline\cT^{\tilde\fp^1}_{r+1,t,{\overline\Xi}} h - \bm\bT^{\wt{\bm\fP}}_{r} \circ \overline\cT^{\tilde\fp^1}_{r,t,{\overline\Xi}} h\right| \le \sum_{r=1}^{t-1} \bm\bT^{\wt{\bm\fP}}_{r} \left| \left( \bm T^{(\tilde\fp^1_r,\dots,\tilde\fp^N_r)}_r - \overline T^{\tilde\fp^1_r}_{r,\overline\xi_r} \right) \circ \overline\cT^{\tilde\fp^1}_{r+1,t,{\overline\Xi}} h\right|.
\end{align*}
Let $g=\overline\cT^{\tilde\fp^1}_{r+1,t,{\overline\Xi}} h$. Clearly, $\|g\|_\infty\le\|h\|_\infty$. In view of \eqref{eq:DefT} and \eqref{eq:DefMeanT}, for $L>1+\eta(2L^{-1})$, 
\begin{align*}
&\left|\bm T^{(\tilde\fp^1_r,\dots,\tilde\fp^N_r)}_r g ({\bm x}) - \overline T^{\tilde\fp^1_r}_{r,\overline\xi_r} g ({\bm x})\right| = \left|\int_\bX g(y) Q^{\tilde\fp^1_{t,x^1,\overline\delta_{\bm x}}}_{r,x^1,\overline\delta_{\bm x}}(\dif y) - \int_\bX g(y) Q^{\tilde\fp^1_{r,x^1,\overline\xi_t}}_{r,x^1,\overline\xi_r}(\dif y)\right|\\
&\quad\le \|h\|_\infty \left( ( L\vartheta + \eta) {\bm e}_t({\bm x})   + 2\eta(2L^{-1}) \right),
\end{align*}
where we have used \cref{lem:QCont} in the last inequality. It follows from \cref{prop:ProcConc} that for $L>1+\eta(2L^{-1})$,
\begin{align*}
\left|\bm\bT^{\wt{\bm\fP}}_{t} h - \overline\bT^{\tilde\fp^1}_{t,{\overline\Xi}} h\right| \le 2\|h\|_\infty \left( \sum_{r=1}^{t-1}( L\vartheta(\fe_t) + \eta(\fe_t)) + (t-1)\eta(2L^{-1}) \right).
\end{align*}
Taking infimum over $L>1+\eta(2L^{-1})$ on both hand sides above, we conclude the proof.
\end{proof}

\subsection{Approximating $N$-player stepwise exploitability}\label{subsec:ApproxExploitability}
In this section, we  write $V_{\overline\xi_{T}} := V(\cdot,\overline\xi_{T})$.

\begin{lemma}\label{lem:EstfTDiffcS}
Let $t=2,\dots,T-1$. Let $\wt{\bm\fP}$, $\overline\Xi$ and $\overline\xi_T$ be as introduced in \cref{subsec:Scenarios}.  Consider a ${\bm\fP}\in\Pi^{N\times(T-1)}$ such that $\fp^1_r\in\wt{\Pi}$ for $r=t,\dots,T-1$ is $\vartheta^1$-symmetrically continuous, and $\fp^n=\tilde\fp^n$ for $n=2,\dots,N$. Suppose \cref{assump:ActionDomain}, \cref{assump:QTight}, \cref{assump:GBasic} (i)-(iii) (v), \cref{assump:PCont}, \cref{assump:GCont} and \cref{assump:SymCont}. Then,\footnote{For $t\ge 2$, $\overline\cS^{\fp^1}_{t,T,{\overline\Xi}}v$ is understood as $\overline S^{\fp^1_t}_{t,\overline\xi_t}\circ\cdots\circ\overline S^{\fp^1_{T-1}}_{T-1,\overline\xi_{T_1}} v$, and does not depend on $\fp^1_1,\dots,\fp^1_{t-1}$.} 
\begin{align*}
&\bm\bT^{{\bm\fP}}_{t}\left|\bm\cS^{{\bm\fP}}_{t,T} {\bm U} - \overline\cS^{\fp^1}_{t,T,{\overline\Xi}} V_{\overline\xi_{T}}\right| \le \sum_{r=t}^{T-1}  \bar{c}^{r-t}\cC_{T-r}(\|V\|_\infty) \big(\zeta(\vartheta^1(\fe_{r})) + \zeta(\fe_{r}) \big) + \bar{c}^{T-t}\iota({\fe_T}).
\end{align*}
\end{lemma}
\begin{proof}
Observe that
\begin{align*}
&\bm\bT^{{\bm\fP}}_{t}\left|\bm\cS^{{\bm\fP}}_{t,T} {\bm U} - \overline\cS^{\fp^1}_{t,T,{\overline\Xi}} V_{\overline\xi_{T}} \right|\\
&\quad\le \sum_{r=t}^{T-1} \bm\bT^{{\bm\fP}}_{t} \left| \bm\cS^{{\bm\fP}}_{t,r+1} \circ \overline\cS^{\fp^1}_{r+1,T,{\overline\Xi}}V_{\overline\xi_{T}} - \bm\cS^{{\bm\fP}}_{t,r} \circ \overline\cS^{\fp^1}_{r,T,{\overline\Xi}}V_{\overline\xi_{T}} \right| +  \bm\bT^{{\bm\fP}}_{t} \left|\bm\cS^{{\bm\fP}}_{t,T} {\bm U} - \bm\cS^{{\bm\fP}}_{t,T} V_{\overline\xi_{T}}\right|\\
&\quad\le \sum_{r=t}^{T-1} \bar{c}^{r-t}  \bm\bT^{{\bm\fP}}_{r} \left| \bm S^{(\fp^1_t,\dots,\fp^N_t)}_{r} \circ \overline\cS^{\fp^1}_{r+1,T,{\overline\Xi}}V_{\overline\xi_{T}} - \overline\cS^{\fp^1}_{r,T,{\overline\Xi}}V_{\overline\xi_{T}} \right| +  \bar{c}^{T-t}\bm\bT^{{\bm\fP}}_{T}\!\! \circ \iota {\bm e}_T,
\end{align*}
where we have used \eqref{eq:DefcS}, \cref{lem:EstDiffS} and \eqref{eq:DeffT} in last inequality. In addition, due to \eqref{eq:DefS}, \cref{assump:GBasic} (v), \eqref{eq:DefMeanS} and \eqref{eq:DefMeancS}, we have
\begin{align*}
&\left| \left[\bm S^{(\fp^1_t,\dots,\fp^N_t)}_{r} \circ \overline\cS^{\fp^1}_{r+1,T,{\overline\Xi}}V_{\overline\xi_{T}}\right]({\bm x}) - \left[\overline\cS^{\fp^1}_{r,T,{\overline\Xi}}V_{\overline\xi_{T}}\right](x^1) \right|\\
&\quad= \left| \left[\overline G^{\fp^1_{r,x^1,\overline\delta_{\bm x}}}_{r,\overline\delta_{\bm x}} \circ \overline\cS^{\fp^1}_{r+1,T,{\overline\Xi}}V_{\overline\xi_{T}}\right](x^1) - \left[\overline G^{\fp^1_{r,x^1,\overline\xi_r}}_{r,\overline\xi_r}\circ\overline\cS^{\fp^1}_{r+1,T,{\overline\Xi}}V_{\overline\xi_{T}}\right](x^1) \right|\\
&\quad\le \cC_{T-r}(\|V\|_\infty)\left(\zeta(\vartheta^1(\|\overline\delta_{\bm x}-\overline\xi_r\|_{BL})) + \zeta(\|\overline\delta_{\bm x}-\overline\xi_r\|_{BL})\right),
\end{align*}
where we have used \cref{assump:GCont} (ii) and \cref{lem:EstcS} in the last inequality. Note $\ell\mapsto\zeta(\vartheta^1(\ell))+\zeta(\ell)$ is also a subadditive modulus of continuity.  
The above together with \cref{prop:ProcConc} completes the proof.
\end{proof}

\begin{proposition}\label{prop:EstfTDiffcSstar}
Let $\wt{\bm\fP}$, $\overline\Xi$ and $\overline\xi_T$ be as introduced in \cref{subsec:Scenarios}.  Consider a ${\bm\fP}\in\Pi^{N\times(T-1)}$ that satisfies $\fp^n=\tilde\fp^n$ for $n=2,\dots,N$. Suppose \cref{assump:ActionDomain}, \cref{assump:QTight}, \cref{assump:GBasic} (i)-(iii) (v), \cref{assump:PCont}, \cref{assump:GCont} and \cref{assump:SymCont}. Then, we have $\bm\bT^{{\bm\fP}}_{T}\left|{\bm U} - V_{\overline\xi_T} \right| \le \iota({\fe_T})$, and
\begin{align*}
&\bm\bT^{{\bm\fP}}_{t} \left| \bm\cS^{*{\bm\fP}}_{t,T} {\bm U} -  \overline\cS^{*}_{t,T,{\overline\Xi}}V_{\overline\xi_{T}} \right| \le (T+1-t)\left(\sum_{r=t}^{T-1} \bar{c}^{r-t} \cC_{T-r}(\|V\|_\infty) \zeta(\fe_{r}) + \bar{c}^{T-t}\iota({\fe_T})\right),\quad t=1,\dots,T-1.
\end{align*}
\end{proposition}

\begin{proof}
For $t=T$, owing Proposition \ref{prop:ProcConc} and the hypothesis that $V$ is $\iota$-symmetrically continuous, we have $\bm\bT^{{\bm\fP}}_{T}\left|{\bm U} - V_{\overline\xi_T} \right| \le \iota({\fe_T})$, which proves the statement for $t=T$. We proceed by backward induction. Suppose the statement is true at $t+1$ for some $t=1,\dots,T-1$. In view of \cref{lem:Sstar}, let $\fp^*\in\Pi^{T-1}$ and $\overline\fp^*\in\overline{\Pi}^{T-1}$ be the optimal policies attaining $\bm\bS^{*\fP}_TV$ and $\overline\bS^{*}_{T,(\overline\xi_{1},\dots,\overline\xi_{T-1})}V_{\overline\xi_{T}}$, respectively. It follows that
\begin{align*}
\bm\cS^{*{\bm\fP}}_{t,T} {\bm U}({\bm x}) - \overline\cS^{*}_{t,T,{\overline\Xi}}V_{\overline\xi_{T}}(x^1) &\le \bm\cS^{(\overline\fp^*,\tilde\fp^2,\dots,\tilde\fp^N)}_{t,T} {\bm U}({\bm x}) - \overline\cS^{\fp^*}_{t,T,{\overline\Xi}}V_{\overline\xi_{T}}(x^1)\\
&\le \left|\bm\cS^{(\overline\fp^*,\tilde\fp^2,\dots,\tilde\fp^N)}_{t,T} {\bm U}({\bm x}) - \overline\cS^{\fp^*}_{t,T,{\overline\Xi}}V_{\overline\xi_{T}}(x^1)\right|.
\end{align*}
Meanwhile,
\begin{align*}
\bm\cS^{*{\bm\fP}}_{t,T} {\bm U}({\bm x}) - \overline\cS^{*}_{t,T,{\overline\Xi}}V_{\overline\xi_{T}}(x^1) 
&\ge \bm S^{(\fp^*_t,\tilde\fp^2_t,\dots,\tilde\fp^N_t)}_{t} \circ \overline\cS^{*}_{t+1,T,{\overline\Xi}} V_{\overline\xi_{T}}({\bm x}) - \overline\cS^{*}_{t,T,{\overline\Xi}}V_{\overline\xi_{T}}(x^1)\\
&\qquad- \left|\bm\cS^{*{\bm\fP}}_{t,T} {\bm U}({\bm x}) - \bm S^{(\fp^*_t,\tilde\fp^2_t,\dots,\tilde\fp^N_t)}_{t} \circ \overline\cS^{*}_{t+1,T,{\overline\Xi}} V_{\overline\xi_{T}}({\bm x})\right|,
\end{align*}
For the first term in the right hand side above, in view of \cref{assump:GBasic} (v), \cref{assump:GCont} (ii) and \cref{lem:EstcS}, we further yield
\begin{align*}
&\left[\bm S^{(\fp^*_t,\tilde\fp^2_t,\dots,\tilde\fp^N_t)}_{t} \circ \overline\cS^{*}_{t+1,T,{\overline\Xi}} V_{\overline\xi_{T}}\right] ({\bm x}) = \left[\overline G^{\fp^*_{t,{\bm x}}}_{t,\overline\delta_{\bm x}} \circ \overline\cS^{*}_{t+1,T,{\overline\Xi}} V_{\overline\xi_{T}}\right] (x^1)\\
&\quad\ge \left[\overline G^{\fp^*_{t,{\bm x}}}_{t,\overline\xi_t} \circ \overline\cS^{*}_{t+1,T,{\overline\Xi}} V_{\overline\xi_{T}}\right] (x^1) -  \left(c_0 + c_1\cC_{T-(t+1)}(\|V\|_\infty)\right)\zeta(\|\overline\delta_{\bm x}-\overline\xi_{t}\|_{BL})\\
&\quad\ge \left[\overline\cS^{*}_{t+1,T,{\overline\Xi}} V_{\overline\xi_{T}}\right](x^1) - \cC_{T-t}(\|V\|_\infty)\zeta({\bm e}_t({\bm x})).
\end{align*}
Therefore,
\begin{align*}
\bm\cS^{*{\bm\fP}}_{t,T} {\bm U}({\bm x}) - \overline\cS^{*}_{t,T,{\overline\Xi}}V_{\overline\xi_{T}}(x^1) \ge - \cC_{T-t}(\|V\|_\infty)\zeta {\bm e}_t({\bm x}) - \left|\bm\cS^{*{\bm\fP}}_{t,T} {\bm U}({\bm x}) - \bm S^{(\fp^*_t,\tilde\fp^2_t,\dots,\tilde\fp^N_t)}_{t} \circ \overline\cS^{*}_{t+1,T,{\overline\Xi}} V_{\overline\xi_{T}}({\bm x})\right|.
\end{align*}
Note additionally that for any $c,\underline c,\bar{c}\in\bR$ with $\underline c\le 0\le \bar{c}$ and $\underline c \le c \le \bar{c}$, we have $|c|\le|\underline c|+|\bar{c}|$. By combining this with the estimation above, we yield
\begin{align*}
\bm\bT^{{\bm\fP}}_{t} \left| \bm\cS^{*{\bm\fP}}_{t,T} {\bm U} -  \overline\cS^{*}_{t,T,{\overline\Xi}}V_{\overline\xi_T} \right| \le I_1 + I_2,
\end{align*}
where
\begin{gather*}
I_1 := \bm\bT^{{\bm\fP}}_{t}\left|\bm\cS^{(\overline\fp^*,\tilde\fp^2,\dots,\tilde\fp^N)}_{t,T} {\bm U} - \overline\cS^{\overline\fp^*}_{t,T,{\overline\Xi}} V_{\overline\xi_{T}} \right|\\
I_2 := \cC_{T-t}(\|V\|_\infty)\bm\bT^{{\bm\fP}}_{t} \circ \zeta {\bm e}_t + \bm\bT^{{\bm\fP}}_{t} \left|\bm\cS^{*{\bm\fP}}_{t,T} {\bm U} - \bm S^{(\fp^*_t,\tilde\fp^2_t,\dots,\tilde\fp^N_t)}_{t} \circ \overline\cS^{*}_{t+1,T,{\overline\Xi}} V_{\overline\xi_{T}}\right|.
\end{gather*}
As for $I_1$, by defining $\fp^1\oplus_{t}\overline\fp^*:=(\fp^1,\dots,\fp^1_{t-1},\overline\fp^*_t,\dots,\overline\fp^*_{T-1})$ if $t\ge 2$ and $\fp^1\oplus_{t}\overline\fp^*:=\overline\fp^*$ if $t=1$, we have $\bm\bT^{{\bm\fP}}_{t} = \bm\bT^{(\fp^1\oplus_{t}\overline\fp^*,\tilde\fp^2,\dots,\tilde\fp^N)}_{t}$ and $\bm\cS^{\overline\fp^*,{\bm\fP}}_{t,T} = \bm\cS^{(\fp^1\oplus_{t}\overline\fp^*,\tilde\fp^2,\dots,\tilde\fp^N)}_{t,T}$. This together with \cref{lem:EstfTDiffcS} provides an estimation on $I_1$ (note $\vartheta^1\equiv 0$ in this case). As for $I_2$, by  \cref{prop:ProcConc}, \eqref{eq:DeffT}, \cref{assump:GBasic} (iii) and the induction hypothesis, we yield\footnote{For the sake of neatness, we set $\sum_{r=T}^{T-1}=0$.}
\begin{align*}
I_2 &\le \cC_{T-t}(\|V\|_\infty)\zeta({\fe_t}) - \bar{c}\bm\bT^{{\bm\fP}}_{t} \circ \bm T^{(\fp^*_t,\tilde\fp^2_t,\dots,\tilde\fp^N_t)}_{t} \left| \bm\cS^{*{\bm\fP}}_{t+1,T} {\bm U} -  \overline\cS^{*}_{t+1,T,{\overline\Xi}}V_{\overline\xi_{T}} \right|\\
&\le (T-t)\left(\cC_{T-t}(\|V\|_\infty)\zeta({\fe_t}) - \sum_{r=t+1}^{T-1} \bar{c}^{r-t} \cC_{T-r}(\|V\|_\infty) \zeta(\fe_{r}) + c_1^{T-t}\iota({\fe_T})\right)\\
&= (T-t)\left(\sum_{r=t}^{T-1} \bar{c}^{r-t} \cC_{T-r}(\|V\|_\infty) \zeta(\fe_{r}) + \overline{C}^{T-t}\iota({\fe_T})\right).
\end{align*}
By combining the estimates of $I_1$ and $I_2$, the proof is complete.
\end{proof}

\begin{proposition}\label{prop:EstDiffRMF}
Let $\wt{\bm\fP}$, $\overline\Xi$ and $\overline\xi_T$ be as introduced in \cref{subsec:Scenarios}. Suppose \cref{assump:ActionDomain}, \cref{assump:QTight}, \cref{assump:GBasic} (i)-(iii) (v), \cref{assump:PCont}, \cref{assump:GCont} and \cref{assump:SymCont}. Then, with $\fE$ defined in \cref{subsec:ErrorTerms}, we have 
\begin{align*}
\left|\bm\fR(\wt{\bm\fP}; {\bm U}) - \sum_{t=1}^{T-1} \bar{c}^t \overline\bT^{\tilde\fp^1}_{t,{\overline\Xi}} \left( \overline S^{\tilde\fp^1}_{t,\overline \xi_t} \circ \overline\cS^{*}_{t+1,T,{\overline\Xi}} V_{\overline\xi_{T}} - \overline\cS^{*}_{t,T,{\overline\Xi}} V_{\overline\xi_{T}} \right)\right| \le {\fE}.
\end{align*}
\end{proposition}

\begin{proof}
Let $t=1,\dots,T-1$. Observe that, by \eqref{eq:DeffT} and \cref{lem:EstDiffS},
\begin{align*}
\left|\bm\bT^{\wt{\bm\fP}}_t \circ \bm S^{\wt{\bm\fP}_t}_t  \circ \bm\cS^{*\wt{\bm\fP}}_{t+1,T} {\bm U} - \bm\bT^{\wt{\bm\fP}}_t \circ \bm S^{\wt{\bm\fP}_t}_t  \circ \overline\cS^{*}_{t+1,T,{\overline\Xi}} V_{\overline\xi_{T}}\right| \le \bar{c} \bm\bT^{\wt{\bm\fP}}_{t+1}\left| \bm\cS^{*\wt{\bm\fP}}_{t+1,T} {\bm U} - \overline\cS^{*}_{t+1,T,{\overline\Xi}} V_{\overline\xi_{T}}\right|.
\end{align*}
In addition, by \eqref{eq:DefS} and \cref{lem:QCont}, we have
\begin{align*}
&\left| \bm S^{\wt{\bm\fP}_t}_t  \circ \overline\cS^{*}_{t+1,T,{\overline\Xi}}V_{\overline\xi_{T}} ({\bm x}) - \overline S^{\tilde\fp^1_t}_{t,\overline \xi_t} \circ \overline\cS^{*}_{t+1,T,{\overline\Xi}}V_{\overline\xi_{T}} ({\bm x}) \right|\\
&\quad= \left| \overline G^{\tilde\fp^1_{t,x^1,\overline\delta_{\bm x}}}_{t,\overline\delta_{\bm x}}  \circ \overline\cS^{*}_{t+1,T,{\overline\Xi}} V_{\overline\xi_{T}} ({\bm x}) - \overline G^{\tilde\fp^1_{t,x^1,\overline\xi_t}}_{t,\overline\xi_t}  \circ \overline\cS^{*}_{t+1,T,{\overline\Xi}} V_{\overline\xi_{T}} ({\bm x}) \right| \le \cC_{T-t}(\|V\|_\infty)(\zeta\vartheta+\zeta){\bm e}_t({\bm x}).
\end{align*}
By combining the above, we yield
\begin{multline*}
\left|\bm\bT^{\wt{\bm\fP}}_t \circ \bm S^{\wt{\bm\fP}_t}_t  \circ \bm\cS^{*\wt{\bm\fP}}_{t+1,T} {\bm U} - \bm\bT^{\wt{\bm\fP}}_t \circ \overline S^{\tilde\fp^1_t}_{t,{\overline\xi_t}} \circ \overline\cS^{*}_{t+1,T,{\overline\Xi}} V_{\overline\xi_{T}} ({\bm x})\right|\\
\le \bar{c} \bm\bT^{\wt{\bm\fP}}_{t+1}\left| \bm\cS^{*\wt{\bm\fP}}_{t+1,T} {\bm U} - \overline\cS^{*}_{t+1,T,{\overline\Xi}}V_{{\overline\Psi}}\right| + \cC_{T-t}(\|V\|_\infty) \bm\bT^{\wt{\bm\fP}}_t \circ (\zeta\vartheta+\zeta){\bm e}_t.
\end{multline*}
This together with \cref{prop:EstfTDiffcSstar} and \cref{prop:ProcConc} implies that
\begin{align*}
& \left| \bm\bT^{{\bm\fP}}_t \left(\bm S^{\wt{\bm\fP}_t}_t  \circ \bm\cS^{*{\bm\fP}}_{t+1,T} {\bm U}  - \bm\cS^{*{\bm\fP}}_{t,T} {\bm U}\right) - \bm\bT^{\wt{\bm\fP}}_t \left( \overline S^{\tilde\fp^1}_{t,{\overline\xi_t}} \circ \overline\cS^{*}_{t+1,T,{\overline\Xi}} V_{\overline\xi_{T}} - \overline\cS^{*}_{t,T,{\overline\Xi}} V_{\overline\xi_{T}} \right) \right| \\
&\quad\le \bar{c} (T-t)\left(\sum_{r=t+1}^{T-1} \bar{c}^{r-(t+1)} \cC_{T-r}(\|V\|_\infty) \zeta(\fe_{r}) + \bar{c}^{T-(t+1)}\iota({\fe_T})\right)\\
&\qquad\resizebox{0.92\hsize}{!}{$ + \cC_{T-t}(\|V\|_\infty) \big(\zeta\vartheta({\fe_t})+\zeta({\fe_t})\big) + (T+1-t) \left(\sum_{r=t}^{T-1} \bar{c}^{r-t} \cC_{T-r}(\|V\|_\infty) \zeta(\fe_{r}) + \bar{c}^{T-t} \iota({\fe_T})\right) $}.
\end{align*}
Finally, recall the definition of $\bm\fR$ and $\fE$ from \eqref{eq:DefStepExploitability} and \cref{subsec:ErrorTerms}, respectively, the proof is complete.
\end{proof}

\subsection{Proof of \cref{lem:VersionIndpi}}\label{subsec:Prooflem:VersionIndpi}
\begin{proof}[Proof of \cref{lem:VersionIndpi}]
Due to \cref{lem:pipsiUnique}, $\int_\bX\overline S^{{\overline\pi^{\psi_t}}}_{t,{\xi^{\psi_t}}} v(y) {\xi^{\psi_t}}(\dif y) = \int_\bX\overline S^{{\overline\pi^{\psi_t}}'}_{t,{\xi^{\psi_t}}} v(y) {\xi^{\psi_t}}(\dif y)$. Next, in view of \eqref{eq:DefMeanfS}, \cref{lem:EstDiffS} and \cref{lem:MFF}, expanding the left hand side below by telescope sum, we yield
\begin{align*}
\left|\overline\bS^{{\overline\fp^\Psi}'}_{T,{\Xi^{\Psi}}} v - \overline\bS^{{\overline\fp^\Psi}}_{T,{\Xi^{\Psi}}} v\right| 
&\le \sum_{t=1}^{T-1} \left|\overline\bS^{{\overline\fp^\Psi}}_{t,{\Xi^{\Psi}}}\circ\overline\cS^{{\overline\fp^\Psi}'}_{t,T,{\Xi^{\overline\Psi}}} v - \overline\bS^{{\overline\fp^\Psi}}_{t+1,{\Xi^{\overline\Psi}}}\circ\overline\cS^{{\overline\fp^\Psi}'}_{t+1,T,{\Xi^{\Psi}}} v\right| \\
&\le \sum_{t=1}^{T-1} \bar{c}^{t} \overline\bT^{{\overline\fp^\Psi}}_{t,{\Xi^{\Psi}}} \left|\overline\cS^{{\overline\fp^\Psi}'}_{t,T,{\Xi^{\overline\Psi}}} v - \overline S^{{\overline\pi^{\psi_t}}}_{t,{\xi^{\psi_t}}} \circ \overline\cS^{{\overline\fp^\Psi}'}_{t+1,T,{\Xi^{\Psi}}} v\right|\\
&= \sum_{t=1}^{T-1} \bar{c}^{t}\int_{\bX} \left|\overline\cS^{{\overline\fp^\Psi}'}_{t,T,{\Xi^{\overline\Psi}}} v(y) - \overline S^{{\overline\pi^{\psi_t}}}_{t,{\xi^{\psi_t}}} \circ \overline\cS^{{\overline\fp^\Psi}'}_{t+1,T,{\Xi^{\Psi}}} v(y)\right| {\xi^{\psi_t}}(\dif y) = 0.
\end{align*}
\end{proof}

\subsection{Proof of \cref{prop:EstEndExploitabilityCont}}\label{subsec:Proofprop:EstEndExploitabilityCont} 
\begin{proof}[Proof of \cref{prop:EstEndExploitabilityCont}(a)]
Note that 
\begin{align}\label{eq:EndExploitabilityTelescope}
\bm\cR({\bm\fP}; {\bm U}) &= \sum_{t=1}^{T-1} \left(\bm\bS^{{\bm\fP}}_{t+1} \circ \bm\cS^{*{\bm\fP}}_{t+1,T} {\bm U}  - \bm\bS^{{\bm\fP}}_{t} \circ \bm\cS^{*{\bm\fP}}_{t,T} {\bm U}\right)  = \sum_{t=1}^{T-1} \left(\bm\bS^{{\bm\fP}}_{t} \circ \bm S^{\bm\fP_t}_t \circ \bm\cS^{*{\bm\fP}}_{t+1,T} {\bm U}  - \bm\bS^{{\bm\fP}}_{t} \circ \bm\cS^{*{\bm\fP}}_{t,T} {\bm U}\right).
\end{align}
Then, by \cref{lem:EstDiffS}, we yields.
\begin{align*}
\bm\cR({\bm\fP}; {\bm U}) \le \sum_{t=1}^{T-1} \bar{c}^t \bm\bT^{{\bm\fP}}_t \left(\bm S^{{\bm\fP}_t}_t \circ \bm\cS^{*{\bm\fP}}_{t+1,T} {\bm U} - \bm\cS^{*{\bm\fP}}_{t,T} {\bm U}\right) = \bm\fR({\bm\fP}; {\bm U}).
\end{align*}
\end{proof}

\begin{proof}[Proof of \cref{prop:EstEndExploitabilityCont}(b)]
With a similar induction as in the proof of \cref{lem:EstDiffS}, for any ${\bm u}\ge{\bm u}'$ we have
\begin{align*}
\bm\bS^{{\bm\fP}}_{t}{\bm u} - \bm\bS^{{\bm\fP}}_{t}{\bm u}' \ge \underline c^{t} \bm\bT^{{\bm\fP}}_{t} \left(u -{\bm u}'\right), \quad t=1,\dots,T.
\end{align*}
This together with \eqref{eq:EndExploitabilityTelescope} implies that
\begin{align*}
\bm\cR({\bm\fP}; {\bm U}) \ge \sum_{t=1}^{T-1} \underline c^t \bm\bT^{{\bm\fP}}_t \left(\bm S^{{\bm\fP}_t}_t \circ \bm\cS^{*{\bm\fP}}_{t+1,T} {\bm U} - \bm\cS^{*{\bm\fP}}_{t,T} {\bm U}\right) \ge \left({\underline c}\middle/{\bar{c}}\right)^{T-1}\bm\fR({\bm\fP}; {\bm U}).
\end{align*}
\end{proof}

\begin{proof}[Proof of \cref{prop:EstEndExploitabilityCont}(c)]
To start with, observe that
\begin{align}\label{eq:EstEndExploitabilityCont}
&\left|{\bm\cR_{\vartheta}}(\wt{\bm\fP};\bm U) - \left( \overline\bS^{\tilde\fp^1}_{T,{\overline\Xi}} V_{\overline\xi_{T}} - \overline\bS^{*}_{T,{\overline\Xi}} V_{\overline\xi_{T}}  \right)\right|\nonumber\\
&\quad= \left| \left( \bm\bS^{\wt{\bm\fP}}_T {\bm U} - \inf_{\tilde\fp\in\wt{\Pi}_\vartheta}\left\{\bm\bS^{(\tilde\fp,\tilde\fp^2,\dots,\tilde\fp^N)}_{t} {\bm U} \right\} \right) - \left( \overline\bS^{\tilde\fp^1}_{T,{\overline\Xi}} V_{\overline\xi_{T}} - \inf_{\tilde\fp\in\wt{\Pi}_\vartheta}\left\{\overline\bS^{\tilde\fp}_{T,{\overline\Xi}} V_{\overline\xi_{T}}  \right\} \right) \right|\nonumber \\
&\quad\le \left| \bm\bS^{\wt{\bm\fP}}_T \bm U - \overline\bS^{\tilde\fp^1}_{T,{\overline\Xi}} V_{\overline\xi_{T}} \right| + \sup_{\tilde\fp\in\wt\Pi_\vartheta}\left| \bm\bS^{(\tilde\fp,\tilde\fp^2,\dots,\tilde\fp^N)}_T \bm U - \overline\bS^{\tilde\fp}_{T,{\overline\Xi}} V_{\overline\xi_{T}} \right| \nonumber\\
&\quad\le 2\left(\sum_{r=1}^{T-1}  \bar{c}^{r}\cC_{T-r}(\|V\|_\infty) \big(\zeta(\vartheta(\fe_{r})) + \zeta(\fe_{r}) \big) + \bar{c}^{T}\iota({\fe_T})\right),
\end{align}
where we have used \cref{assump:GBasic} (v), \eqref{eq:EstDifffS} in \cref{lem:EstDiffS} with $t=1$, and \cref{lem:EstfTDiffcS} in the last inequality. 

In order to proceed, in view of \cref{lem:Sstar}, we let $\fp^*$ and $\bar\fp^*$ be the optimal policies that attains $\bm\bS^{\wt{\bm\fP}}{\bm U}$ and $ \overline\bS^{*}_{T,{\overline\Xi}} V_{\overline\xi_{T}}$, respectively. Then,
\begin{align*}
&\left|\bm\cR(\wt{\bm\fP};\bm U) - \left( \overline\bS^{\tilde\fp^1}_{T,\overline\Xi} V_{\overline\xi_{T}} - \overline\bS^{*}_{T,\overline\Xi} V_{\overline\xi_{T}} \right)
\right| \le \left|\bm\bS^{\wt{\bm\fP}}_T {\bm U} - \overline\bS^{\tilde\fp^1}_{T,\overline\Xi} V_{\overline\xi_{T}} \right| + \left| \bm\bS^{*\wt{\bm\fP}}_{T} {\bm U} - \overline\bS^{*}_{T,\overline\Xi} V_{\overline\xi_{T}} \right|.
\end{align*}
By \cref{assump:GBasic} (v), \eqref{eq:EstDifffS} in \cref{lem:EstDiffS} with $t=1$, \cref{lem:EstfTDiffcS}, and \cref{prop:EstfTDiffcSstar}, we yield
\begin{align*}
&\left|\bm\cR(\wt{\bm\fP};V) - \left( \overline\bS^{\tilde\fp^1}_{T,{\overline\Xi}} V_{\overline\xi_{T}} - \overline\bS^{*}_{T,{\overline\Xi}} V_{\overline\xi_{T}} \right)
\right|\\
&\quad \le \sum_{r=1}^{T-1}  \bar{c}^{r}\cC_{T-r}(\|V\|_\infty) \big(\zeta(\vartheta(\fe_{r})) + \zeta(\fe_{r}) \big) + \bar{c}^{T}\iota({\fe_T}) + T\left(\sum_{r=1}^{T-1} \bar{c}^{r} \cC_{T-r}(\|V\|_\infty) \zeta(\fe_{r}) + \bar{c}^{T}\iota({\fe_T})\right)
\end{align*}
This together with \eqref{eq:EstEndExploitabilityCont} and triangle inequality completes the proof.
\end{proof}

\subsection{Proof of \eqref{eq:GoodDescription} in \cref{thm:MFApprox}}\label{subsec:Proofthm:MFApprox1}
\begin{proof}[Proof of \eqref{eq:GoodDescription}]
Recall the definition of $\breve{\bm e}_t$ from \eqref{eq:DefStateActionEmpErr}. The statement for $t=1$ follows immediately from \cref{lem:EmpMeasConc}. We now suppose $t\ge 2$. Note for any $\set{(x^k,a^k)}_{k=1}^N$, we have
\begin{align}\label{eq:BLBound}
\left\|\overline\delta_{(x^k,a^k)_{k=1}^N}-{\overline\psi}_t\right\|_{BL} - \inf_{(x^n,a^n)}\left\|\overline\delta_{(x^k,a^k)_{k=1}^N}-{\overline\psi}_t\right\|_{BL} \le \frac1N,\quad n=1,\dots,N,
\end{align}
and 
$(x^n,a^n) \mapsto \left\|\overline\delta_{(x^k,a^k)_{k=1}^N}-{\overline\psi}_t\right\|_{BL}$ is $N^{-1}$-Lipschitz continuous. We define 
\begin{align*}
\bar {\bm e}_t({\bm x}) := \int_{\bA^N}\left\|\overline\delta_{(x^k,a^k)_{k=1}^N}-{\overline\psi}_t\right\|_{BL} \left[\bigotimes_{n=1}^N\tilde\fp_t^n(x^n,\overline\xi_t)\right](\dif {\bm a}).
\end{align*}
The above together with \cref{lem:EstIndepInt} and the \ref{assump:SymCont} that $\tilde\fp^n$'s are $\vartheta$-symmetrically continuous, implies
\begin{align*}
&\left|\breve {\bm e}_t({\bm x}) - \bar {\bm e}_t({\bm x})\right|\\
&\quad\le \sum_{n=1}^N \sup_{\set{a_1,\dots,a_N}\setminus\set{a_n}}\left|\int_\bY \left\|\overline\delta_{(x^k,a^k)_{k=1}^N}-{\overline\psi}_t\right\|_{BL}\left({\tilde\fp^n_{t,x^n,\overline\delta_{\bm x}}}(\dif a^n) - \tilde\fp_t^n(x^n,\overline\xi_t)(\dif a^n)\right) \right|\\
&\quad\le \frac2N\sum_{n=1}^N\left\|{\tilde\fp^n_{t,x^n,\overline\delta_{\bm x}}} - \tilde\fp_t^n(x^n,\overline\xi_t)\right\|_{BL} \le 2\vartheta({\bm e}_t({\bm x})),\quad{\bm x}\in\bX^N,
\end{align*}
where we have replaced the integrand by the left hand side of \eqref{eq:BLBound} to get the second inequality. In view of the estimation above and \cref{prop:ProcConc}, what is left to bound is $\bm\bT^{\wt{\bm\fP}}_t\bar e^t$; the procedure is similar to the proof \cref{prop:ProcConc}. We continue to finish the proof for the sake of completeness. To this end we fix arbitrarily $i\in\set{1,\dots,N}$ and consider $\hat{\bm x}=(\hat x^1,\dots,\hat x^N)\in\bX^N$ such that $\hat x^k=x^k$ except the $i$-th entry. Note that
\begin{align*}
\left|\bar {\bm e}_t(\hat{\bm x}) - \int_{\bA^N}\left\|\overline\delta_{(x^k,a^k)_{k=1}^N}-{\overline\psi}_t\right\|_{BL} \left[\bigotimes_{k=1}^N\tilde\fp_t^k(\hat x^k,\overline\xi_t)\right](\dif {\bm a})\right| \le \frac1N.
\end{align*}
Then, by triangle inequality and a similar reasoning as above, we yield
\begin{align*}
&\left|\bar {\bm e}_t({\bm x}) - \bar {\bm e}_t(\hat{\bm x})\right|\\
&\quad\le \frac1N + \sup_{(x^k,a^k),k\neq i}\left|\int_{\bA^N}\left\|{\overline\delta_{(x^n,a^n)_{n=1}^N}}-{\overline\psi}_t\right\|_{BL} \left({\tilde\fp^i_{t,x^i,\overline\xi_t}}(\dif a^i)-{\tilde\fp^i_{t,\hat x^i,\overline\xi_t}}(\dif a^i)\right)\right| \le \frac{3}{2N}. 
\end{align*}
Recall the definitions of $\overline{\bm T}$ and $\check{\bm\cT}$ from and below \eqref{eq:DefMeancTN}. Since the above is true for any $i\in\set{1,\dots,N}$, by \eqref{eq:DecompMeanT} and \cref{lem:EstIndepInt}, we have
\begin{align}\label{eq:EstMeancTinfMeancT}
\left|{\check{\bm\cT}}^{\wt{\bm\fP}}_{s,t,{\overline\Xi}}\bar {\bm e}_t({\bm x}) - \inf_{x^i\in\bX}{\check{\bm\cT}}^{\wt{\bm\fP}}_{s,t,{\overline\Xi}}\bar {\bm e}_t({\bm x})\right| \le \frac3{2N},\quad {\bm x}\in\bX^N,\;i=1,\dots,N.
\end{align}
Next, recall the definition of $\bm\bT$ from \eqref{eq:DeffT} and note that 
\begin{align*}
\bm\bT^{\wt{\bm\fP}}_{t}\bar {\bm e}_t \le \mathring{\bm T} \circ \overline{\bm \cT}^{\wt{\bm\fP}}_{1,t,{\overline\Xi}}\bar {\bm e}_t + \sum_{r=1}^{t-1} \bm\bT^{\wt{\bm\fP}}_{r} \left| \bm T^{\wt{\bm\fP}}_{r+1}  \circ \check{\bm\cT}^{\wt{\bm\fP}}_{r+1,t,{\overline\Xi}}\bar {\bm e}_t - \check{\bm\cT}^{\wt{\bm\fP}}_{r,t,{\overline\Xi}}\bar {\bm e}_t \right|.
\end{align*}
Due to \cref{lem:EmpMeasConc} and \cref{lem:sigmaTight}, for the first term in the right hand side above, we have $\mathring{\bm T} \circ \overline{\bm T}^{\wt{\bm\fP}}_{1,t,{\overline\Xi}}\bar {\bm e}_t \le \fr_{\fK\times\bA}(N)$. We proceed to estimate the quantity with absolute sign in the second term, 
\begin{align*}
&\left| \bm T^{\wt{\bm\fP}}_{r} \circ \check{\bm\cT}^{\wt{\bm\fP}}_{r+1,t,{\overline\Xi}}\bar {\bm e}_t ({\bm x}) - \check{\bm\cT}^{\wt{\bm\fP}}_{r,t,{\overline\Xi}}\bar {\bm e}_t ({\bm x}) \right|\\
&\quad = \left| \int_{\bX^N} \check{\bm\cT}^{\wt{\bm\fP}}_{r+1,t,{\overline\Xi}}\bar {\bm e}_t({\bm y}) \left( \left[\bigotimes Q^{{\tilde\fp^n_{t,x^n,\overline\delta_{\bm x}}}}_{t,x^n,\overline\delta_{\bm x}}\right](\dif{\bm y}) - \left[\bigotimes Q^{{\tilde\fp^n_{t,x^n,\overline\xi_t}}}_{t,x^n,\overline\xi_t}\right](\dif{\bm y}) \right) \right|\\
&\quad\le \sum_{n=1}^N \sup_{\set{y^1,\dots,y^N}\setminus\set{y^N}} \left| \int_{\bX^N} \check{\bm\cT}^{\wt{\bm\fP}}_{r+1,t,{\overline\Xi}}\bar {\bm e}_t({\bm y}) \left( Q^{{\tilde\fp^n_{t,x^n,\overline\delta_{\bm x}}}}_{t,x^n,\overline\delta_{\bm x}}(\dif y^n) - Q^{{\tilde\fp^n_{t,x^n,\overline\xi_t}}}_{t,x^n,\overline\xi_t}(\dif y^n) \right) \right|
\end{align*}
where we have used \cref{lem:EstIndepInt} in the last inequality. It follows from \cref{lem:QCont} and \eqref{eq:EstMeancTinfMeancT} that, for $L\ge 1+\eta(2L^{-1})$, 
\begin{align*}
\left| \bm T^{\wt{\bm\fP}}_{r} \circ \check{\bm\cT}^{\wt{\bm\fP}}_{r+1,t,{\overline\Xi}}\bar {\bm e}_t ({\bm x}) - \check{\bm\cT}^{\wt{\bm\fP}}_{r,t,{\overline\Xi}}\bar {\bm e}_t ({\bm x}) \right| \le 3 \left((L\vartheta+\eta)e_t({\bm x})+\eta(2L^{-1})\right).
\end{align*}
Finally, by combining the above, we yield 
\begin{align*}
\bm\bT^{\wt{\bm\fP}}_{t}\bar {\bm e}_t \le \fr_{\fK\times\bA}(N) + 3\sum_{r=1}^{t-1} \bm\bT^{\wt{\bm\fP}}_{r} \circ (L\vartheta+\eta)e_{t-1} + 3(t-1)\eta(2L^{-1}).
\end{align*}
Invoking \cref{prop:ProcConc}, the proof is complete.
\end{proof}

\subsection{Proof of \eqref{eq:NoMiss} in \cref{thm:MFApprox}}\label{subsec:Proofthm:MFApprox2}
\begin{proof}[Proof of \eqref{eq:NoMiss}] 
Recall the definition of ${\bm\fR^n}$ and $\overline\fR$ from \cref{eq:DefStepExploitability} and \eqref{eq:DefMeanStepExploitability}, respectively. In view of \cref{lem:MFF}, below we always use $\overline\xi_t$ and $\overline\Xi$ instead of $\xi^{\overline\psi_t}$ and $\Xi^{\overline\psi_t}$. By \cref{prop:EstDiffRMF}, we have
\begin{align}\label{eq:EstDiffRMF}
\left|{\bm\fR^n}({\bm\fP}; {\bm U}) - \sum_{t=1}^{T-1} \bar{c}^t \overline\bT^{\tilde\fp^n}_{t,{\overline\Xi}} \left( \overline S^{\tilde\fp^n}_{t,{\overline\xi_t}} \circ \overline\cS^{*}_{t+1,T,{\overline\Xi}} V_{{\overline\Psi}} - \overline\cS^{*}_{t,T,{\overline\Xi}} V_{{\overline\Psi}} \right)\right| \le {\fE}.
\end{align}
Therefore, it is sufficient to show
\begin{align}\label{eq:IneqSuff}
\frac{1}{N}\sum_{n=1}^N\sum_{t=1}^{T-1} \bar{c}^t \overline\bT^{\tilde\fp^n}_{t,{\overline\Xi}} \left( \overline S^{\tilde\fp^n}_{t,{\overline\xi_t}} \circ \overline\cS^{*}_{t+1,T,{\overline\Xi}} V_{{\overline\Psi}} - \overline\cS^{*}_{t,T,{\overline\Xi}} V_{{\overline\Psi}} \right) \ge \overline\fR(\overline\Psi;V).
\end{align}
To this end notice that, by \cref{lem:MeanPsiMFF},
\begin{align}\label{eq:IndAveMeanfTS} 
\frac1N\sum_{n=1}\overline\bT^{\tilde\fp^n}_{t,{\overline\Xi}} \circ \overline\cS^{*}_{t,T,{\overline\Xi}}V_{{\overline\Psi}} = \int_{\bX} \overline\cS^{*}_{t,T,{\overline\Xi}}V_{{\overline\Psi}}(y)\,\xi^{{\overline\psi}_t}(\dif y) =  \overline\bT^{\overline\fp^{\overline\Psi}}_{t,\overline\Xi} \circ  \overline{\cS}^{*}_{t,T,\overline\Xi} V_{{\overline\Psi}}
\end{align}
Next, let $[N]:=\set{1,\dots,N}$ and consider the probability space $([N]\times\bX,2^{[N]}\otimes\cB(\bX),\rho)$, where $\rho$ satisfies $\rho(\set{n}\times B) := \frac1N \overline\bT^{\tilde\fp^n}_{t,\overline\Xi}\1_B$ for any $B\in\cB(\bX)$. We write $\rho^n(B):=\rho(\set{n}\times B)$. Let $q$ be the regular conditional probability of $Z(n,x):=n$ given $\set{[N],\emptyset}\otimes\cB(\bX)$. Note $q$ is constant in $n$ and $q_x$ can be treated as a $N$-dimensional probability vector. We subsequently write $q^n(x):=q_x(\set{n})$. By partial averaging property\footnote{In terms of expectation under $\rho$, it writes $\bE\left(H \1_{\set{n}}(Z) \right) = \bE\left(H \bE\left(\1_{\set{n}}(Z)\big|\set{[N],\emptyset}\otimes\cB(\bX)\right) \right)$, where $H(n,x)=h(x)$.} 
and \cref{lem:MeanPsiMFF}, for $n\in [N]$, $q^n$ satisfies
\begin{align}\label{eq:IndPAP}
\frac1N \overline\bT^{\tilde\fp^n}_{t,\overline\Xi}h = \sum_{k=1}^N \int_{\bX}h(y) q^n(y) \rho^k(\dif y)  = \int_{\bX}h(y) q^n(y) \overline\xi_t(\dif y),\; h\in B_b(\bX).
\end{align}
For $\tilde\pi^1,\dots,\tilde\pi^N\in\wt\Pi$, let $\left[\sum_{n=1}^N q^n\tilde\pi^n\right]_{y,\xi}:=\sum_{n=1}^N q^n(y)\tilde\pi^n_{y,\xi}$ and note that it is $\cB(\bX\times\cP(\bX))$-$\cE(\cP(\bA))$ measurable. Recall the definition of $\overline S^{\tilde{\bm\pi}}_{t,\xi}$ from \eqref{eq:DefMeanS}. It follows from \eqref{eq:IndPAP}, \cref{assump:GBasic} (vi) and generalized Jensen's inequality \cite{Ting1975Generalized} that 
\begin{align}\label{eq:EstAveMeanfTS}
\frac1N\sum_{n=1}^N \overline\bT^{\tilde\fp^n}_{t,\overline\Xi} \circ \overline S^{\tilde\fp^n_t}_{t,\overline\xi_t} \circ \overline\cS^{*}_{t+1,T,\overline\Xi}V_{{\overline\Psi}} &= \sum_{n=1}^N\int_{\bX}   \overline S^{\tilde\fp^n_t}_{t,\overline\xi_t} \circ \overline\cS^{*}_{t+1,T,\overline\Xi}V_{{\overline\Psi}}(y)\; q^n(y) \overline\xi_t(\dif y)\nonumber\\
&\quad\ge \int_{\bX} \overline S^{\sum_{n=1}^N q^n \tilde\fp^n_t}_{t,\xi^{{\overline\psi}_t}} \circ \overline\cS^{*}_{t+1,T,\overline\Xi}V_{{\overline\Psi}}(y) \overline\xi_t(\dif y).
\end{align}
On the other hand, note that for any $B\in\cB(\bX)$,
\begin{align*}
&\int_{\bX} \1_{B}(y) \int_{\bA} \1_{A}(a) \left[\sum_{n=1}^N q^n(y)\tilde\fp^n_{t,y,\overline\xi_t}\right](\dif a)\, \overline\xi_t(\dif y)\\
&\quad= \sum_{n=1}^N \int_{\bX} \1_{B}(y)\tilde\fp^n_{t,y,\overline\xi_t}(A)\; q^n(y)  \overline\xi_t(\dif y) = \frac1N\sum_{n=1}^N \int_{\bX} \1_{B}(y)\tilde\fp^n_{t,y,\overline\xi_t}(A)\; \overline\bQ^{\tilde\fp^n}_{t,\overline\Xi}(\dif y) = {\overline\psi}_t(B\times A),
\end{align*}
where in the second equality we have used \eqref{eq:IndPAP} and definition of $\overline\bQ^{\tilde\fp^n}_{t,\Xi}$ from below \eqref{eq:DefMeanfT}, and the last equality follows from the definition of $\overline\Psi$ in \eqref{eq:DefPsiBar}. Therefore, by \cref{lem:pipsiUnique},  $y\mapsto\left[\sum_{n=1}^N q^n(y)\tilde\fp^n_{t,y,\overline\xi_t}\right]$ is a version of $\overline\pi^{{\overline\psi}_t}$. This together with \eqref{eq:EstAveMeanfTS}, \cref{lem:VersionIndpi} and \cref{lem:MFF} implies that 
\begin{align}\label{eq:EstAveMeanfTSRe} 
\frac1N\sum_{n=1} \overline\bT^{\tilde\fp^n}_{t,\overline\Xi} \circ \overline S^{\tilde\fp^n}_{t,\overline\xi_t} \circ \overline\cS^{*}_{t+1,T,\overline\Xi}V_{{\overline\Psi}} \ge \overline\bT^{\overline\fp^{\overline\Psi}}_{t,\overline\Xi} \circ \overline S^{\overline\pi^{\overline{\psi}_t}}_{t,\overline\xi_t} \circ \overline{\cS}^{*}_{t+1,T,\overline\Xi} V_{{\overline\Psi}}.
\end{align}
Finally, by combining \eqref{eq:IndAveMeanfTS} and \eqref{eq:EstAveMeanfTSRe}, we have verified \eqref{eq:IneqSuff}.
\end{proof}

\subsection{Proof of \cref{thm:MFConstr}}\label{Proofthm:MFConstr}
\begin{proof}[Proof of \cref{thm:MFConstr}]
To start with note that $\overline{\bm\fP}^\Psi$ is $0$-symmetrically continuous. Let $\overline\Xi$ and ${\overline\Psi}$ be induced by $\overline{\bm\fP}^\Psi$ as in \cref{subsec:Scenarios} but with $\wt{\bm\fP}=\overline{\bm\fP}^{\Psi}$.

We claim that ${\overline\Psi}=\Psi$.
Recall the definition of $\overline\bQ^{\tilde\fp}_{t,\overline\Xi}$ from below \eqref{eq:DefMeanfT}. By \cref{lem:MFF}, we have $\overline\bQ^{{\overline\fp^\Psi}^n}_{t,{\Xi^{\Psi}}}=\xi^{\psi_{t}}$ for $n=1,\dots,N$ and $t=1,\dots,T-1$.  It follows from \eqref{eq:Defxibar} and induction that $\overline\xi_{t}={\xi^{\psi_t}}$ for all $n$ and $t$. In addition, by \cref{lem:MeanPsiMFF}, $\xi^{\overline\psi_t}=\overline\xi_t=\xi^{\psi_t}$ and ${\overline\Psi}$ is a mean field flow. This together with the homogeneous policies and \eqref{eq:DefPsiBar} implies $\overline\psi_t(B\times A)=\psi_t(B\times A)$ for $B\in\cB(\bX)$ and $A\in\cB(\bA)$ for each $t$. An application of monotone class lemma (\cite[Section 4.4, Lemma 4.13]{Aliprantis2006book}) shows $\overline\psi_t=\psi_t$ for each $t$, and thus ${\overline\Psi}=\Psi$.

Recall the definitions of $\bm\fR^n$ and $\overline\fR$ from \eqref{eq:DefPermExploitabilities} and \eqref{eq:DefMeanStepExploitability}, respectively, invoking \cref{prop:EstDiffRMF}, we yield 
\begin{align*}
\left|{\bm\fR^n}(\overline{\bm\fP}^\Psi;{\bm U}) - \overline\fR(\Psi;V)\right| = \left|{\bm\fR^n}(\overline{\bm\fP}_\Psi;{\bm U}) - \overline\fR({\overline\Psi};V)\right| \le {\fE^0},
\end{align*}
where $\fE^0$ is introduced in \cref{subsec:ErrorTerms}.

In order to finish the proof, by \eqref{eq:DeffS}, \eqref{eq:DeffSstar}, \eqref{eq:DefMeanfS} and \eqref{eq:DefMeanfSstar}, we have
\begin{align*}
\left|\bm{\cR^n}(\overline{\bm\fP}^\Psi;{\bm U}) - \overline\cR(\Psi;V)\right| &\le \left|\mathring{\bm G} \circ\bm\cS^{\overline{\bm\fP}^\Psi}_{1,T}{\bm U} - \mathring{\bm G} \circ \overline\cS^{{\overline\fp^\Psi}}_{T,{\Xi^{\Psi}}}  \right| + \left| \mathring{\bm G} \circ \bm\cS^{*,\overline{\bm\fP}^\Psi}_{1,T}{\bm U} - \mathring{\bm G} \circ \overline\cS^{*}_{T,{\Xi^{\Psi}}} \right|.
\end{align*}
Due to what is proved above, we can replace $\Xi^\Psi$ with $\overline\Xi$. Invoking \cref{assump:GBasic} (iii), \cref{lem:EstfTDiffcS} and \cref{prop:EstfTDiffcSstar} (with $\vartheta^1\equiv0$ and $\vartheta\equiv 0$ thus $\fe_t=\fe^0_t$), the proof is complete.
\end{proof}

\subsection{Proofs of \cref{thm:ZeroExploitability}}\label{subsec:ProofMFEandOthers}

To prepare for the proof of \cref{thm:ZeroExploitability}, we introduce a few more notations. Note that $\cP(\bA)$ with weak topology is a Polish space because $\bA$ is (cf. \cite[Section 15.3, Theorem 15.15]{Aliprantis2006book}). It follows that $\bX\times\cP(\bA)$ is also Polish. We let $\cP(\bX\times\cP(\bA))$ be the set of probability measure on $\cB(\bX)\otimes\cB(\cP(\bA))$ endowed with weak topology. Consider $\fm=(\mu_1,\dots,\mu_{T-1})\in\cP(\bX\times\cP(\bA))^{T-1}$. We equip $\cP(\bX\times\cP(\bA))^{T-1}$ with product topology. 

For $\mu\in\cP(\bX\times\cP(\bA))$, we let ${\psi^\mu}\in\cP(\bX\times\bA)$ satisfy  
\begin{align}\label{eq:Defpsifu}
{\psi^\mu}(B\times A) = \int_{\bX\times\cP(\bA)} \int_{\bA} \1_{B}(y) \1_{A}(a) \lambda(\dif a)  \mu(\dif y \dif\lambda), \quad B\in\cB(\bX),\, A\in\cB(\bA),
\end{align}
and denote ${\Psi^\fm} = (\psi^{\mu_1},\dots,\psi^{\mu_{T-1}})$. Thanks to Carath\'eodory extension theorem (cf. \cite[Section 10.4, Theorem 10.23]{Aliprantis2006book}), it is valid to define a measure on $\cB(\bX\times\bA)$ by imposing \eqref{eq:Defpsifu}. Using monotone class lemma \cite[Section 4.4, Lemma 4.13]{Aliprantis2006book}, simple function approximation \cite[Section 4.7, Theorem 4.36]{Aliprantis2006book} and monotone convergence \cite[Section 11.4, Theorem 11.18]{Aliprantis2006book}, for any $f\in B(\bX\times\bA)$ we have
\begin{align}\label{eq:PsiUpsilon}
\int_{\bX\times\bA}f(x,a){\psi^\mu}(\dif x\dif a) = \int_{\bX\times\cP(\bA)} \int_{\bA} f(x,a) \lambda(\dif a)  \mu(\dif a \dif\lambda).
\end{align}
It is well-known that $\cP(\bX\times\cP(\bA))^{T-1}$ endowed with product of weak topologies is a locally convex topological vector space \cite[Section 5.13, Theorem 5.73]{Aliprantis2006book}.

Recall the notations introduced in \cref{assump:QTightsigma}. In what follows, we let $\check K_{t,i} := \check K_{\lceil\check c^t i\rceil}$ and $\check\fK_t:=(\check K_{t,i})_{i\in\bN}$. We let ${\bM}_0$ be a subset of $\cP(\bX\times\cP(\bA))^{T-1}$ such that for any $\fm=(\mu_1,\dots,\mu_{T-1})\in{\bM}_0$. Note that $\xi^{\mu_1}=\mathring\xi$ and $\xi^{\mu_t}$ is $\check\fK_t$-tight for $t\ge 2$ by \cref{lem:sigmaTight}.\footnote{We continue using $\xi$ for the marginal measure on $\cB(\bX)$.} In addition, under \cref{assump:ActionDomain}, due to the compactness of $\cP(\bA)$, we yield that $\mu_t$ is $\check\fK\times\set{\cP(\bA)}$-tight, for $t=1,\dots,T-1$. It follows that $\bM_0$ is $(\check\fK_T\times\set{\cP(\bA)})^{T-1}$-uniformly tight, and thus pre-compact due to Prokhorov's theorem \cite[Section 15.5. Theorem 15.22]{Aliprantis2006book}. Therefore, the completion of $\bM_0$, denoted by ${\bM}$, is compact. We also note that ${\bM}_0$ is convex and so is ${\bM}$. 

In what follows, we consider $\fm=(\mu_1,\dots,\mu_{T-1})$ and $\fn=(\nu_1,\dots,\nu_{T-1})$. For $\fm\in{\bM}$, we define
\begin{align}
\Gamma\fm &:= \bigg\{\fn\in{\bM}:\;  \xi^{\nu_{t+1}}(B) = \int_{\bX\times\cP(\bA)}\int_{\bA}P_{t,x,\xi^{\mu_t},a}(B)\lambda(\dif a)\mu_t(\dif x\dif\lambda), B\in\cB(\bX),\label{eq:IndGammaMarginal}\\
&\qquad \int_{\bX\times\cP(\bA)} \left(\overline G^{\lambda}_{t,\xi^{\mu_t}} \circ \overline{\cS}^{*}_{t+1,T,{\Xi^\fm}} V_{\fm}(y) - \overline{\cS}^{*}_{t,T,{\Xi^\fm}} V_{\fm}(y)\right) \nu_t(\dif y \dif\lambda) = 0  \bigg\}, \label{eq:IndGammaOpt}
\end{align}
where we recall \eqref{eq:DefVM} and \eqref{eq:PsiUpsilon}, and define with a sight abuse of notation
\begin{align}\label{eq:DefVfu}
V_\fm(x) := V_{{\Psi^\fm}}(x) = V\left(x,\int_{\bX\times\cP(\bA)} \int_{\bA} P_{T-1,y,\xi_{\mu_{T-1}},a}(\cdot)\lambda(\dif a)\; \mu_{T-1}(\dif y\dif\lambda)\right).
\end{align}
We are aware that $\Psi^{\fn}$ need not to be a mean field flow even if $\fn\in\Gamma\fm$.

The lemmas below will prove useful.
\begin{lemma}\label{lem:JointContGSS}
Suppose \cref{assump:ActionDomain}, \cref{assump:GBasic} (i) (ii) (iii), \cref{assump:PCont2} and \cref{assump:GCont2}. Let $V\in C_b(\bX\times\cP(\bX))$. Then, $(x,\mu_{T-1})\mapsto V_\fm(x)$ is continuous. Moreover, for $t=1,\dots,T-1$, both
\begin{gather*}
(\xi,\lambda,\mu_{t+1},\dots,\mu_{T-1},x)\mapsto\overline G^{\lambda}_{t,\xi} \circ \overline{\cS}^{*}_{t+1,T,{\Xi^\fm}} V_{\fm}(x)\\
(\mu_{t},\dots,\mu_{T-1},x)\mapsto \overline{\cS}^{*}_{t,T,{\Xi^\fm}} V_{\fm}(x)
\end{gather*}
are continuous.
\end{lemma}

\begin{proof}
In what follows, we consider $\fm^k$ and $\fm^0$ such that $\mu^k_{t}$ converge weakly to $\mu^0_{t}$ as $k\to\infty$.\footnote{In the current setting where we aim to prove the existence of MFE, superscript is no longer related to players in the $N$pG.} Note that weak convergence of a sequence of joint probability measures implies weak convergence of the marginal probability measures. Therefore, $\lim_{k\to\infty}\xi^{\mu^k_t}=\xi^{\mu^0_t}$. 

We first show the continuity of $(x,\mu_{T-1})\mapsto V_\fm(x)$. Note that under \cref{assump:PCont2}, for any $h\in C_b(\bX)$, $(y,\xi,a)\mapsto\int_{\bX} h(z) P_{T-1,y,\xi,a}(\dif z)$ is continuous. By \cref{lem:ConvInvVaryingMeas}, we yield the continuity of $(y,\xi,\lambda)\mapsto\int_{\bA} \int_{\bX} h(z) P_{T-1,y,\xi,a}(\dif z) \lambda(\dif a)$. By \cref{lem:ConvInvVaryingMeas} again, we have 
\begin{align*}
&\lim_{k\to\infty}\int_{\bX\times\cP(\bA)}\int_{\bA} \int_{\bX} h(z)P_{T-1,y,\xi^{\mu_{T-1}^k},a}(\dif z)\; \lambda(\dif a)\; \mu_{T-1}^k(\dif y\dif\lambda)\\
&\quad= \int_{\bX\times\cP(\bA)}\int_{\bA} \int_{\bX} h(z)P_{T-1,y,\xi^{\mu_{T-1}^0},a}(\dif z)\; \lambda(\dif a)\; \mu_{T-1}^0(\dif y\dif\lambda).
\end{align*}
This proves the continuity of $(x,\mu_{T-1})\mapsto V_\fm(x)$. 

Next, we proceed by backward induction. Suppose for some $t=1,\dots,T-1$ we have $(\mu_{t+1},\dots,\mu_{T-1},x) \mapsto \overline{\cS}^{*}_{t+1,T,{\Xi^\fm}} V_{\fm}(x)$ is continuous, where we recall $\overline{\cS}^{*}_{T,T,{\Xi^\fm}}$ is the identity operator. Below we also consider $x^k$ and $\xi^k$ such that $\lim_{k\to\infty} x^k=x^0$ and $\lim_{k\to\infty} \xi^k=\xi^0$. Then, by  \cref{assump:GBasic} (iii),
\begin{align*}
&\left| \overline G^{\lambda^0}_{t,\xi^0} \circ \overline{\cS}^{*}_{t+1,T,{\Xi^{\fm^0}}} V_{\fm^0}(x^0) - \overline G^{\lambda^k}_{t,\xi^k} \circ \overline{\cS}^{*}_{t+1,T,{\Xi^{\fm^k}}} V_{\fm^k}(x^k) \right|\\
&\quad\le \left| \overline G^{\lambda^0}_{t,\xi^0} \circ \overline{\cS}^{*}_{t+1,T,{\Xi^{\fm^0}}} V_{\fm^0}(x^k) - \overline G^{\lambda^k}_{t,\xi^k} \circ \overline{\cS}^{*}_{t+1,T,{\Xi^{\fm^0}}} V_{\fm^0}(x)(x^k) \right| \\
&\qquad+ \bar{c} \int_{\bX} \left| \overline{\cS}^{*}_{t+1,T,{\Xi^{\fm^0}}} V_{\fm^0}(y) - \overline{\cS}^{*}_{t+1,T,{\Xi^{\fm^k}}} V_{\fm^k}(y) \right| Q^{\lambda^k}_{t,x^k,\xi^k}(\dif y).
\end{align*}
Due to the induction hypothesis, Assumption \ref{assump:GCont2}, \cref{lem:QCont2} and \cref{lem:ConvInvVaryingMeas}, the right hand side above vanishes as $k\to\infty$. We have shown the continuity of $(\xi,\lambda,\mu_{t+1},\dots,\mu_{T-1},x)\mapsto\overline G^{\lambda}_{t,\xi} \circ \overline{\cS}^{*}_{t+1,T,{\Xi^\fm}} V_{\fm}(x)$. 

Finally, by combining the above with \cref{assump:ActionDomain} and \cite[Section 17.5, Lemma 17.29 and Lemma 17.30]{Aliprantis2006book}, we obtain the continuity of $(\mu_{t},\dots,\mu_{T-1},x)\mapsto \overline{\cS}^{*}_{t,T,{\Xi^\fm}} V_{\fm}(x)$.
\end{proof}

\begin{proposition}\label{prop:GammaFixedpt}
Suppose \cref{assump:ActionDomain}, \cref{assump:GBasic} (i) (ii) (iii), \cref{assump:QTightsigma}, \cref{assump:PCont2} and \cref{assump:GCont2}. Then, the set of fixed points of $\Gamma$ is non-empty and compact.
\end{proposition}
\begin{proof}
We will apply Kakutani–Fan–Glicksberg fixed point theorem \cite[Section 17.9, Corollary 17.55]{Aliprantis2006book}. Below we fix $\fm\in{\bM}$ arbitrarily and verify the conditions of the theorem. 

Note that ${\bM}$ is compact by definition.

We then argue that $\Gamma\fm$ is non-empty. Thanks to \cref{assump:QTightsigma} and \cref{lem:sigmaTight}, the right hand side of \eqref{eq:IndGammaMarginal}, as a probability on $\cB(\bX)$, is $\check\fk_{t+1}$-tight. In view of Lemma \ref{lem:JointContGSS} and measurable maximum theorem (cf. \cite[Section 18.3, Theorem 18.19]{Aliprantis2006book}), let $\overline\pi^*_t:\bX\to\cP(\bA)$ be an optimal selector that attains $\inf_{\lambda\in\cP(\bA)}\overline G^{\lambda}_{t,\xi^{\mu_t}} \circ \overline{\cS}^{*}_{t+1,T,{\Xi^\fm}} V_{\fm}(y)$ for $y\in\bX$. In view of Carath\'eodory extension theorem \cite[Section 10.4, Theorem 10.23]{Aliprantis2006book}, let $\nu_t\in\cP(\bX\times\cP(\bA))$ be characterized by $B\times\Lambda \mapsto \int_{\bX} \1_{B}(x)\1_{\Lambda}(\overline\pi^*_t(x)) \xi_t(\dif x)$ for $B\in\cB(\bX)$ and $\Lambda\in\cB(\cP(\bA))$. Consequently, this $\nu_t$ belongs to $\bM$ and satisfies both \eqref{eq:IndGammaMarginal} and \eqref{eq:IndGammaOpt}.  

Next, we verify that $\Gamma\fm$ is convex. To see this, let $\gamma\in(0,1)$, and $\fn^1,\fn^2 \in \Gamma\fm$. We consider a $\fn$ satisfying $\nu_t=\gamma\nu_t^1+(1-\gamma)\nu_t^2$ for $t=1,\dots,T-1$. Clearly, $\nu_t$ satisfies \eqref{eq:IndGammaMarginal} as \eqref{eq:IndGammaMarginal} specifies the state marginal uniquely. It follows from the linearity with respect to the integrating measure that $\nu_t$ also satisfies \eqref{eq:IndGammaOpt}.

Finally, we show that $\Gamma$ has a closed graph. To this end let $(\fm^k)_{k\in\bN}\subseteq{\bM}$, $\fn^k\in\Gamma\fm^k$ for $k\in\bN$, and suppose $\fm^k$ and $\fn^k$ converges to $\fm^0$ and $\fn^0$, respectively, as $k\to\infty$. For any $h\in C_b(\bX)$, invoking \cref{assump:PCont2} and \eqref{eq:IndGammaMarginal}, then applying \cref{lem:ConvInvVaryingMeas} twice, we yield
\begin{align*}
\int_{\bX}h(y)\xi^{\nu^{0}_{t+1}}\!(\dif y) &= \lim_{k\to\infty}\int_{\bX}h(y)\xi_{\nu^{k}_{t+1}}\!(\dif y) = \lim_{n\to\infty} \int_{\bX\times\cP(\bA)}\int_{\bA} \int_{\bX}h(y)P_{t,x,\xi^{\mu^k_t},a}(\dif y)\lambda(\dif a)\mu^k_t(\dif x\dif\lambda)\\
&= \int_{\bX\times\cP(\bA)}\int_{\bA} \int_{\bX}h(y)P_{t,x,\xi^{\mu^0_t},a}(\dif y)\lambda(\dif a)\mu^0_t(\dif x\dif\lambda),
\end{align*}
i.e., \eqref{eq:IndGammaMarginal} is also true for $\fm^0$ and $\fn^0$. By combining \cref{lem:JointContGSS} and \cref{lem:ConvInvVaryingMeas}, we have verified \eqref{eq:IndGammaOpt} for $\fm^0$ and $\fn^0$. The proof is complete.
\end{proof}

Let $Y(x,\lambda):=\lambda$ for $(x,\lambda)\in\bX\times\cP(\bA)$ and $Z^\mu$ be the Bochner conditional expectation of $Y$ given $\cB(\bX)\otimes\set{\emptyset,\cP(\bA)}$ under $\mu\in\cP(\bX\times\cP(\bA))$. We refer to \cite[Chapter 2, Theorem 3.1 and Definition 3.2]{Scalora1958Abstract} for the validity and definition of Bochner conditional expectation. Clearly, $Z^\mu$ is constant in $\lambda\in\cP(\bA)$. Recall the definition of induced action kernel from \cref{subsec:MFF}. The next lemma will be useful.
\begin{lemma}\label{lem:UpsilonInducedpi}
We have $\overline\pi^{{\psi^\mu}}=Z^\mu$ for $\xi_{\mu}$ almost every $x\in\bX$.
\end{lemma}
\begin{proof}
In view of  \cite[Chapter 2, Theorem 3.1]{Scalora1958Abstract}, it is sufficient to prove the partial averaging principle in Bochner's sense,
\begin{align*}
\int_{B\times\cP(\bA)} Y(y,\lambda) \mu(\dif y \dif\lambda) = \int_{B\times\cP(\bA)} \overline\pi^{{\psi^\mu}}_y \mu(\dif y \dif\lambda) = \int_{B} \overline\pi^{{\psi^\mu}}_y \xi^{{\psi^\mu}}(\dif y),\quad B\in\cB(\bX),
\end{align*}
where the second equality is due to the fact that $y\mapsto\overline\pi^{{\psi^\mu}}_y$ is constant in $\lambda$. The integrals above are indeed well-defined in Bochner's sense due to \cite[Section 11.8, Theorem 11.44]{Aliprantis2006book}. To proceed, notice that, for any $h\in B_b(\bA)$,
\begin{align*}
&\int_{\bA} h(a)\left[\int_{B} \overline\pi^{{\psi^\mu}}_y \xi^{{\psi^\mu}}(\dif y)\right](\dif a) = \int_{B} \int_{\bA}h(a) \overline\pi^{{\psi^\mu}}_y (\dif a) \xi_{{\psi^\mu}}(\dif y) = \int_{B\times\bA} h(a) {\psi^\mu}(\dif y \dif a) \\
&\quad = \int_{\bX\times\cP(\bA)} \int_{\bA} h(a) \lambda(\dif a)  \mu(\dif y \dif\lambda) = \int_{\bA} h(a)\left[\int_{B\times\cP(\bA)} Y(y,\lambda) \mu(\dif y \dif\lambda)\right](\dif a),
\end{align*}
where we have used \cite[Section 11.8, Lemma 11.45]{Aliprantis2006book} in the first and last equality, \eqref{eq:PsiUpsilon} in the third. Since $h\in B_b(\bA)$ is arbitrary, the proof is complete.
\end{proof}

We are in the position to prove \cref{thm:ZeroExploitability}.
\begin{proof}[Proof of \cref{thm:ZeroExploitability}]
In view of \cref{prop:GammaFixedpt}, we consider $\fm^*\in{\bM}$ such that $\fm^*=\Gamma\fm^*$ and let $\Psi^*=\Psi^{\fm^*}$ in the sense of \eqref{eq:Defpsifu}. We will show that $\Psi^*$ is a MFF with zero exploitability. 

To start with, note that, by definition, $\xi^{\psi^*_t}=\xi^{\psi^{\mu^*_t}}=\xi^{\mu^*_t}$ for $t=1,\dots,T-1$. Then, by \eqref{eq:IndGammaMarginal} and \eqref{eq:PsiUpsilon}, we have
\begin{align}\label{eq:ImpliesMFF}
\xi^{\psi^*_{t+1}}(B) &= \xi^{\mu^*_{t+1}}(B) = \int_{\bX\times\cP(\bA)}\int_{\bA}P_{t,y,\xi^{\mu^*_t},a}(B)\lambda(\dif a)\mu^*_t(\dif y\dif\lambda) = \int_{\bX\times\bA} P_{t,y,\xi^{\psi^*_t},a}(B) \psi^*_t(\dif y\dif a)
\end{align}
for any $B\in\cB(\bX)$, i.e., $\Psi^*$ satisfies \eqref{eq:DefMFF} and is a mean field flow. 

Note that $\xi^{\mu^*_t}=\xi^{\psi^*_t}$, $V_{\fm^*}=V_{\Psi^*}$, and $\Xi^{\fm^*}=\Xi^{\Psi^*}$ by definition. In view of \eqref{eq:DefMeanStepExploitability} and \eqref{eq:IndGammaOpt}, it remains to show  
\begin{align}\label{eq:ImpliesOpt}
\int_{\bX\times\cP(\bA)} \overline G^{\lambda}_{t,\xi^{\psi^*_t}} \circ \overline{\cS}^{*}_{t+1,T,\Xi^{\Psi^*}} V_{\Psi^*}(y)  \mu^*_t(\dif y \dif\lambda) \ge \int_{\bX} \overline S^{\overline\pi^{\psi^*_t}}_{t,\xi^{\psi^*_t}} \circ \overline{\cS}^{*}_{t+1,T,\Xi^{\Psi^*}} V_{\Psi^*}(y) \xi^{\psi^*_t}(\dif y).
\end{align}
for arbitrarily fixed $t=1,\dots,T-1$. Let $Y(x,\lambda):=\lambda$ for $(x,\lambda)\in\bX\times\cP(\bA)$ and $Z^{\mu^*_t}$ be the Bochner conditional expectation of $Y$ given $\cB(\bX)\otimes\set{\emptyset,\cP(\bA)}$ under $\mu^*_t$. In view of \cref{assump:GBasic} (vi), by generalized Jensen's inequality \cite{Ting1975Generalized},  
\begin{align*}
\int_{\bX\times\cP(\bA)} \! \overline G^{\lambda}_{t,\xi^{\psi^*_t}} \!\! \circ\! \overline{\cS}^{*}_{t+1,T,\Xi^{\Psi^*}} V_{\Psi^*}(y)  \mu^*_t(\dif y \dif\lambda) \ge \int_{\bX} \overline G^{Z^{\mu^*_t}(y)}_{t,\xi^{\psi^*_t}} \!\! \circ\! \overline{\cS}^{*}_{t+1,T,\Xi^{\Psi^*}} V_{\Psi^*}(y)  \xi^{\psi^*_t}(\dif y).
\end{align*}
Finally, we conclude the proof by observing that, by \cref{lem:UpsilonInducedpi}, $Z^{\mu^*_t}=\overline\pi^{\psi^*_t}$ for $\xi^{\mu^*_t}=\xi^{\psi^*_t}$ almost every $x\in\bX$.
\end{proof}

\appendix

\newpage

\section{Examples}\label{sec:Exmp}

\subsection{No-one-get-it game}\label{subsec:ExmpGame}
Consider a deterministic $N$pG with $T=3$ and $\bX=\bA=\set{0,1}$.  All players start from $0$ at $t=1$. At $t=1,2$, by choosing $a\in\bA$, each player will transit to the corresponding state for the next epoch. If a player arrives at state $1$, the player must stay in state $1$ till the end of the game, i.e., state $1$ is an absorption state. At $t=2$, if a player reaches state $1$, a cost of $-1$ (i.e., a reward of $1$) incurs; $0$ otherwise. At $t=3$, each player is penalized by a cost of the size $10\,\,\times$ the portion of players on the same level. All costs are constant in actions. Players are allowed to use randomized policy. The performance criteria is the expectation of sum of the costs. 

Let us consider a scenario where all players use the following policy:
\begin{align*}
\tilde\fp_1(0,\delta_0) = 0, \quad \tilde\fp_2(0,(1-p)\delta_0+p\delta_1) = \begin{cases}0.5\delta_0+0.5\delta_1,&p=0\\ \delta_1, &p\in(0,1]\end{cases}.
\end{align*}
The above policy implies that every player chooses to stay on state $0$ for $t=2$. At $t=2$, if no player is at state $1$, all players randomly go to state $0$ or $1$ with even chances for $t=3$. But if there is at least one player at state $1$ at $t=2$, the other players  go to state $1$ at $t=3$, inducing massive penalization to all players.

In this $N$-player scenario, all players have $0$ exploitabilities, i.e., it is an equilibrium. Note that the policy can be made continuous in $p$ by interpolating over $p=0$ and $p=\frac1n$, but the modular of continuity has a steep slope of size $\frac{N}2$ near $0$. This scenario can not be well approximated in the sense of \cref{subsec:HeurDerv}. Indeed, when lifted to MFG via \eqref{eq:HeurDefMFF}, we have
\begin{align*}
\overline\psi_1=\delta_0\otimes\delta_0,\quad \overline\psi_2 = \delta_0 \otimes (0.5\delta_0 + 0.5\delta_1),\quad \overline\xi_3 = 0.5\delta_0 + 0.5\delta_1.
\end{align*}
The population distribution concentrates on state $0$ at $t=1,2$, and distributes evenly on states $0$ and $1$ at $t=3$ while the induced policy is to stay at state $0$ at $t=1,2$ and going randomly to either $0$ or $1$ with even chances. Consequently, the representative player has incentive to reach $1$ at $t=2$, as the population state distribution is unaffected in this mean field setting. The resulting mean field exploitability is $1$, indicating a vacuous mean field approximation.

\subsection{An illustrating example for \cref{assump:GBasic} (vi)}\label{subsec:ExmpConvex}

In this example, we consider a single-period ($T=2$) MFG with $\bX=\bA=\set{0,1}$. Suppose all players start from the same position $x_1=0$. Upon realizing an action $a$, a player will be transmitted to the corresponding state at $t=2$, regardless of other players, i.e., $P_{1,\bm x, a}=\delta_a$. Note that, under such setting, all (state-action) mean field flows take the form of $\delta_0\otimes\lambda$, where $\lambda=(p,1-p)\in\cP(\bA)$. For convenience, we let $p\in[0,1]$ represent the representative player's policy and the associated mean field flow. The corresponding state distribution at $t=2$ is then given by $(p,1-p)$.

Consider a MFG that imposes a congestion penalty at time $t=2$. For $i\in\bX$ and $\xi=(p,1-p)\in\cP(\bX)$, we specify the terminal cost as $V(i,\xi)=p\1_{\set{0}}(i) + (1-p)\1_{\set{1}}(i)$. Furthermore, we consider an entropic cost that violates \cref{assump:GBasic} (vi): $c_1(p)= - p\ln(p) - (1-p)\ln(1-p)$. Notably, this cost penalizes randomized actions at $t=1$. Suppose all players aim to find 
\begin{align*}
\argmin_{p\in[0,1]}\bE\left[C(p)+V(X_2,(p,1-p))\right],
\end{align*}
where $X_2\sim\textsf{Binomial}(1-p)$.

We claim that, in this MFG setting, there is no mean field equilibrium in terms of mean field flow. Indeed, for any $p\in[0,1]$
\begin{align*}
\overline\cR(p) = - p\ln(p) - (1-p)\ln(1-p) + p^2 + (1-p)^2 - \min\set{p,1-p} > 0,
\end{align*}
because both the entropic cost $C(p)$ and $p^2 + (1-p)^2 - \min\set{p,1-p}$ are non-negative for $p\in[0,1]$, and $p^2 + (1-p)^2=\min\set{p,1-p}$ only at $p=\frac12$ while $C(\frac12)=-\ln(\frac12)>0$. However, it is possible to construct $N$pE under a similar setting. For example, let $N$ be even. Then, the $N$-player scenario where player-$n$ employ action $(n\bmod2)$ deterministically is an equilibrium.

Lastly, to connect with the discussion following \cref{thm:ZeroExploitability}, we can heuristically construct an equilibrium situation in the mean field game setting, and relate this to the idea of state-action-kernel distribution. Consider a situation where half of the (infinitely many) players deterministically move to $x_2=0$, while the remaining half deterministically move to $x_2=1$. This situation is an equilibrium. Moreover, it admits a representation in $\cP(\bX\otimes\cP(\bA))$. That is,
\begin{align*}
    \delta_0 \otimes \left( \frac12 \delta_{\delta_0} + \frac12 \delta_{\delta_1} \right).
\end{align*}
Above, the $\delta_0$ to the left of $\otimes$ indicates the initial state marginal. The parenthesis to the right of $\otimes$ provides a statistical summary of the action-kernel employed by players (located at $x_1=0$): half of the players use a deterministic policy symbolized by $\delta_0$ (which entails a deterministic move to $x_2 = 0$), whereas the remaining half use a deterministic policy symbolized by $\delta_1$.

\subsection{Examples of scores operators}\label{subsec:ExmpG}
Consider a cost function $C:\bX\times\cP(\bX)\times\bA\to\bR$ with $\|C\|_\infty\le c_0$. Let $J>0$ be an integer. Consider a probability vector $(w_1,\dots,w_J)$, and $\kappa_j\in(0,1]$ for $j=1,\dots,J$. We define
\begin{align}
\mathring{\bm G}{\bm u} &:=  \sum_{j=1}^J w_j\inf_{q\in\bR}\left\{ q+\kappa_j^{-1}\int_{\bX^N} \left({\bm u}({\bm y})-q\right)_{+} \left[\bigotimes_{n=1}^N\mathring{\xi}\right](\dif{\bm y})  \right\} \label{eq:DefExmpGInit}\\
\bm G^{{\bm\lambda}}_t{\bm u}({\bm x}) &:= \int_{\bA} C(x,\overline\delta_{\bm x},a^1)\lambda^1(\dif a^1),\nonumber\\
&\quad + \int_{\bA^N}\!\!\left( \sum_{j=1}^J\!w_j\!\inf_{q\in\bR}\left\{ \resizebox{0.45\hsize}{!}{$ q+\kappa_j^{-1}\!\!\!\int_{\bX^N}  \left({\bm u}({\bm y})-q\right)_{+}\! \left[\bigotimes_{n=1}^N\! P_{t,x^n,\overline\delta_{\bm x},a^n}\right](\dif{\bm y}) $} \right\} \right) \left[\bigotimes_{n=1}^N\lambda^n\right](\dif{\bm a}\!)\label{eq:DefExmpG},
\end{align}  
and 
\begin{align}
\mathring{\overline G} v &:=  \sum_{j=1}^J w_j \inf_{q\in\bR}\left\{ q+\kappa_j^{-1}\int_{\bX} \left(v(y)-q\right)_{+} \mathring{\xi}(\dif y)  \right\},\label{eq:DefExmpMeanGInit}\\
\overline G^{\lambda}_{t,\xi} v(x) &:= \int_{\bA} c(x,\xi,a)\lambda(\dif a) + \int_{\bA}\left( \sum_{j=1}^J w_j \inf_{q\in\bR}\left\{ q+\kappa_j^{-1}\int_{\bX} \left(v(y)-q\right)_{+} P_{t,x,\xi,a}(\dif y)  \right\}\right) \lambda (\dif a).\label{eq:DefExmpMeanG}
\end{align}
Note that
\begin{align*}
{\bm u} \mapsto \inf_{q\in\bR}\left\{ \resizebox{0.5\hsize}{!}{$ q+\kappa^{-1}\int_{\bX^N}  \left({\bm u}({\bm y})-q\right)_{+} \left[\bigotimes_{n=1}^N P_{t,x^n,\overline\delta_{\bm x},a^n}\right](\dif{\bm y}) $} \right\}
\end{align*}
and 
\begin{align*}
v \mapsto \inf_{q\in\bR}\left\{ q+\kappa^{-1}\int_{\bX}  \left(v(y)-q\right)_{+} P_{t,x,\xi,a}(\dif y)  \right\}
\end{align*}
are the average value at risks (cf. \cite[Section 6.2.4]{Shapiro2021book}) of the random variables ${\bm y}\mapsto{\bm u}({\bm y})$ and $y\mapsto v(y)$ under the probabilities $\bigotimes_{n=1}^N P_{t,x^n,\overline\delta_{\bm x},a^n}$ and $P_{t,x,\xi,a}$, respectively. Additionally, $\inf_{q\in\bR}$ in $\bm G^{{\bm\lambda}}_t{\bm u}$ can be replaced by $\inf_{q\in[-\|{\bm u}\|_\infty,\|{\bm u}\|_\infty]}$.The analogue holds true for $\overline G^{\lambda}_{t,\xi} v$. Additionally, $w_j$ and $\kappa_j$ may depend on $(t,x,\xi)$, although this possibility is not considered here for simplicity.

For illustration, by setting $J=1$ and $\kappa_1=1$, we recover the case of risk neutral decision making. In particular, with definition \eqref{eq:DeffS}, we have 
\begin{align*}
\bm\bS^{\bm{\bm\fP}}_{T}{\bm u} = \bE\left[\sum_{t=1}^{T-1} C\left(X^1_t,\overline\delta_{(X^1_t,\dots,X^N_t)}, A^1_t\right) +{\bm u}\left(X^1_t,\dots,X^N_t\right)\right],
\end{align*}
where under $\bP$,
\begin{gather*}
\left(X^1_1,\dots,X^N_1\right)\sim\mathring{\xi}^{\otimes N}\,,\quad
\left(A^1_t,\dots,A^N_t\right) \sim \bigotimes_{n=1}^N\fp^n_{t,(X^1_t,\dots,X^N_t)}\,,\quad
\left(X^1_{t+1},\dots,X^N_{t+1}\right) \sim \bigotimes_{n=1}^N P_{t,X^n_t,\overline\delta_{(X^1_t,\dots,X^N_t)},A^n_t}\,.
\end{gather*}
Similarly, with definition \eqref{eq:DefMeanfS}, given $\Xi=(\xi_1,\dots,\xi_{T-1})\in\cP(\bX)^{T-1}$ and $\xi_T\in\cP(\bX)$, we have
\begin{align*}
\overline\bS^{\tilde\fp}_{T,\Xi} v = \overline\bE\left[\sum_{t=1}^{T-1}C\left(X_t,\xi_{t}, A_t\right) + v(\xi_{T})\right],
\end{align*}
where under $\overline\bP$,
\begin{gather*}
X_1\sim{\xi}_1\,,\quad
A_t\sim\tilde\fp_{t,X_t,\xi_t}\,,\quad
X_{t+1}\sim P_{t,X_t,\xi_t,A_t}\,.
\end{gather*}

For the remainder of this example, we will provide some discussion on when the aforementioned score operators satisfies the assumptions in \cref{subsec:Assumptions}.

In view of \eqref{eq:DefExmpGInit} - \eqref{eq:DefExmpMeanG}, it is straightforward to verify \cref{assump:GBasic} (i) (ii) (iv) (v) (vi). 
Below we verify \cref{assump:GBasic} (iii). It is sufficient to consider the simplified case where $J=1$ and $\kappa_1=\kappa\in(0,1]$.  Observe that 
\begin{align*}
&\left|\bm G^{{\bm\lambda}}_t{\bm u}({\bm x}) - \bm G^{{\bm\lambda}}_t{\bm u}'({\bm x})\right|\\
&\quad= \left| \int_{\bA^N}\inf_{q\in\bR}\left\{ \resizebox{0.5\hsize}{!}{$ q+\kappa^{-1}\int_{\bX^N}  \left({\bm u}({\bm y})-q\right)_{+} \left[\bigotimes_{n=1}^N P_{t,x^n,\overline\delta_{\bm x},a^n}\right](\dif{\bm y}) $} \right\} \left[\bigotimes_{n=1}^N\lambda^n\right](\dif{\bm a}) \right.\\
&\qquad - \left. \int_{\bA^N} \inf_{q\in\bR}\left\{ \resizebox{0.5\hsize}{!}{$ q+\kappa^{-1}\int_{\bX^N}  \left({\bm u}'({\bm y})-q\right)_{+} \left[\bigotimes_{n=1}^N P_{t,x^n,\overline\delta_{\bm x},a^n}\right](\dif{\bm y}) $} \right\} \left[\bigotimes_{n=1}^N\lambda^n\right](\dif{\bm a}) \right|\\
&\quad\le  \int_{\bA^N} \inf_{q\in\bR}\left\{ \resizebox{0.6\hsize}{!}{$ q+\kappa^{-1}\int_{\bX^N}  \left( \left|{\bm u}({\bm y})-{\bm u}'({\bm y})\right|-q\right)_{+} \left[\bigotimes_{n=1}^N P_{t,x^n,\overline\delta_{\bm x},a^n}\right](\dif{\bm y}) $} \right\} \left[\bigotimes_{n=1}^N\lambda^n\right](\dif{\bm a})
\end{align*}
due to the convexity of average value at risk (cf. \cite[Section 6.2.4 and Section 6.3]{Shapiro2021book}). By taking $q=0$ and invoking \eqref{eq:IndIntProdQ}, we continue to obtain
\begin{align*}
&\left|\bm G^{{\bm\lambda}}_t{\bm u}({\bm x}) - \bm G^{{\bm\lambda}}_t{\bm u}'({\bm x})\right|\\
&\quad\le  \kappa^{-1}\int_{\bA^N} \int_{\bX^N}  \left|{\bm u}({\bm y})-{\bm u}'({\bm y})\right| \left[\bigotimes_{n=1}^N P_{t,x^n,\overline\delta_{\bm x},a^n}\right](\dif{\bm y})   \left[\bigotimes_{n=1}^N\lambda^n\right](\dif{\bm a})\\
&\quad= \kappa^{-1} \int_{\bX^N} \left|{\bm u}({\bm y})-{\bm u}'({\bm y})\right| \left[\bigotimes_{n=1}^N Q^{\lambda^n}_{t,x^n,\overline\delta_{\bm x}}\right](\dif{\bm y}).
\end{align*}
Letting $\bar{c}=\kappa^{-1}$, we have verified \cref{assump:GBasic} (iii) for $\bm G^{\bm\lambda}_t$. The assumption for $\overline G$ can be verified with similar argument.

For \cref{assump:GCont} to hold, we additionally assume \cref{assump:PCont} and that, there is a subadditive modulus of continuity $\zeta_0$ such that 
$$|C(x,\xi, a) - C(x,\xi, a')| \le \zeta_0(d_\bA(a,a')),\quad (a,a')\in\bA^2,\,(x,\xi)\in\bX\times\cP(\bX).$$ 
With similar reasoning leading to \cref{lem:QCont}, for $L>1+\eta(2L^{-1})$ we have
\begin{align*}
\left|\int_\bA c(x,\xi,a) \lambda^1(\dif a) - \int_\bA c(x,\xi,a) {\lambda'}^1(\dif a)\right| \le 2c_0(L\|\lambda^1-{\lambda'}^1\|_{BL}+\zeta_0(2L^{-1})).
\end{align*} 
Next, by \cref{assump:PCont} and \cref{lem:EstIndepInt}, 
\begin{align*}
a^1\mapsto q+\kappa^{-1}\int_{\bX^N}  \left({\bm u}({\bm y})-q\right)_{+} \left[\bigotimes_{n=1}^N P_{t,x^n,\overline\delta_{\bm x},a^n}\right](\dif{\bm y})
\end{align*}
has the modular of continuity of $2\kappa^{-1}\|{\bm u}\|_\infty \eta(d_\bA(a,a'))$. This modular holds regardless of $q\in[-\|{\bm u}\|_\infty,\|{\bm u}\|_\infty]$, ${\bm x}\in\bX^N$ and $a^2,\dots,a^N\in\bA$. Morevoer, this modulus preserves after acted by $\inf_{q\in[-\|{\bm u}\|_\infty,\|{\bm u}\|_\infty]}$. 
Let $\hat\zeta$ be a subadditive modular of continuity with $\hat\zeta(\ell)\ge\max\set{\zeta_0(\ell),\eta(\ell)}$ for $\ell\in\bR_+$ (e.g., $\hat\zeta=\zeta_0 + \eta$). It follows from \eqref{eq:IndIntProdQ} and \cref{lem:QCont} that 
\begin{align*}
\left\|\bm G^{{\bm\lambda}}_t{\bm u} - \bm G^{{\bm\lambda}'}_t{\bm u}\right\|_\infty \le 2(c_0+2\kappa^{-1}\|{\bm u}\|_\infty)\inf_{L>1+\hat\zeta(2L^{-1})}\left(L\|\lambda^1-{\lambda'}^1\|_{BL} + \hat\zeta(2L^{-1})\right).
\end{align*}
We define $\zeta(\ell):=\inf_{L>1+\hat\zeta(2L^{-1})}\left\{L\ell + \hat\zeta(2L^{-1})\right\}$. Clearly, $\zeta$ is non-decreasing. Moreover, for $\ell$ small enough, letting $L(\ell)=\ell^{-\frac12}$ we yield $\zeta(\ell)\le \ell^{\frac12} + \zeta_0(2\ell^{\frac12}) \xrightarrow[\ell\to0+]{}0$. As for the subadditivity, note that $\hat\zeta$ is continuous (due to the concavity) and $L\ell + \hat\zeta(2L^{-1})$ tends to infinity as $L\to\infty$, we have $\inf_{L>1+\hat\zeta(2L^{-1})}\left\{L\ell + \hat\zeta(2L^{-1})\right\}$ is attainable for $\ell\in\bR_+$. Let $L_1^*$ and $L_2^*$ be the points of minimum for $\inf_{L>1+\hat\zeta(2L^{-1})}\left\{L\ell_1 + \hat\zeta(2L^{-1})\right\}$ and $\inf_{L>1+\hat\zeta(2L^{-1})}\left\{L\ell_2 + \hat\zeta(2L^{-1})\right\}$, respectively. Without loss of generality we assume $L_1^*\le L_2^*$. Thus,
\begin{align*}
\zeta(\ell_1)+\zeta(\ell_2) \ge L_1^*(\ell_1+\ell_2) + \zeta_0(2L_1^{*-1}) \ge \zeta(\ell_1+\ell_2).
\end{align*}
The above verifies \cref{assump:GCont} (i). \cref{assump:GCont} (ii) can be verified by similar reasoning with an additional application of triangle inequality.

Finally, to validate \cref{assump:GCont2}, we assume \cref{assump:PCont2} and that $C$ is jointly continuous. It follows immediately from \cref{assump:GCont2} and \cref{lem:ConvInvVaryingMeas} that 
\begin{align*}
(x,\xi,\lambda) \mapsto \int_{\bA} C(x,\xi,a)\lambda(\dif a)
\end{align*}
is continuous. Next, notice that $\inf_{q\in\bR}$ in \eqref{eq:DefExmpMeanG} can be replaced by $\inf_{q\in[-\|v\|_\infty,\|v\|_\infty}$. This together with \cref{assump:PCont2} and \cite[Section 17.5, Lemma 17.29 and Lemma 17.30]{Aliprantis2006book} implies that
\begin{align*}
(x,\xi,a)\mapsto\inf_{q\in\bR}\left\{ q+\kappa^{-1}\int_{\bX} \left(v(y)-q\right)_{+} P_{t,x,\xi,a}(\dif y)  \right\}
\end{align*}
is continuous. The rest follows from \cref{lem:ConvInvVaryingMeas} again.

\subsection{Ancillary example for \cref{assump:QTightsigma}}\label{subsec:ExmpQTightSigma}
Let $\bX=\bR$ and consider $f:\bX\times\cP(\bX)\times\bA\to[-1,1]$. Consider the transition kernels represented by a dummy random process $X=(X_t)_{t=1,\dots,T}$
\begin{align*}
X_{t+1} = X_t + f(X_t,\xi,a)  + Z_t,\quad t=1,\dots,T-1
\end{align*}
with $X_1=0$, where $Z_1,\dots,Z_{T-1}\stackrel{i.i.d.}{\sim}\cN(0,1)$. Additionally, let $\sigma(x) = x^2+1$ for $x\in\bR$ and $\check c=4$. Clearly, \eqref{eq:MomentCondInit} holds true. For \eqref{eq:MomentCond}, we have 
\begin{align*}
\int_{\bX}\sigma(y)P_{t,x,\xi,a}(\dif y) &= \int_{\bR} (y^2+1) \frac{1}{\sqrt{2\pi}} \exp\left(-\frac{(y+x+f(x,\xi,a))^2}{2}\right)\dif y\\
&= \int_{\bR} (y-(x+f(x,\xi,a)))^2\frac{1}{\sqrt{2\pi}} \exp\left(-\frac{(y)^2}{2}\right)\dif y + 1 \\
&= (x+f(x,\xi,a))^2 + 2 \le x^2 + 2|x| + 3 \le \check c \sigma(x).
\end{align*}

\subsection{Simplifying $\overline\fR$}\label{subsec:ExmpMeanStepExploitability}
This example compliments \cref{rmk:MFPathExploitability}.  Recall from \eqref{eq:DefS} and \eqref{eq:DefMeanStepExploitability} that
\begin{align*}
\overline{\fR}(\Psi;V) = \sum_{t=1}^{T-1} \overline{C}^t \int_{\bX} \left(\overline G^{\overline{\fp}_{\psi_t}(y)}_{t,{\xi^{\psi_t}}} \circ \overline{\cS}^{*}_{t+1,T,{\Xi^{\Psi}}} V_{\Psi}(y) - \overline{\cS}^{*}_{t,T,{\overline\Xi^{\Psi}}} V_{\Psi}(y)\right) {\xi^{\psi_t}}(\dif y).
\end{align*}
Let $\overline G$ be defined in \eqref{eq:DefExmpMeanG}. For convenience, we write $V^*_{t+1}:=\overline{\cS}^{*}_{t+1,T,{\Xi^{\Psi}}} V_{\Psi}$. Then, 
\begin{align*}
&\int_{\bX} \overline G^{{\overline\pi^{\psi_t}}(y)}_{t,{\xi^{\psi_t}}} V^*_{t+1}(y) {\xi^{\psi_t}}(\dif y)\\
&\quad= \int_{\bX}\int_{\bA} c(x,\xi,a)\left[{\overline\pi^{\psi_t}}(y)\right](\dif a){\xi^{\psi_t}}(\dif y)\\
&\qquad + \int_{\bX}\int_{\bA}\inf_{q\in\bR}\left\{ q+\kappa^{-1}\int_{\bX} \left(V^*_{t+1}(y)-q\right)_{+} \left[P_t(x,\xi,a)\right](\dif y)  \right\} \left[{\overline\pi^{\psi_t}}(y)\right](\dif a) {\xi^{\psi_t}}(\dif y)\\
&\quad= \int_{\bX\times\bA} c(x,\xi,a)\psi_t(\dif y \dif a)\\
&\qquad + \int_{\bX\times\bA}\inf_{q\in\bR}\left\{ q+\kappa^{-1}\int_{\bX} \left(V^*_{t+1}(y)-q\right)_{+} \left[P_t(x,\xi,a)\right](\dif y)  \right\} \psi_t(\dif y \dif a).
\end{align*}

\section{Technical results}\label{sec:Lemmas}

\subsection{Consequences of various assumptions}
The following dynamic programming principle is an immediate consequence of measurable maximal theorem (e.g., \cite[Section 18.3, Theorem 18.19]{Aliprantis2006book}) and the fact that $\cB(\cP(\bA))=\cE(\cP(\bA))$ due to \cref{lem:sigmaAlgBE}.
\begin{lemma}\label{lem:Sstar}
Under \cref{assump:ActionDomain}, \cref{assump:GBasic} (i) (ii) and \cref{assump:GCont}, the following is true for any $t=1,\dots,T-1$, $\xi\in\cP(\bX)$, $N\in\bN$, $\bm\pi=(\pi^1,\dots,\pi^N)\in\Pi^N$, ${\bm u}\in B_b(\bX^N)$ and $v\in B_b(\bX)$:
\begin{itemize}
\item[(a)] $\bm S^{*\bm\pi}_t{\bm u}$ is $\cB(\bX^N)$-$\cB(\bR)$ measurable and there is a $\pi^{*}:(\bX^N,\cB(\bX^N))\to(\cP(\bA),\cE(\cP(\bA)))$ such that $\bm S^{(\pi^{*},\pi^2,\dots,\pi^N)}_t{\bm u} = \bm S^{*\bm\pi}_t{\bm u}$;
\item[(b)] $\overline S^*_{t,\xi}v$ is $\cB(\bX)$-$\cB(\bR)$ measurable, and there is a $\overline\pi^*:(\bX,\cB(\bX))\to(\cP(\bA),\cE(\cP(\bA)))$ such that $\overline S^{\overline\pi^*}_{t,\xi} v = \overline S^*_{t,\xi} v$.
\end{itemize}
Alternatively, under \cref{assump:ActionDomain}, \cref{assump:GBasic} (i) (ii) and \cref{assump:GCont2}, statement (b) holds true for $v\in C_b(\bX)$.
\end{lemma}

Recall the definition of $\overline{\bQ}^{\tilde\fp}_{t,\Xi}$ from below \eqref{eq:DefMeanfT}. The next lemma pertains to the relationship between \cref{assump:QTight} and \cref{assump:QTightsigma}. 
\begin{lemma}\label{lem:sigmaTight}
\cref{assump:QTightsigma} implies \cref{assump:QTight} with $K_i=\check K_{\lceil\check c^T i\rceil},\, i\in\bN$.
\end{lemma}
\begin{proof}
We first show that $\overline{\bT}^{\tilde\fp}_{t,\Xi}\sigma \le \check c^t$ for $t=1,\cdots,T$.  For $t=1$, this follows immediately from \eqref{eq:DefMeanfT} and \eqref{eq:MomentCondInit}. We proceed by induction. Suppose the statement is true for some $t=1,\cdots,T-1$. Then, by \eqref{eq:DefMeanT} and \eqref{eq:MomentCond},
\begin{align*}
\overline{T}^{\tilde\fp}_{t,\xi_t}\sigma(x) = \int_{\bA}\int_{\bX}\sigma(y)P_{t,x,\xi,a}(\dif y)\tilde\fp_{t,x,\xi_t}(\dif a) \le \check c\int_{\bA}\sigma(x)\tilde\fp_{t,x,\xi_t}(\dif a) = \check c\sigma(x).
\end{align*}
Consequently, by \eqref{eq:DefMeanfT} again, we obtain
\begin{align*}
\overline{\bT}^{\tilde\fp}_{t+1,\Xi}\sigma = \overline{\bT}^{\tilde\fp}_{t,\Xi}\circ \overline T^{\tilde\fp}_{t,\xi_t}\sigma \le \check c\overline{\bT}^{\tilde\fp}_{t,\Xi}\sigma\le \check c^{t+1}.
\end{align*} 
The above proves $\overline{\bT}^{\tilde\fp}_{t,\Xi}\sigma \le \check c^{t}$ for $t=1,\cdots,T$. To finish the proof, we consider $h(x)=i\1_{\check K_{\lceil\check c^T i\rceil}^c}(x)$. Note that $h\le\check c^{-T}\sigma$ by the setting in \cref{assump:QTightsigma}. Due to the linearity and monotonicity of $\overline{\bT}^{\tilde\fp}_{t,\Xi}$, we have
\begin{align*}
i\; \overline{\bT}^{\tilde\fp}_{t,\Xi}\1_{\check K_{\lceil\check c^T i\rceil}^c} = \overline{\bT}^{\tilde\fp}_{t,\Xi} h \le \check c^{-T}\overline{\bT}^{\tilde\fp}_{t,\Xi}\sigma \le 1.
\end{align*}
It follows that $\overline{\bT}^{\tilde\fp}_{t,\Xi}\1_{\check K_{\lceil\check c^T i\rceil}^c} \le i^{-1}$. The proof is complete.
\end{proof}

Below are some important bounds.
\begin{lemma}\label{lem:EstcS}
Suppose \cref{assump:GBasic} (ii). For any $\xi_1,\dots,\xi_{T-1}\in\cP(\bX)$, $\tilde\fp\in\wt{\Pi}$ and $v\in B_b(\bX)$, we have
\begin{gather*}
\left\|\overline\cS^{\tilde\fp}_{s,t,(\xi_s,\dots,\xi_{t-1})} v\right\|_\infty \le \cC_{t-s}(\|v\|_\infty),\quad 1\le s<t\le T,\\
\left|\overline\bS^{\tilde\fp}_{t,(\xi_1,\dots,\xi_{t-1})} v\right| \le \cC_{t}(\|v\|_\infty),\quad t=1,\dots,T,
\end{gather*}
where $\cC_{r}(z) := c_0\sum_{k=1}^{r} c_1^{k-1} + c_1^{r} z$ for $r\in\bN_+$ and note $\cC_{r+1}(z) = c_0 + c_1 \cC_{r}(z)$.
\end{lemma}
\begin{proof}
The statement can be established through a straightforward backward inductive argument. The details are therefore omitted.
\end{proof}

\begin{lemma}\label{lem:EstDiffS}
Suppose \cref{assump:GBasic} (ii) (iii). For any  ${\bm\fP}\in\Pi^{N\times(T-1)}$, ${\bm u},{\bm u}'\in B_b(\bX^N)$, and ${\bm x}\in\bX^N$, we have
\begin{gather}
\left|\bm\cS^{{\bm\fP}}_{s,t}{\bm u} - \bm\cS^{{\bm\fP}}_{s,t}{\bm u}'\right| \le \bar{c}^{t-s} \bm\cT^{{\bm\fP}}_{s,t} \left|u -{\bm u}'\right|, \quad 1\le s\le t\le T,\label{eq:EstDiffcS}\\
\left|\bm\bS^{{\bm\fP}}_{t}{\bm u} - \bm\bS^{{\bm\fP}}_{t}{\bm u}'\right| \le \bar{c}^{t} \bm\bT^{{\bm\fP}}_{t} \left|u -{\bm u}'\right|, \quad t=1,\dots,T.\label{eq:EstDifffS}
\end{gather}
For any $\Xi=(\xi_1,\dots,\xi_{T-1})\in\cP(\bX)^{T-1}$, $\tilde\fp\in\bm{\wt\Pi}$, and $v,v'\in B_b(\bX)$, we have 
\begin{gather*}
\left|\overline\cS^{\tilde\fp}_{s,t,\Xi} v - \overline\cS^{\tilde\fp}_{s,t,\Xi} v'\right| \le \bar{c}^{t-s}\overline\bT^{\tilde\fp}_{s,t,\Xi}|v-v'|,\quad 1\le s<t\le T,\\
\left|\overline\bS^{\tilde\fp}_{t,\Xi} v - \overline\bS^{\tilde\fp}_{t,\Xi} v'\right| \le \bar{c}^t\overline\bT^{\tilde\fp}_{t,\Xi}|v-v'|,\quad t=1,\dots,T.
\end{gather*}
\end{lemma}
\begin{proof}
The statement is primarily a consequence of \cref{assump:GBasic} (iii), while \cref{assump:GBasic} (ii) is needed for the operators to be well-defined. For $t=s$, $\bm\cS^{{\bm\fP}}_{s,s}$ and $\bm\cT^{{\bm\fP}}_{s,s}$ are the identity operators and thus
\begin{align*}
\left|\bm\cS^{{\bm\fP}}_{s,s}{\bm u}({\bm x}) - \bm\cS^{{\bm\fP}}_{s,s}({\bm x}){\bm u}'({\bm x})\right| =  \bm\cT^{{\bm\fP}}_{s,s} \left|u -{\bm u}'\right| ({\bm x}).
\end{align*}
We  proceed by induction.  Suppose \eqref{eq:EstDiffcS} is true for some $1\le s\le t<T$. Then, by \eqref{eq:DefcS}
\begin{align*}
\left|\bm\cS^{{\bm\fP}}_{s,t+1}{\bm u} - \bm\cS^{{\bm\fP}}_{s,t+1}{\bm u}'\right| &= \left|\bm\cS^{{\bm\fP}}_{s,t}\circ\bm S^{(\fp^n_t)_{n=1}^N}_{t}{\bm u} - \bm\cS^{{\bm\fP}}_{s,t}\circ\bm S^{(\fp^n_t)_{n=1}^N}_{t}{\bm u}'\right| \\
&\le \bar{c}^{t-s} \bm\cT^{{\bm\fP}}_{s,t} \left|\bm S^{(\fp^n_t)_{n=1}^N}_t{\bm u} - \bm S^{(\fp^n_t)_{n=1}^N}_t{\bm u}'\right| \le \bar{c}^{t+1-s} \bm\cT^{{\bm\fP}}_{s,t+1} \left|{\bm u} -{\bm u}'\right|,
\end{align*}
where we have used the induction hypothesis in the first inequality, \cref{assump:GBasic} (iii) and \eqref{eq:DeffT} in the second inequality. This proves \eqref{eq:EstDiffcS}. By combining \eqref{eq:EstDiffcS} and \cref{assump:GBasic} (iii), we yield \eqref{eq:EstDifffS}. A similar reasoning finishes the rest of the proof for score operators in mean field settings.
\end{proof}

The next two lemmas are consequence of the continuity assumed in \cref{assump:PCont}.
\begin{lemma}\label{lem:QCont}
Under \cref{assump:PCont}, for any $h\in B_b(\bX)$, $\xi,\xi'\in\cP(\bX)$, $\lambda,\lambda'\in\cP(\bA)$ and $L\ge 1+\eta(2L^{-1})$, we have
\begin{align*}
\left|\int_{\bX}h(y)Q_{t,x,\xi}^{\lambda}(\dif y) - \int_{\bX}h(y)Q_{t,x,\xi'}^{\lambda'}(\dif y)\right| \le 2\|h\|_\infty\left( L\|\lambda-\lambda'\|_{BL} + \eta(2L^{-1}) + \eta(\|\xi-\xi'\|_{BL}) \right).
\end{align*}
\end{lemma}

\begin{proof}
Without loss of generality, we assume $\|h\|_\infty=1$. Notice that
\begin{align*}
&\left|\int_{\bX}h(y)Q_{t,x,\xi}^{\lambda}(\dif y) - \int_{\bX}h(y)Q_{t,x,\xi'}^{\lambda'}(\dif y)\right|\\
&\quad=\left|\int_{\bA}\int_{\bX}h(y)P_{t,x,\xi,a}(\dif y)\lambda(\dif a) - \int_{\bA}\int_{\bX}h(y)P_{t,x,\xi',a}(\dif y)\lambda'(\dif a)\right|\\
&\quad\le \left|\int_{\bA}\int_{\bX}h(y)P_{t,x,\xi,a}(\dif y)\lambda(\dif a) - \int_{\bA}\int_{\bX} h(y)P_{t,x,\xi,a}(\dif y)\lambda'(\dif a)\right|\\
&\qquad+ \left|\int_{\bA}\int_{\bX}h(y)P_{t,x,\xi,a}(\dif y)\lambda'(\dif a) - \int_{\bA}\int_{\bX} h(y)P_{t,x,\xi',a}(\dif y)\lambda'(\dif a)\right|\\
&\quad=: I_1 + I_2.
\end{align*}
Regarding $I_1$, we define $$H_{x,\xi}(a):=\int_{\bX}h(y)P_{t,x,\xi,a}(\dif y).$$ By \cref{assump:PCont}, we have $\left|H_{x,\xi}(a)-H_{x,\xi}(a')\right| \le \eta(d_{\bA}(a,a'))$, i.e., $H_{x,\xi}$ has a modular of continuity being $\eta$. For $L\ge 1+\eta(2L^{-1})$, we let $H^L_{x,\xi}$ be defined as in \cref{lem:LipApprox} with $g=H_{x,\xi}$. Note that $\|H^L_{x,\xi}\|_\infty\le 1+\eta(2L^{-1}) \le L$. Then, 
\begin{align*}
I_1 &\le \left|\int_{\bA} H^L_{x,\xi}(a)\lambda(\dif a) - \int_{\bA} H^L_{x,\xi}(a)\lambda'(\dif a)\right| + 2\eta\left(2L^{-1}\right) \\
&\le  2\|\lambda-\lambda'\|_{L-BL} + 2\eta\left(2L^{-1}\right) = 2L\|\lambda-\lambda'\|_{BL} + 2\eta\left(2L^{-1}\right).
\end{align*}
Finally, regarding $I_2$, by \cref{assump:PCont} again we yield
\begin{align*}
I_2 &\le \sup_{x\in\bX, a\in\bA}\left|\int_{\bX} h(y)P_{t,x,\xi,a}(\dif y) - \int_{\bX} h(y)P_{t,x,\xi',a}(\dif y)\right| \le 2\eta(\|\xi-\xi'\|_{BL}). 
\end{align*}
The proof is complete.
\end{proof}

We recall the definition of $Q^\lambda_{t,x,\xi}$ from \eqref{eq:DefQ}.
\begin{lemma}\label{lem:QCont2}
Under \cref{assump:PCont2}, $(x,\xi,\lambda)\mapsto Q^{\lambda}_{t,x,\xi}$ is weakly continuous. 
\end{lemma}
\begin{proof}
Due to simple function approximation (cf. \cite[Section 4.7, Theorem 4.36]{Aliprantis2006book}) and monotone convergence, we have 
\begin{align*}
\int_{\bX}h(y)Q^{\lambda}_{t,x,\xi}(\dif y) = \int_{\bA}\int_{\bX} h(y) P_{t,x,\xi,a}(\dif y) \lambda(\dif a), \quad h\in B_b(\bX).
\end{align*}
The rest follows from \cref{lem:ConvInvVaryingMeas}.
\end{proof}

\subsection{Auxiliary technical results}


\begin{lemma}\label{lem:IntfMeasurability}
Let $(\bX,\sX)$ and $(\bY,\sY)$ be measurable spaces. Consider nonnegative $f:(\bX\times\bY,\sX\otimes\sY)\to(\bR,\cB(\bR))$ and $M:(\bX,\sX)\to(\cP,\cE(\cP))$, where $\cP$ is the set of probability measure on $\sY$ and $\cE(\cP)$ is the corresponding evaluation $\sigma$-algebra, i.e., $\cE(\cP)$ is the $\sigma$-algebra generated by sets $\set{\zeta\in\cP:\int_{\bY}f(y)\zeta(\dif y) \in B}\,$ for any real-valued bounded $\sY$-$\cB(\bR)$ measurable $f$ and $B\in\cB(\bR)$. Then, $x\mapsto\int_\bY f(x,y)\,{M_x}(\dif y)$ is $\sX$-$\cB(\bR)$ measurable. 
\end{lemma}
\begin{proof}
We first consider $f(x,y)=\1_D(x,y)$, where $D\in\sX\otimes\sY$. Let $\sD$ consist of $D\in\sX\otimes\sY$ such that $x\mapsto\int_\bY f(x,y) {M_x}(\dif y)$ is $\sX$-$\cB(\bR)$ measurable. Note $\set{A\times B:A\in\sX,B\in\sY}\subseteq\sD$ because $\int_\bY \1_{A\times B}(x,y) {M_x}(\dif y) = \1_{A}(x) \, {M_x}(B)$ and 
\begin{align*}
\set{x\in\bX:{M_x}(B)\in C} = \set{x\in\bX:{M_x}\in\set{\zeta\in\cP:\zeta(B)\in C}}\in\sX,\quad C\in\cB([0,1]).
\end{align*}
If $D^1,D^2\in\sD$ and $D^1\subseteq D^2$, then
\begin{align*}
\int_{\bY}\1_{D^2\setminus D^1}(x,y)\,{M_x}(\dif y) = \int_{\bY}\1_{D^2}(x,y)\,{M_x}(\dif y) - \int_{\bY}\1_{D^1}(x,y)\,{M_x}(\dif y)
\end{align*}
is also $\sX$-$\cB(\bR)$ measurable. Similarly, if $D^1,D^2\in\sD$ are disjoint, then $D^1\cup D^2\in\sD$. The above shows that the algebra generated by $\set{A\times B:A\in\sX,B\in\sY}$ is included by $\sD$. Now we consider an increasing sequence $(D^n)_{n\in\bN}\subseteq\sD$ and set $D^0=\bigcup_{n\in\bN}D^0$, then by monotone convergence \cite[Section 11.4, Theorem 11.18]{Aliprantis2006book}, 
\begin{align*}
\lim_{n\to\infty}\int_{\bY}\1_{D^n}(x,y)\,{M_x}(\dif y) = \int_{\bY}\1_{D^0}(x,y)\,{M_x}(\dif y),\quad x\in\bX.
\end{align*}
It follows from \cite[Section 4.6, 4.29]{Aliprantis2006book} that $D^0\in\sD$. Invoking monotone class lemma \cite[Section 4.4, Lemma 4.13]{Aliprantis2006book}, we yield $\sD=\sX\otimes\sY$.

Now let $f$ be any non-negative measurable function. Note that $f$ can be approximated by a sequence of simple function $(f^n)_{n\in\bN}$ such that $f^n\uparrow f$ \cite[Section 4.7, Theorem 4.36]{Aliprantis2006book}. Since $\sD=\sX\otimes\sY$, for $n\in\bN$, $x\mapsto\int_\bY f^n(x,y) {M_x}(\dif y)$ is $\sX$-$\cB(\bR)$ measurable. Finally, by monotone convergence \cite[Section 11.4, Theorem 11.18]{Aliprantis2006book} and the fact that pointwise convergence preserves measurability \cite[Section 4.6, Lemma 4.29]{Aliprantis2006book}, we conclude the proof.
\end{proof}


Let $\bY$ be a topological space. Let $\cB(\bY)$ be the Borel $\sigma$-algebra of $\bY$, and $\cP$ be the set of probability measures on $\cB(\bY)$. We endow $\cP$ with weak topology and let $\cB(\cP)$ be the corresponding Borel $\sigma$-algebra. Let $\cE(\cP)$ be the $\sigma$-algebra on $\cP$ generated by sets $\set{\zeta\in\cP:\zeta(A)\in B},\,A\in\cB(\bY),\,B\in\cB([0,1])$. Equivalently, $\cE(\cP)$ is the $\sigma$-algebra generated by sets $\set{\zeta\in\cP:\int_{\bY}f(y)\zeta(\dif y) \in B}\,$ for any real-valued bounded $\cB(\bY)$-$\cB(\bR)$ measurable $f$ and $B\in\cB(\bR)$. The lemma below regards the equivalence of $\cB(\cP)$ and $\cE(\cP)$.
\begin{lemma}\label{lem:sigmaAlgBE}
If $\bY$ is a separable metric space, then $\cB(\cP)=\cE(\cP)$.
\end{lemma}
\begin{proof}
Notice that the weak topology of $\cP$ is generated by sets $\set{\zeta\in\cP:\int_{\bY}f(y)\zeta(\dif y)\in U}$ for any $f\in C_b(\bY)$ and open $U\subseteq\bR$. Therefore, $\cB(\cP)\subseteq\cE(\cP)$. On the other hand, by \cite[Section 15.3, Theorem 15.13]{Aliprantis2006book}, for any bounded real-valued  $\cB(\bY)$-$\cB(\bR)$ measurable $f$ and any $B\in\cB(\bR)$, we have $\set{\zeta\in\cP:\int_{\bY}f(y)\zeta(\dif y) \in B}\in\cB(\cP)$. The proof is complete.
\end{proof}

\begin{lemma}\label{lem:ConvInvVaryingMeas}
Let $\bY$ and $\bZ$ separable metric spaces. Let $\bY\times\bZ$ be endowed with product Borel $\sigma$-algebra $\cB(\bY)\otimes\cB(\bZ)$. Let $f\in\ell^\infty(\bY\times\bZ,\cB(\bY)\otimes\cB(\bZ))$ be continuous. Let $(y^n)_{n\in\bN}\subset\bY$ converges to $y^0\in\bY$ and let $(\upsilon^n)_{n\in\bN}$ be a sequence of probability measure on $\cB(\bZ)$ converging weakly to $\upsilon^0$. Then,
\begin{align*}
\lim_{n\to\infty}\int_\bZ f(y^n,z)\, \upsilon^n(\dif z) = \int_\bZ f(y^0,z)\, \upsilon^0(\dif z).
\end{align*}
\end{lemma}
\begin{proof}
To start with, note that $(\delta_{y^n})_{n\in\bN}$ converges weakly to $\delta_{y^0}$. Therefore, by \cite[Theorem 2.8 (ii)]{Billingsley1999book}, $(\delta_{y^n}\otimes\upsilon^n)_{n\to\bN}$ converges to $\delta_{y^0}\otimes\upsilon^0$. It follows that 
\begin{align*}
\lim_{n\to\infty}\int_\bZ f(y^n,z)\, \upsilon^n(\dif z) &= \lim_{n\to\infty}\int_{\bY\times\bZ} f(y,z)\, \delta_{y^n}\otimes\upsilon^n(\dif y\,\dif z)\\
&= \int_{\bY\times\bZ} f(y,z)\, \delta_{y^0}\otimes\upsilon^0(\dif y\,\dif z) = \int_\bZ f(y^0,z)\, \upsilon^0(\dif z).
\end{align*}
The proof is complete.
\end{proof}

\begin{lemma}\label{lem:LipApprox}
Let $\bY$ be a metric space. Suppose $g:\bY\to\bR$ is a continuous function satisfying $|g(y)-g(y')| \le \beta(d(y,y'))$ for any $y,y'\in\bY$
for some non-decreasing $\beta:\overline{\bR}_+\to\overline{\bR}_+$. Then, for any $L>0$, $g_L(y):=\sup_{q\in\bA}\set{g(q) - L d(y,q)}$ is a $L$-Lipschitz continuous function $g_L$ satisfying $\|g-g_L\|_\infty \le \beta(2L^{-1}\|g\|_\infty)$.
\end{lemma}
\begin{proof}
Because
\begin{align*}
g_L(y) \ge \sup_{q\in\bA}\set{g(q) - L(d(y',q) + d_{\bA}(y,y') )}  = g_L(y') - L d_{\bA}(y,y'),
\end{align*}
$g_L$ is $L$-Lipschitz continuous $L d_{\bA}(y,y')$. Note $g\le g_L$ dy definition. Finally,
\begin{align*}
g(y) & = \sup_{q\in B_{2L^{-1}\|g\|_\infty}(y)}\left\{g(q) - (g(q)-g(y)) \right\} \ge \sup_{q\in B_{2L^{-1}\|g\|_\infty}(y)}\left\{g(q) - \beta\left(2L^{-1}\|g\|_\infty\right)\right\}  \\
&\ge \sup_{q\in B_{2L^{-1}\|g\|_\infty}(y)}\left\{g(q) - Ld_\bA(y,y')\right\} - \beta\left(2L^{-1}\|g\|_\infty\right) = g_L(y) - \beta\left(2L^{-1}\|g\|_\infty\right),
\end{align*}
where we have used in the last equality the fact that $q$ must stay within the $2L^{-1}\|g\|_\infty$-ball centered at $y$ to be within $[-\|g\|_\infty,\|g\|_\infty]$. The proof is complete.
\end{proof}

\begin{lemma}\label{lem:EstIndepInt}
Let $(\bY,\sY)$ be a measurable space. Let $m^1,\dots,m^N, {m^1}',\dots,{m^N}'$ be probability measures on $\sY$. Then, for any $\bm u\in\ell^\infty(\bY^N,\sY)$, we have 
\begin{align*}
&\left|\int_{\bY^N} \bm u({\bm y}) \left[\bigotimes_{n=1}^N m^n\right](\dif{\bm y}) -  \int_{\bY^N} f({\bm y}) \left[\bigotimes_{n=1}^N {m^n}'\right](\dif{\bm y})\right|\\
&\quad\le \sum_{n=1}^N \sup_{\set{y_1,\dots,y_N}\setminus\set{y_n}}\left|\int_\bY \bm u(y_1,\dots,y_N)m^n(\dif y_n) - \int_\bY f(y_1,\cdots,y_N){m^n}'(\dif y_n)\right|.
\end{align*}
\end{lemma}
\begin{proof}
For $N=2$, we have 
\begin{align*}
&\left|\int_{\bY^2} \bm u(y_1,y_2)m^1(\dif y_1)m^2(\dif y_2) - \int_{\bY^2} \bm u(y_1,y_2){m^1}'(\dif y_1){m^2}'(\dif y_2)\right|\\
&\quad \begin{multlined}[b] \le \left|\int_{\bY^2} \bm u(y_1,y_2)m^1(\dif y_1)m^2(\dif y_2) - \int_{\bY^2} \bm u(y_1,y_2)m^1(\dif y_1){m^2}'(\dif y_2)\right|\\
+ \left|\int_{\bY^2} \bm u(y_1,y_2)m^1(\dif y_1){m^2}'(\dif y_2) - \int_{\bY^2} \bm u(y_1,y_2){m^1}'(\dif y_1){m^2}'(\dif y_2)\right| \end{multlined}\\
&\quad \begin{multlined}[b] = \left|\int_{\bY} \left(\int_{\bY} \bm u(y_1,y_2){m^2}(\dif y_2) -  \int_{\bY} \bm u(y_1,y_2){m^2}'(\dif y_2)\right){m^1}(\dif y_1) \right|\\
 + \left|\int_{\bY} \left(\int_{\bY} \bm u(y_1,y_2)m^1(\dif y_1) -  \int_{\bY} \bm u(y_1,y_2){m^1}'(\dif y_1)\right){m^2}'(\dif y_2) \right| \end{multlined}\\
&\quad\le \sup_{y_1\in\bY}\left|\int_{\bY} \bm u(y_1,y_2)m^2(\dif y_2) -  \int_{\bY}\bm u(y_1,y_2){m^2}'(\dif y_2)\right| + \sup_{y_2\in\bY}\left|\int_{\bY} \bm u(y_1,y_2)m^1(\dif y_1) -  \int_{\bY}\bm u(y_1,y_2){m^1}'(\dif y_1)\right|.
\end{align*}
The proof for $N>2$ follows analogously by considering the corresponding telescoping sum.
\end{proof}

\section{Supplementary to \cref{sec:Setup}}\label{sec:Suppl}

\subsection{Proof of \cref{lem:pipsiUnique}}\label{subsec:Prooflem:pipsiUnique}
\begin{proof}[Proof of \cref{lem:pipsiUnique}]
Note that $\cB(\bA)=\sigma((A_n)_{n\in\bN})$ for some $A_n\subseteq\bA,\,n\in\bN$, i.e., $\cB(\bA)$ is countably generated (cf. \cite[Example 6.5.2]{Bogachev2007book}). This together with \eqref{eq:IndInthpsi} implies that 
\begin{align*}
\xi^\psi\left(\left\{x\in\bX:\overline\pi^\psi_x(A_n)=\hat\pi_x(A_n),\; n\in\bN \right\}\right) = 1.
\end{align*}
Then, with probability $1$, $\overline\pi^\psi$ and $\hat\pi$ coincide in the algrebra generated by $(A_n)_{n\in\bN}$. By monotone class lemma (\cite[Section 4.4, Lemma 4.13]{Aliprantis2006book}), we have $\overline\pi^\psi_x(A_n)=\hat\pi_x(A_n)$ for all $n\in A_n$ implies $\overline\pi^\psi_x(B)=\hat\pi_x(B)$ for all $B\in\cB(\bA)$. It follows that
\begin{align*}
\left\{x\in\bX:\overline\pi^\psi_x(A_n)=\hat\pi_x(A_n),\; n\in\bN \right\} = \left\{x\in\bX:\overline\pi^\psi_x=\hat\pi_x \right\},
\end{align*}
and thus $\xi^\psi(\set{x:\overline\pi^\psi_x=\hat\pi_x})=1$. 
\end{proof}

\subsection{Proof of \cref{lem:MFF}}\label{subsec:Prooflem:MFF}
\begin{proof}[Proof of \cref{lem:MFF}]
It is sufficient to consider $h(x)=\1_B(x)$, where $B\in\cB(\bX)$. Note for $t=1$, by \eqref{eq:DefMeanfT} and \eqref{eq:DefMFF}, we have $\overline\bT^{\overline\fp^\Psi}_{1,\emptyset}\1_B=\mathring\xi(B)=\xi_{1}(B)$. We proceed by induction. Suppose $\overline\bT^{\overline\fp^\Psi}_{t,\Xi^\Psi}\1_B=\xi^{\psi_t}(B)$ for $B\in\cB(\bX)$  for some $t=1,\dots,T-1$. By \eqref{eq:DefMeanfT} and the induction hypothesis implies that 
\begin{align*}
&\overline\bT^{\overline\fp^\Psi}_{t+1,\Xi^\Psi}\1_B =  \overline\bT^{\overline\fp^\Psi}_{t,\Xi^\Psi} \circ \overline T^{\overline\pi^{\psi_t}}_{t,\xi^{\psi_t}} \1_B = \int_{\bX} T^{\overline\pi^{\psi_t}}_{t,\xi^{\psi_t}} \1_B (y) \, \xi^{\psi_{t}}(\dif y)\\
&\quad= \int_{\bX} \int_{\bA} P_{t,y,\xi^{\psi_t},a}(B) \overline\pi^{\psi_t}_y(\dif a) \, \xi^{\psi_{t}}(\dif y) = \int_{\bX\times\bA} P_{t,y,\xi^{\psi_t},a}(B) \, \psi_t(\dif y\dif a) = \xi^{\psi_{t+1}}(B),
\end{align*}
where in the last equality we have used the fact that, for any $h'\in\ell^\infty(\bX\times\bA,\cB(\bX\times\bA))$ and $\psi\in\cP(\bX)$,
\begin{align*}
\int_{\bX\times\bA} h'(y,a) \psi(\dif y\dif a) = \int_{\bX} \int_{\bA} h'(y,a)\overline\pi^\psi_y(\dif a)\, \xi^\psi(\dif y)
\end{align*}
due to a combination of \eqref{eq:IndInthpsi}, monotone class lemma (\cite[Section 4.4, Lemma 4.13]{Aliprantis2006book}), simple function approximation (cf. \cite[Section 4.7, Theorem 4.36]{Aliprantis2006book}), and monotone convergence (cf. \cite[Section 11.4, Theorem 11.18]{Aliprantis2006book}). The proof is complete.
\end{proof}

\subsection{Proof of \cref{lem:MeanPsiMFF}}\label{subsec:Prooflem:MeanPsiMFF}
\begin{proof}[Proof of \cref{lem:MeanPsiMFF}]
The first statement that $\xi^{\overline\psi_t}=\overline\xi_t$ is an immediate consequence of the definition of $\overline\bQ^{t,\tilde\fp}_\Xi$ below \eqref{eq:DefMeanfT}. What is left to prove regards the mean field flow defined in \eqref{eq:DefMFF}:
\begin{align*}
\overline\xi_{t+1}(B) = \int_{\bX\times\bA} P_{t,y,\xi^{\overline\psi_t},a}(B)\; {\overline\psi}_t(\dif y\dif a),\quad B\in\cB(\bX).
\end{align*}
Notice that by definition,
\begin{align*}
\overline{\psi}_{t}(B\times A) = \frac1N\sum_{n=1}^N\int_{\bX} \int_{\bA} \1_B(y)\1_A(a) \tilde\fp^n_{t,y,\overline\xi_{t}}(\dif a) \;\overline{\bQ}^{\tilde{\fp}^n}_{t,\overline\Xi} (\dif y).
\end{align*}
Then, by monotone class lemma (\cite[Section 4.4, Lemma 4.13]{Aliprantis2006book}), 
simple function function approximation (cf. \cite[Section 4.7, Theorem 4.36]{Aliprantis2006book}) and monotone convergence (cf. \cite[Section 11.4, Theorem 11.18]{Aliprantis2006book}), for any $B\in\cB(\bX)$ we have
\begin{align*}
&\int_{\bX\times\bA} \left[P_t(y,{\xi^{\psi_t}},a)\right](B) \;{\overline\psi}_t(\dif y\dif a)\\
&\quad = \frac1N\sum_{n=1}^N \int_{\bX}\int_{\bA} \left[P_t(y,\overline\xi_{t},a)\right](B) \;\left[\tilde\fp^n_t(y,\overline\xi_{t})\right](\dif a)\;\overline{\bQ}^{\tilde{\fp}^n}_{t,(\overline\xi_1,\cdots,\overline\xi_{t-1})} (\dif y)\\
&\quad=  \frac1N\sum_{n=1}^N  \int_{\bX}\overline T^{\tilde\fp^n_t}_{t,\overline\xi_t}\1_B(y)\;\overline{\bQ}^{\tilde{\fp}^n}_{t,(\overline\xi_1,\cdots,\overline\xi_{t-1})} (\dif y) =  \frac1N\sum_{n=1}^N  \overline\bT^{\tilde\fp^n}_{t,(\overline\xi_1,\cdots,\overline\xi_{t-1})} \circ \overline T^{\tilde\fp^n_t}_{t,\overline\xi_t}\1_B \\
&\quad= \frac1N\sum_{n=1}^N \overline{\bQ}^{\tilde{\fp}^n}_{t+1,(\overline\xi_1,\cdots,\overline\xi_{t})}(B) = \overline\xi_{t+1}(B),
\end{align*}
where we have also used \eqref{eq:IndIntProdQ} in the first equality, and \eqref{eq:DefMeanT}, \eqref{eq:DefMeanfT} in the second last equality. This completes the proof.
\end{proof}

\subsection{Details on error terms introduced in \cref{subsec:ErrorTerms}}\label{subsec:DError}
We recall that $\fK$  are introduced in \cref{assump:QTight}, $c_0,c_1,\bar{c}$ in \cref{lem:EmpMeasConc} \cref{assump:GBasic}, $\eta$ in \cref{assump:PCont}, $\zeta$ in \cref{assump:GCont}, $\vartheta,\iota,V$ in \cref{assump:SymCont}, and $\fr_\fA$ in \cref{lem:EmpMeasConc}.  Additionally, we let $\cC_{r}(z) := c_0\sum_{k=1}^{r} c_1^{k-1} + c_1^{r} z$.

We let $\fe_1(N):=\fr_{\fK}(N)$, and
\begin{align*}
\fe_{t}(N)&:=2\inf_{L>1+\eta(2L^{-1})}\left\{\sum_{r=1}^{t-1}(L\vartheta(\fe_{r})+\eta(\fe_{r})) + \eta(2L^{-1}) \right\} + \frac{t}{N} + \fr_{\fK}(N),\quad t=2,\dots,T.
\end{align*}
Additionally, we define
\begin{align*}
{\breve\fe_t} := 2\vartheta({\fe_t}) + \frac32{\fe_t} + \fr_{\fk\times\bA}(N)-\fr_{\fk}(N),
\end{align*}
where $\fk\times\bA=(K_i\times\bA)_{i\in\bN}$. We also define
\begin{align*}
\fe^0_1(N):=\fe_1(N),\quad\text{and}\quad \fe^0_{t}(N) &:=2\sum_{r=1}^{t-1}\eta(\fe^0_r(N)) + \frac{t}{N} + \fr_{\fK}(N),\; t=2,\dots,T.
\end{align*}
Note that $\fe^0_t=\fe_{t}$ if $\vartheta\equiv 0$.

Based on $\fe_t$ and $\fe^0_t$, we further define
\begin{gather*}
\underline{\fE}(N) := \sum_{r=1}^{T-1}  \bar{c}^{r}\cC_{T-r}(\|V\|_\infty) \big(3\zeta(\vartheta(\fe_{r})) + (3+T)\zeta(\fe_{r}) \big) + (3+T)\bar{c}^{T}\iota({\fe_T}),\nonumber\\
\begin{aligned}
{\fE}(N) &:= (1+\bar{c}) \sum_{t=1}^{T-1} (T+1-t)\left( \sum_{r=t}^{T-1} \bar{c}^{r-t} \cC_{r,T}(\|V\|_\infty) \zeta({\fe_t}) + \bar{c}^{T-t}\iota({\fe_T}) \right)\nonumber\\
&\quad+ \bar{c} \iota({\fe_T}) + \sum_{t=1}^{T-1} \cC_{T-t}(\|V\|_\infty)\big(\zeta(\vartheta({\fe_t}))+\zeta({\fe_t}) + \fe_{t}(N)\big),\nonumber
\end{aligned}\\
\begin{aligned}
{\fE^0}(N) &:= (1+\bar{c}) \sum_{t=1}^{T-1} (T+1-t)\left( \sum_{r=t}^{T-1} \bar{c}^{r-t} \cC_{r,T}(\|V\|_\infty) \zeta(\fe^0_t(N)) + \bar{c}^{T-t}\iota(\fe^0_T(N)) \right)\nonumber\\
&\quad+ \bar{c} \iota(\fe^0_T(N)) + \sum_{t=1}^{T-1} \cC_{T-t}(\|V\|_\infty)\big(\zeta(\fe^0_t(N)) + \fe^0_{t}(N)\big),\nonumber
\end{aligned}\\
{\fE^\diamond}(N) := (T+1)\left(\sum_{r=1}^{T-1} \bar{c}^{r} \cC_{T-r}(\|V\|_\infty) \zeta(\fe^0_r(N)) + \bar{c}^{T-t}\iota(\fe^0_T(N))\right).\nonumber
\end{gather*}
Above, we note that $\fE^0=\fE$ if $\vartheta\equiv 0$.

\begin{proof}[Proof of \cref{lem:ErrorConv}]
We will only prove $\lim_{N\to\infty}\fe_{t}(N)=0$ for $t=1,\dots,T$, as the rest of the proof follows automatically. In view of \cref{lem:EmpMeasConc}, the claim is obviously true for $t=1$. We will proceed by induction. Suppose for some $t\ge 1$ we have $\lim_{N\to\infty}\fe_{t}(N)=0$. For sufficiently large $N$, we let $$L(N):=\left(\max_{r=1,\dots,t}\set{\vartheta(\fe_{r})}\right)^{-\frac12}.$$ 
Then, by \cref{lem:EmpMeasConc} and the induction hypothesis, we have
\begin{align*}
\fe_{t+1}(N) \le \frac{2t}{L(N)}  + 2\sum_{r=1}^t \eta(\fe_{r}) + t\left(\frac1N+2\eta\left(\frac{2}{L(N)}\right)\right) + \fr_{\fK}(N) \xrightarrow[]{N\to\infty}0.
\end{align*}
\end{proof}

\subsection{Proof of \cref{lem:EmpMeasConc}}\label{subsec:Prooflem:EmpMeasConc}
\begin{proof}[Proof of \cref{lem:EmpMeasConc}]
To start with, we point out a useful observation
\begin{align}\label{eq:EstEMeandelta}
\bE\left[\overline\delta_{\bm Y}(A_i^c)\right]=\frac1N\sum_{n=1}^N\bE\left[\1_{A_i^c}(Y^n)\right]= \frac1N\sum_{n=1}^N\upsilon^n(A_i^c)\le i^{-1}.
\end{align}
Next, by the uniform tightness of $(\upsilon^n)_{n\in\bN}$, we have
\begin{align}\label{eq:deltaintensityBLUpperBound}
\left\|\overline\delta_{\bm Y}-\overline\upsilon\right\|_{BL} &\le \sup_{h\in C_{BL}(\bY)} \frac12\left|\int_{A_i}h(y)\overline\delta_{\bm Y}(\dif y) - \int_{A_i}h(y)\overline\upsilon(\dif y) \right| + \frac12\left(\overline\delta_{\bm Y}(A_i^c) + i^{-1}\right).
\end{align}
Regarding the first term in the right hand side of \eqref{eq:deltaintensityBLUpperBound}, we observe
\begin{align}\label{eq:supCYsupCK}
\sup_{h\in C_{BL}(\bY)}\left|\int_{A_i}h(y) \overline\delta_{\bm Y}(\dif y) - \int_{A_i}h(y)\overline\upsilon(\dif y) \right| 
&\quad\le \sup_{h\in C_{BL}(A_i)}\left|\int_{A_i}h(y)\overline\delta_{\bm Y}(\dif y) - \int_{A_i}h(y)\overline\upsilon(\dif y) \right|\nonumber\\
&\quad\le \max_{h\in\fc_{j}(A_i)}\left|\int_{A_i}h(y)\overline\delta_{\bm Y}(\dif y) - \int_{A_i}h(y)\overline\upsilon(\dif y) \right| + j^{-1},
\end{align}
where $\fc_{j}(A_i)$ is a finite set with cardinality of $\fN_{j}(A_i)$ such that $C_{BL}(A_i)\subseteq\bigcup_{h\in\fc_{j}(A_i)}B_{j^{-1}}(h)$. By \eqref{eq:deltaintensityBLUpperBound}, \eqref{eq:supCYsupCK} and \eqref{eq:EstEMeandelta}, we have 
\begin{align}\label{eq:EstExpectedBL}
&\bE\left[\left\|\overline\delta(\bm Y)-\overline\upsilon\right\|_{BL}\right]\nonumber\\
&\quad\le \bE\left[\sup_{h\in C_{BL}(A_i)}\frac12\left|\int_{A_i}h(y)\delta_{\bm Y}(\dif y) - \int_{A_i}h(y)\overline\upsilon(\dif y) \right|\right] + \frac12\left(\bE\left( \delta_{\bm Y}(A_i^c) \right) + i^{-1}\right)\nonumber\\
&\quad\le \bE\left(\max_{h\in\fc_{j}(A_i)}\frac12\left|\int_{A_i}h(y)\delta_{\bm Y}(\dif y) - \int_{A_i}h(y)\overline\upsilon(\dif y)\right|\right) + \frac12j^{-1} + i^{-1},
\end{align}
In order to proceed, for $\varepsilon>0$ and $h\in\fc_{j}(A_i)$, we invoke Heoffding's inequality to yield
\begin{align*}
&\bP\left[\frac12\left|\int_{A_i}h(y)\delta_{\bm Y}(\dif y) - \int_{A_i}h(y)\overline\upsilon(\dif y) \right| \ge \varepsilon\right] \\
&\quad= \bP\left[\frac12\left|\sum_{n=1}^N\1_{A_i}(Y_n)h(Y_n) - \sum_{n=1}^N\bE\left[\1_{A_i}(Y_n)h(Y_n)\right] \right| \ge N\varepsilon \right] \le 2\exp\left(-2N\varepsilon^2\right).
\end{align*}
Thus,
\begin{align*}
\bP\left[\max_{h\in\fc_{j}(A_i)}\frac12\left|\int_{K_i}h(y)\delta_{\bm Y}(\dif y) - \int_{A_i}h(y)\overline\upsilon(\dif y)\right| \ge \varepsilon\right] \le 2 \fN_{j}(A_i) \exp\left(-2N\varepsilon^2\right). 
\end{align*}
It follows from the fact that $\bE[Z]=\int_{\bR_+}(1-F_{Z}(r))\dif r$ for any non-negative random variable $Z$ that
\begin{align*}
\bE\left[\max_{h\in\fc_{j}(A_i)}\frac12\left|\int_{K_i}h(y)\delta_{\bm Y}(\dif y) -  \int_{A_i}h(y)\overline\upsilon(\dif y)\right|\right] \le 2\fN_{j}(A_i) \int_{\bR_+} \exp(-2Nr^2) \dif r = \frac{\sqrt{\pi}\;\fN_{j}(A_i)}{\sqrt{2 N}}.
\end{align*}
This together with \eqref{eq:EstExpectedBL} proves \eqref{eq:EmpMeasConc}.

Finally, in order to show $\lim_{N\to\infty}\fr_\fA(N)=0$, it is sufficient to construct non-decreasing $i(N),j(N)$ such that
\begin{align*}
\hat\fr_\fA(N) := \frac12j(N)^{-1} + i(N)^{-1} + \frac{\sqrt{\pi}\;\fN_{j(N)}(A_{i(N)})}{\sqrt{2 N}} \xrightarrow[N\to\infty]{} 0.
\end{align*}
We set $i(1)=j(1)=1$ and construct $i(N),j(N)$ in a way such that they increase simultaneously with increments of size $1$. We define $N_0:=1$ and $N_\ell:=\min\set{N>N_{\ell-1}:i(N)\neq i(N-1)}$ for $\ell\in\bN$, i.e., $N_\ell$ is the $N$ where $i(N),j(N)$ increase for the $\ell$-th time. Notice that $\frac{\fN_{j}(A_{i})}{2\sqrt{2\pi N}}$ can be arbitrarily small as long as $i,j$ are fixed and $N$ is large enough, we construct $i(N),j(N)$ such that
\begin{align*}
N_{\ell} = \min\left\{N>N_{\ell-1}: \frac{\sqrt{\pi}\;\fN_{j(N)}(A_{i(N)})}{\sqrt{2N}} \le \frac1\ell \right\}.
\end{align*}
Therefore, $\hat\fr_\fA(N)$ is decreasing in $N=N_{\ell},\dots,N_{\ell+1}-1$ and $\hat\fr_\fA(N_\ell)\le \frac32\ell^{-1} + \ell^{-1} = \frac52\ell^{-1}$. The proof is complete.
\end{proof}

\newpage

\begin{minipage}{0.9\linewidth}
\section{Glossary of notations}\label{sec:GoN}
\vspace{.1in}
\begin{center}
\resizebox{0.9\columnwidth}{!}{
\begin{tabular}[ht]{c c} 
\hline
\bf{Notations} & \makecell{\bf{Definitions}}\\
\hline
$\beta,\eta,\zeta,\vartheta,\iota$ & \makecell{ Subadditive modulus of continuity. \\  See \cref{subsec:Notation}, \cref{assump:PCont}, \cref{assump:GCont} and \cref{assump:SymCont}. }\\
$\bX, \cP(\bX)$ & \makecell{ State space, and the set of Borel-probabilities on $\bX$.  See \cref{subsec:Notation}.}\\
$\bA,\cP(\bA)$ & \makecell{ Action domain, and the set of Borel-probabilities on $\bA$.  See Section \cref{subsec:Notation}.}\\
$\cP(\bX\times\bA)$ & \makecell{ The set of Borel-probabilities on $\bX\times\bA$.  See \cref{subsec:Notation}.}\\
$\xi^\psi,\xi^{\mu}$ & \makecell{ Marginal distributions of $\xi$ and $\mu$ on $\bX$, respectively. \\  See \cref{subsec:Notation} and \cref{subsec:ProofMFEandOthers}.}\\
$\delta_y$ & \makecell{ Dirac measure at $y$.  See \cref{subsec:Notation}.}\\
$\overline{\delta}_{{\bm y}}$ & \makecell{Defined as $\frac{1}{N}\sum_{n=1}^n\delta_{y^n}$.  See \cref{subsec:Notation}.}\\
$\|\cdot\|_{L-BL}$, $\|\cdot\|_{BL}$ & \makecell{ Bounded-Lipschitz norm parameterized by $L$, and the abbreviation of $\|\cdot\|_{1-BL}$.\\ See \cref{subsec:Notation}. }\\ 
$\mathring\xi,P_t$ & \makecell{ Initial distribution and transition kernel. See \cref{subsec:Dynamics}. }\\ 
$\Pi, \wt\Pi, \overline\Pi$ & \makecell{ The sets of Markovian , symmetric, and oblivious action kernels. See \cref{subsec:Dynamics}. }\\ 
$N$ & \makecell{ The number of players. See \cref{subsec:Dynamics}. }\\ 
$\fp^n,{\bm\fP}$ & \makecell{ The policy of player-$n$, and the policies of all $N$ players. See \cref{subsec:Dynamics}. }\\ 
$Q^\lambda_{t,x,\xi}, \overline\bQ^{\tilde\fp}_{t,x,\xi}$ & \makecell{ Auxiliary probabilities. See \eqref{eq:DefQ} and below \eqref{eq:DeffT}. }\\ 
$\bm T^{\bm\fP}_t, \bm\cT^{\bm\fP}_{s,t}, \mathring{\bm\bT}, \bm\bT^{\bm\fP}_t$ & \makecell{ Transition operators for $N$pGs. See \eqref{eq:DefT}, \eqref{eq:DefcT}, and \eqref{eq:DeffT}. }\\ 
$\overline T_{t,\xi}, \overline\cT^{\tilde\fp}_{t,\Xi}, \mathring{\overline\bT}, \overline\bT^{\tilde\fp}_{t,\Xi}$ & \makecell{ Transition operators for mean field games. See \eqref{eq:DefMeanT}, \eqref{eq:DefMeancT}, and \eqref{eq:DefMeanfT}. }\\ 
${\bm U}, V, V_{\xi}, V_{\Psi}$ & \makecell{ Terminal cost functions. See above \cref{eq:NplayerObj}, above \eqref{eq:MFObj}, and \eqref{eq:DefVM}. }\\ 
$\mathring{\bm G}, \bm G^{\bm\lambda}_t, \mathring{\overline G}, \overline G^\lambda_{t,\xi}$ & \makecell{ Auxiliary operators for score operators. See \cref{subsec:ScoreOp}. }\\ 
$\bm S^{\bm\pi}_t, \bm\cS^{\bm\fP}_{s,t}, \bm\bS^{\bm\fP}_t$ & \makecell{ Score operators for $N$pGs. See \eqref{eq:DefS}, \eqref{eq:DefcS}, and \eqref{eq:DeffS}. }\\ 
$\bm S^{*\bm\pi}_t, \bm\cS^{*\bm\fP}_{s,t}, \bm\bS^{*\bm\fP}_t$ & \makecell{ Optimization operators for $N$pGs. See \eqref{eq:DefSstar}, \eqref{eq:DefcSstar}, and \eqref{eq:DeffSstar}. }\\ 
$\overline S^{\pi}_{t,\xi}, \overline\cS^{\tilde\fp}_{s,t,\Xi}, \overline\bS^{\tilde\fp}_{t,\Xi}$ & \makecell{ Score operators for mean field games. See \eqref{eq:DefMeanS}, \eqref{eq:DefMeancS} and \eqref{eq:DefMeanfS}. }\\ 
$\overline S^{*}_{t,\xi}, \overline\cS^{*}_{s,t,\Xi}, \overline\bS^*_{t,\Xi}$ & \makecell{ Optimization operators for mean field games. See \eqref{eq:DefMeanSstar}, \eqref{eq:DefMeancSstar} and \eqref{eq:DefMeanfSstar}. }\\ 
$\xi^\psi,\overline\pi^{\psi},\Xi^\Psi,\overline\fp^{\Psi}$ & \makecell{ Induced term related to state marginals and action kernels.\\ See \cref{subsec:MFF}. }\\ 
$\wt{\bm\fP},\overline\Xi,{\overline\Psi}$ & \makecell{ The $N$-player scenario of interest, an auxiliary vector of state distributions,\\ and the mean field scenario corresponding to $\wt{\bm\fP}$.\\ See \cref{subsec:Scenarios}. }\\ 
$c_0, c_1, \bar{c}, \cC$ & \makecell{ Constants and an auxiliary function for scaling.\\ See \cref{assump:GBasic}, \cref{assump:GCont} and \cref{lem:EstcS}. }\\ 
$\|\cdot\|_{TV}$ & \makecell{ Total-variation norm. See above \cref{assump:PCont}. }\\ 
$\fr_\fA$ & \makecell{ Auxiliary rate function. See \cref{lem:EmpMeasConc}}\\
$\fK$ & \makecell{ A sequence of increasing compact subsets. See \cref{assump:QTight}. }\\ 
$\bm\cR_\vartheta,\bm\cR,\bm\fR,\overline\cR,\overline\fR$ & \makecell{ Various notions of exploitability. See \cref{sec:Exploitability}. }\\ 
${\bm e}_t,\breve{\bm e}_t$ & \makecell{ Functions for evaluating empirical concentrations. See \cref{eq:DefEmpErr} and \cref{eq:DefStateActionEmpErr}. }\\ 
\makecell{$\fe_t,\breve\fe_t,\fe^0_t$\\$\underline\fE,\fE,\fE^0,\fE^{\diamond}$} & \makecell{ Constants related to approximation errors. See \cref{subsec:ErrorTerms} and \cref{subsec:DError}. }\\ 
$\check{\bm T}^{\tilde{\bm\pi}}_{t,\xi}, \check{\bm\cT}^{\wt{\bm\fP}}_{s,t,\overline\Xi}$ & \makecell{ Auxiliary transition operators. See \eqref{eq:DefMeancTN}. }\\
$\cP(\bX\times\cP(\bA))$ & \makecell{ The set of Borel-probabilities on $\bX\times\cP(\bA)$. See \cref{subsec:ProofMFEandOthers}. }\\ 
$\bM_0,{\bM}$ & \makecell{ Subsets of $\cP(\bX\times\cP(\bA))$. See below \cref{eq:PsiUpsilon}, above \eqref{eq:IndGammaMarginal}. }\\ 
$\xi^\mu,{\psi^\mu},\Xi^\mu,\Psi^\mu$ & \makecell{ Auxiliary induced terms related to state marginals and action kernels.\\
See \cref{subsec:ProofMFEandOthers}. }\\
$\Gamma$ & \makecell{ Auxiliary set-valued map for Kakutani fixed point theorem.\\ See \eqref{eq:IndGammaMarginal} and \eqref{eq:IndGammaOpt}. }\\
\hline
\end{tabular}
}
\end{center}
\end{minipage}

\newpage

\bibliographystyle{alpha}
\bibliography{refs}

\end{document}